\IfFileExists{nag.sty}{\RequirePackage[l2tabu, orthodox, abort, 
						]{nag}}{}

\documentclass[
		openany,	
		final,
		a4paper,	
		]
		{amsart}	

\usepackage[T1]{fontenc}	
\usepackage{lmodern}	
\usepackage{cmap}	
\usepackage{verbatim}

\usepackage{amsmath, amsfonts, amsthm}	
\usepackage{amssymb, latexsym}  
\usepackage{etex, fixltx2e}
\usepackage{stmaryrd}	
\usepackage{amscd}	
\usepackage[obeyFinal]{todonotes}

\usepackage{centernot}
\usepackage{braket}	
{\catcode`\|=\active
  \xdef\Sequence{\protect\expandafter\noexpand\csname Sequence \endcsname}
  \expandafter\gdef\csname Sequence \endcsname#1{\left\langle%
     \:{
     \mathcode`\|32768\let|\SetVert
     #1}\:\right\rangle}
}
\def\SetVert{\egroup\;\middle|\;\bgroup}

\usepackage[
	final,		
	pdfusetitle,	
	]
	{hyperref}	

\usepackage[notcite,notref]{showkeys}	

\hypersetup{
	unicode,	
	pdfauthor = {Ari Meir Brodsky, Adi Jarden}	
}

\usepackage{refcount}   
\usepackage{url}  

\makeatletter
\@ifpackagelater{showkeys}{2011/11/24}{
	\typeout{Using showkeys v3.16 (or later); this is fine.}
}{
	\@ifpackagelater{showkeys}{2006/01/10}{
		\@ifpackagewith{showkeys}{notcite}{
			\typeout{Using showkeys v3.14 or 3.15 with notcite option; this is fine.}
		}{
			\PackageError{showkeys}{
				Bad version of showkeys, 3.14 or 3.15%
			}{%
				Bug introduced in showkeys v3.14, \MessageBreak
				reported in tools/4173 in LaTeX bugs database: \MessageBreak
				http://www.latex-project.org/cgi-bin/ltxbugs2html?pr=tools/4173 \MessageBreak
				You must either upgrade the showkeys package to [2011/11/24 v3.16] \MessageBreak
				or use it with the [notcite] option.
			}
		}
	}{
		\typeout{Using showkeys v3.13 or earlier (such as v3.12 on Coxeter); this is fine.}
	}
}
\makeatother

\renewcommand\restriction{\mathbin\upharpoonright}	

\newcommand\embeds{\hookrightarrow}

\DeclareMathOperator\tp{tp}
\DeclareMathOperator\LST{LST}
\DeclareMathOperator\cl{cl}
\DeclareMathOperator\I{I}	

\DeclareMathOperator\dom{dom}	
\DeclareMathOperator\rank{rank}

\newcommand\id{\textup{id}}
\newcommand\sat{\textup{sat}}
\newcommand\na{\textup{na}}
\newcommand\bs{\textup{bs}}
\newcommand\uq{\textup{uq}}

\newcommand\univ{\textup{univ}}
\newcommand\FO{\mathbf{FO}}
\newcommand\unif{\textup{unif}}
\newcommand\ns{\textup{ns}}

\newcommand\AECK{\mathbf K}
\newcommand\leqK{\preceq_\AECK}
\newcommand\lessK{\prec_\AECK}
\newcommand\lessKuniv{\prec_\AECK^\univ}
\newcommand\leqbs{\preceq_\bs}
\newcommand\esm{\prec_\FO}

\newbox\noforkbox \newdimen\forklinewidth
\setbox0\hbox{$\textstyle\smile$}
\setbox1\hbox to \wd0{\hfil\vrule width \forklinewidth depth-2pt
 height 10pt \hfil}
\setbox\noforkbox\hbox{\lower 2pt\box1\lower 2pt\box0\relax}
\def\unionstick{\mathop{\copy\noforkbox}\limits}

\def\nonfork_#1{\unionstick_{\textstyle #1}}

\setbox0\hbox{$\textstyle\smile$}
\setbox1\hbox to \wd0{\hfil{\sl /\/}\hfil}
\setbox2\hbox to \wd0{\hfil\vrule height 10pt depth -2pt width
               \forklinewidth\hfil}
\newbox\doesforkbox
\setbox\doesforkbox\hbox{\lower 2pt\box1 \lower 2pt\box2\lower2pt\box0\relax}
\def\nunionstick{\mathop{\copy\doesforkbox}\limits}

\def\fork_#1{\nunionstick_{\textstyle #1}}

\newcommand\dnf\unionstick

\newcommand\dnsplit{\dnf\nolimits_{\ns}}

\newtheorem{theorem}{Theorem}[section]
\newtheorem{lemma}[theorem]{Lemma}
\newtheorem{corollary}[theorem]{Corollary}
\newtheorem{fact}[theorem]{Fact}
\newtheorem{facts}[theorem]{Facts}

\newtheorem{claim}{Claim}[theorem]

\theoremstyle{definition}
\newtheorem{definition}[theorem]{Definition}
\newtheorem{example}[theorem]{Example}
\newtheorem{examples}[theorem]{Examples}
\newtheorem{notation}[theorem]{Notation}

\theoremstyle{remark}
\newtheorem{remark}[theorem]{Remark}
\newtheorem{remarks}[theorem]{Remarks}

\newcounter{condition}	

\title[Density of uniqueness triples from the diamond axiom]%
	{Density of uniqueness triples \\ from the diamond axiom}
\author{Ari Meir Brodsky}
\author{Adi Jarden}
\address{Department of Mathematics\\
	Ariel University\\
	Ariel 4070000 \\
	Israel}
\email{\href{mailto:aribr@ariel.ac.il}{aribr@ariel.ac.il}}

\email{\href{mailto:jardena@ariel.ac.il}{jardena@ariel.ac.il}}

\thanks{This research was carried out with the assistance of the Center for Absorption in Science, Ministry of Aliyah and Integration, State of Israel.}

\subjclass[2010]{Primary
	03C48
	; Secondary
	03C55,		
	03E35,		
	03E65,		
	03E75
}

\keywords{Abstract elementary class, non-forking relation, pre-$\lambda$-frame, density of uniqueness triples, 
diamond axiom, 
trivial $\lambda$-frame, $*$-domination triple, isomorphic amalgamations, equivalence of amalgamations,
universal model, saturated model, homogeneous model,
representation of a model, non-splitting relation.
}

\begin{document}
\begin{abstract}

We work with a \emph{pre-$\lambda$-frame}, 
which is an abstract elementary class (AEC) endowed with a collection of basic types and a non-forking relation 
satisfying certain natural properties with respect to models of cardinality $\lambda$.

We investigate the density of \emph{uniqueness triples} in a given pre-$\lambda$-frame $\mathfrak s$, 
that is, under what circumstances every basic triple admits a non-forking extension that is a uniqueness triple.
Prior results 
in this direction
required 
strong hypotheses on $\mathfrak s$. 

Our main result is an improvement, in that 
we assume far fewer hypotheses on $\mathfrak s$.
In particular, we do not require $\mathfrak s$ to satisfy the extension, uniqueness, stability, or symmetry properties, 
or any form of local character, 
though we do impose the amalgamation and stability properties in $\lambda^+$,
and we do assume $\diamondsuit(\lambda^+)$.


As a corollary, by applying our main result to the \emph{trivial $\lambda$-frame},
it follows that in any AEC $\AECK$ satisfying modest hypotheses on $\AECK_\lambda$ and $\AECK_{\lambda^+}$, 
the set of \emph{$*$-domination triples} in $\AECK_\lambda$ is dense among the non-algebraic triples.
We also apply our main result to the non-splitting relation, obtaining the density of uniqueness triples from very few hypotheses.


\end{abstract}

\maketitle


\section{Introduction}

In Shelah's two-volume series \emph{Classification Theory for Abstract Elementary Classes} \cite{Sh:h, Sh:i},
he invented a version of domination, called uniqueness triples.
By results of Shelah and others,
the existence 
of uniqueness triples in a good $\lambda$-frame allows us to obtain:
\begin{itemize}
\item 
the existence of a pre-$({\leq}\lambda,\lambda)$-frame
\cite{Sh:600} and~\cite[\S5]{JrSh:875};

\item 
the existence of a good $\lambda^+$-frame,
which in turn implies 
the existence of models of cardinality $\lambda^{+++}$
\cite[\S8]{Sh:600} and~\cite{JrSh:875};

\item the existence of primeness triples \cite[\S4]{Sh:705};

\item a non-forking relation on sets \cite[\S5]{Sh:705} and~\cite{JrSitton};

\item orthogonality on $\lambda$ \cite[\S6]{Sh:705};

\item the amalgamation property in $\AECK_{\lambda^+}$,
in the presence of tameness \cite{Jr-Tameness}.
\todo{Check \dots With a restriction of the original $\leqK$ relation?}

\end{itemize}



The last chapter of~\cite{Sh:i} presents a complicated proof of the following Theorem, 
which provides the density of uniqueness triples in an almost-good $\lambda$-frame,
assuming weak diamond at the two successive cardinals $\lambda^+$ and $\lambda^{++}$
and a constraint on the number of models of cardinality $\lambda^{++}$:

\newpage

\begin{theorem}[{\cite[Theorem~4.32(1)]{Sh:838} and \cite[Corollary~7.12]{JrSh:966}}]\label{Shelah-thm}
Suppose that:%
\footnote{Here, 
$\I(\lambda, \AECK)$ denotes the number of isomorphism types of models in $\AECK_\lambda$,
which is variously denoted 
$\dot{I}(\lambda, \AECK)$ in~\cite[Definition~0.2(2)]{Sh:838},
$I(\lambda, \AECK)$ in~\cite[Definition~1.0.15]{JrSh:875},
and
$I(\AECK, \lambda)$ in~\cite[Notation~4.17]{Baldwin-Categoricity}.
The definition of $\mu_\unif$ and some of its properties can be found in~\cite[Section~5]{JrSh:966}.}
\begin{enumerate}
\item $2^\lambda < 2^{\lambda^+} < 2^{\lambda^{++}}$;
\item $\mathfrak s$ is an almost good or a semi-good $\lambda$-frame;
\item $\I (\lambda^{++}, \AECK) < \mu_\unif(\lambda^{++}, 2^{\lambda^+}) \sim 2^{\lambda^{++}}$.
\end{enumerate}
Then every basic triple 
has a non-forking extension that is 
a uniqueness triple.
\end{theorem}

 

In~\cite[Fact~11.3 and Appendix~E]{Vasey14},
extracting and clarifying a result of Shelah implicit in \cite[$\odot_4$ in the proof of Claim~7.12]{Sh:734},
Vasey obtains the existence of uniqueness triples in a good $\lambda$-frame from weak diamond at $\lambda^+$,
assuming categoricity in $\lambda$ as well as amalgamation and stability in $\lambda^+$.
While the categoricity hypothesis can be removed if the goal is to obtain density of uniqueness triples
(rather than their existence in every basic type),%
\footnote{For the connection between density and existence of uniqueness triples,
see~\cite[Proposition~4.1.12]{JrSh:875}.}
and some of the axioms of the good $\lambda$-frame are not used in Vasey's proof,
the proof appears to make essential use of the extension and stability axioms on the pre-$\lambda$-frame.%
\footnote{\cite[Fact~6.4]{Vasey31} appeals to superstability in order to obtain the required axioms.}

The main objective of this paper is to obtain 
the density of uniqueness triples with as few constraints as possible on the pre-$\lambda$-frame $\mathfrak s$
(Theorem~\ref{main-thm}).
In particular, we do not require $\mathfrak s$ to satisfy the extension, uniqueness, stability, or symmetry properties, 
or any form of local character, 
though we do impose the amalgamation and stability properties in $\lambda^+$,
and we do assume $\diamondsuit(\lambda^+)$.


As a corollary, by applying our main result to the \emph{trivial $\lambda$-frame},
it follows that in any AEC $\AECK$ satisfying modest hypotheses on $\AECK_\lambda$ and $\AECK_{\lambda^+}$, 
the set of \emph{$*$-domination triples} in $\AECK_\lambda$ is dense among the non-algebraic triples
(Corollary~\ref{*domination-triples-dense}).
In another corollary, we obtain the density of uniqueness triples with respect to the \emph{non-splitting relation},
assuming minimal hypotheses (Corollary~\ref{non-splitting-cor}).



\subsection{Outline of this paper}
The reader who is already familiar with abstract elementary classes (AEC), non-forking frames, and uniqueness triples
may begin reading from Section~\ref{section:diamond-H-kappa}. \todo{Maybe from Section \ref{section:uniqueness}?}
In order to make the paper almost self-contained, we include some preliminaries in the early sections. 
However, these sections do in fact include some very easy new results. 
In Section~\ref{section:AEC} we present basic properties of AECs. 
In Section~\ref{section:isomorphic} we define the notion of isomorphic amalgamations. 
In Section~\ref{section:equivalence} we define the notion of equivalent amalgamations,
as a preliminary for the definition of a uniqueness triple. 
In Sections \ref{section:frames} through \ref{section:continuity}, 
we present the material about non-forking frames that is needed for the main results of the paper. 
In Section~\ref{section:universal+} we discuss universal, saturated and homogeneous models. 
In Section~\ref{section:uniqueness} we define uniqueness triples and give observations relating to this definition. 
In Section~\ref{section:diamond-H-kappa} we present an important version of the diamond principle, 
that was formulated by Assaf Rinot and the first author. 
In Section~\ref{section:density} we prove our main result. 
In Section~\ref{section:trivial+*domination} we give an application of our main result, where the existence of a non-forking frame is not assumed,
and in Section~\ref{section:splitting} we apply our main result to the non-splitting relation.  
In Sections~\ref{section:representations} and~\ref{section:from-top} we give variants of the main result.      
  



\section{Notation and conventions}

Throughout this paper, $\AECK$ denotes an abstract elementary class (AEC, see Definition~\ref{AEC-def} below),
and $\lambda$ denotes an infinite cardinal.
We use $\leqK$ for the (reflexive) \emph{strong submodel} relation on $\AECK$,
so that $A \lessK B$ is reserved to mean ``$A \leqK B$ and $A \neq B$'' (cf.~\cite[Definition~1.0.14]{JrSh:875}).
The term \emph{embedding} will always mean $\AECK$-embedding.
$\AECK_\lambda$ refers to the class of models in $\AECK$ with cardinality $\lambda$,
accompanied by the corresponding restriction of the relation $\leqK$,
and $\AECK_{>\lambda}$, $\AECK_{\geq\lambda}$ have similar meanings.
Models labelled with the letter $N$ will generally be in $\AECK_{>\lambda}$,
\todo{Ensure this is true.}
while models labelled $A, B, \dots, M$ will generally be in $\AECK_\lambda$.

We consider $S^\na$ to be a function with domain $\AECK$
that associates to every model $A \in \AECK$ the collection $S^\na(A)$ of \emph{non-algebraic types} over $A$.

\section{Abstract elementary classes}
\label{section:AEC}

The definitions and 
basic facts relating to abstract elementary classes may be found in \cite{Baldwin-Categoricity}, \cite{Sh:h}, 
 \cite{JrSh:875}, and~\cite{Grossberg}.


\begin{definition}[e.g.~{\cite[Definition~4.1]{Baldwin-Categoricity}}]\label{AEC-def}
\todo{Refer to this definition as needed.}
Suppose $\tau$ is a fixed vocabulary.
An \emph{abstract elementary class} is a class $\AECK$ of $\tau$-models, 
with an associated binary relation $\leqK$ on $\AECK$ (the \emph{strong-submodel relation}),
satisfying the following properties:
\begin{enumerate}
\item 
Respecting isomorphisms:
For all $\tau$-models $A$, $B$, and $C$, and every isomorphism $\varphi : B \cong C$:
\begin{enumerate}
\item If $B \in \AECK$ then $C \in \AECK$.
\item If $A, B \in \AECK$ satisfy $A \leqK B$,
then $\varphi[A] \leqK C$.
\end{enumerate}
\item $\leqK$ is a reflexive and transitive binary relation 
that refines the submodel relation $\subseteq$
(that is, $A \leqK B$ implies $A \subseteq B$ for all $A, B \in \AECK$).
\item 
If $\delta>0$ and $\langle A_\alpha \mid \alpha<\delta \rangle$ is a $\leqK$-increasing sequence of models in $\AECK$, then:
\begin{enumerate}
\item
$\bigcup_{\alpha<\delta} A_\alpha \in \AECK$.

\item 
$A_\beta \leqK \bigcup_{\alpha<\delta} A_\alpha$
for every $\beta<\delta$.
\end{enumerate}

\item Smoothness:
If $\delta>0$, $\langle A_\alpha \mid \alpha\leq\delta \rangle$ is a $\leqK$-increasing, continuous sequence of models in $\AECK$, 
$N \in \AECK$,
and $A_\alpha \leqK N$ for all $\alpha<\delta$, then $A_\delta \leqK N$.
\item Coherence:
For all $A, B, C \in \AECK$,
if $A \subseteq B \subseteq C$, $A \leqK C$, and $B \leqK C$, 
then $A \leqK B $.
\item L\"owenheim-Skolem-Tarski:
There exists a smallest 
infinite cardinal $\lambda$ 
(called the \emph{L\"owenheim-Skolem-Tarski number for $\AECK$}, or $\LST(\AECK)$)
such that $\lambda \geq \left|\tau\right|$ and for every model $N \in \AECK$ and every subset $Z$ of the universe of $N$,
there exists some $M \in \AECK$ such that $M \leqK N$, $|M| \leq \max\{\lambda, \left|Z\right| \}$,
and $M$ contains every element of $Z$.
\end{enumerate}
\end{definition}

\begin{definition}[{cf.~\cite[Definition~1.0.20(1)]{JrSh:875}}]\label{def-amalgamation}
Suppose $A, B, C, D \in \AECK$, and
$g_B : A \embeds B$, $g_C : A \embeds C$, $f_B : B \embeds D$, and $f_C : C \embeds D$ are embeddings
such that $f_B \circ g_B = f_C \circ g_C$.
Then we say that $(f_B, f_C, D)$ is an \emph{amalgamation of $B$ and $C$ over $(A, g_B, g_C)$}
and $D$ is an \emph{amalgam of $B$ and $C$ over $(A, g_B, g_C)$}.

When $g_B = g_C = \id_{A}$, we simply say that the amalgamation is \emph{over $A$}.
\todo{Should ``over $A$'' mean, in addition, that $f_B \restriction A = f_C \restriction A = \id_{A}$?}

When $D \in \AECK_\lambda$, we call the amalgamation a \emph{$\lambda$-amalgamation}.
\end{definition}

The following observation will be particularly useful in our applications 
(see Definitions \ref{def-predicted-amalg-variant} and~\ref{def-predicted-amalg}):

\begin{fact}\label{predicted-is-amalgamation}
Suppose 
$A, B, N^\bullet \in \AECK$ are models such that $A \lessK N^\bullet$,
and $g : A \embeds B$ is an embedding.
If $\Omega : B \embeds N^\bullet$ is an embedding 
such that $\Omega \circ g = \id_{A}$,
then
for every $M \in \AECK$ satisfying $A \lessK M \lessK N^\bullet$,
$(\Omega, \id_M, N^\bullet)$ is an amalgamation of $B$ and $M$ over $(A, g, \id_{A})$.
\end{fact}


\begin{remark}\label{LST-remark}
If $N \in \AECK$ is any model and $Z$ is any subset of the universe of $N$, where $\left|Z\right| = \lambda$,
the hypothesis $\LST(\AECK) \leq\lambda$ (such as in Definition~\ref{def-pre-frame}(1) of a pre-$\lambda$-frame, below)
implies the existence of some model $M \in \AECK_\lambda$
such that $M \leqK N$ and $M$ contains every element of $Z$.
If $Z$  happens to include the universe of some model $A \in \AECK$, where $A \leqK N$,
then $A$ is necessarily a submodel of $M$, so that the coherence property of the AEC 
(Definition~\ref{AEC-def}(5))
guarantees that in fact $A \leqK M$.
\end{remark}


Recall that for a model $A \in \AECK$, $S(A)$ denotes the collection of all types over~$A$.

\begin{definition}
Suppose $A, B \in \AECK$ are models satisfying $A \leqK B$.
\begin{enumerate}
\item For any $p \in S(A)$ and $b \in B$, we say that \emph{$b$ realizes $p$ in $B$} 
if $p = \tp(b/A; B)$.
\item For any $p \in S(A)$, we say that \emph{$B$ realizes $p$} if
there is some $b \in B$ that realizes $p$ in $B$. 
\item For any $P \subseteq S(A)$, we say that \emph{$B$ realizes $P$} if
$B$ realizes every type in $P$.%
\footnote{Clauses (2) and~(3) of this definition appear in Definition~1.24 in the \emph{arXiv} version of~\cite{JrSh:875}, but do not appear in the published version.}
\item \cite[Definition~1.0.25]{JrSh:875}
We say that \emph{$B$ is full over $A$} if
$B$ realizes $S(A)$.
\end{enumerate}
\end{definition}

\begin{fact}\label{full-gives-bounded}
For all $A, B \in \AECK$ with $A \leqK B$,
if $B$ is full over $A$, then $\left|S(A)\right| \leq \left|B\right|$.
\end{fact}

\begin{lemma}\label{embedding-realizes-type}
Suppose $A, B, C \in \AECK$ are models such that $A \leqK B$ and $A \leqK C$,
$f : B \embeds C$ is an embedding such that $f \restriction A = \id_A$,
and $P \subseteq S(A)$.
If $B$ realizes $P$, then so does $C$.
\end{lemma}

\begin{proof}
Suppose $B$ realizes $P$, and consider arbitrary $p \in P$.
Since $B$ realizes $p$, we fix $b \in B$ such that $p = \tp(b/A; B)$.
But then $p = \tp(f(b)/A; C)$, showing that $C$ realizes $p$, and completing the proof.
\end{proof}

\begin{lemma}\label{get-ambient-lambda}
\todo{Refer to this Lemma where appropriate.}
Suppose $\LST(\AECK) \leq\lambda$, $A \in \AECK_\lambda$, and $p \in S(A)$.
Then there is some model $B \in \AECK_\lambda$ that realizes $p$.
\end{lemma}

\begin{proof}
As $p \in S(A)$, we can fix some ambient model $C \in \AECK$ with $A \leqK C$ and some $b \in C$ that realizes $p$ in $C$.
As $A \in \AECK_\lambda$,
by $\LST(\AECK) \leq\lambda$ and Remark~\ref{LST-remark}
we can fix a model 
$B \in \AECK_\lambda$ containing $b$ and every point of $A$, and such that $A \leqK B \leqK C$.
It is clear that $p = \tp(b/A; C) = \tp(b/A; B)$,
so that in fact $B$ realizes $p$.
\end{proof}


\begin{lemma}\label{count-algebraic}
Suppose $A \in \AECK$.  Then
\[
\left|S(A)\right| = 
\left| A \right| + \left|S^\na(A)\right|.
\]
In particular, if $A \in \AECK_\lambda$, then for every cardinal $\mu\geq\lambda$,
\[
\left|S(A)\right| \leq \mu \iff
\left|S^\na(A)\right| \leq \mu.
\]

\end{lemma}

\begin{proof}
We have $S^\na(A) \subseteq S(A)$,
and the difference $S(A) \setminus S^\na(A)$ consists of the set of algebraic types over $A$,
which is in one-to-one correspondence with the set of points in $A$.

The ``In particular'' then follows from the fact that $\lambda$ is an infinite cardinal.
\end{proof}

\begin{definition}[cf.~{\cite[Definition~8.19(2)]{Baldwin-Categoricity}}]
For an infinite cardinal $\mu$,
the class $\AECK_\mu$ satisfies:
\begin{enumerate}
\item
\emph{stability} if
$|S^\na (M)| \leq\mu$ for every $M \in \AECK_\mu$.
\item
\emph{almost stability} if
$|S^\na (M)| \leq\mu^+$ for every $M \in \AECK_\mu$.
\end{enumerate}
\end{definition}

\begin{corollary}
Suppose $\mu$ is some infinite cardinal.
\begin{enumerate}
\item
The class $\AECK_\mu$ satisfies stability iff
$|S(M)| =\mu$ for every $M \in \AECK_\mu$.
\item
The class $\AECK_\mu$ satisfies almost stability iff
$|S(M)| \leq\mu^+$ for every $M \in \AECK_\mu$.
\end{enumerate}
\end{corollary}

\begin{proof}
This follows immediately from Lemma~\ref{count-algebraic}.
%
\end{proof}

\begin{fact}\label{bounded-card-of-types}
For every $A \in \AECK$, if $\left|A\right| \geq\LST(\AECK)$, then $\left|S(A)\right| \leq 2^{\left|A\right|}$.
\end{fact}

\begin{corollary}\label{almost-stability-from-CH}
Suppose $\mu$ is an infinite cardinal such that $\mu \geq \LST(\AECK)$ and $2^\mu = \mu^+$.
Then $\AECK_\mu$ satisfies almost stability.
\end{corollary}


\begin{definition}[cf.~{\cite[Definition~1.0.29]{JrSh:875}}]
Suppose $N \in \AECK_{\lambda^+}$.
A \emph{representation of $N$} is a $\leqK$-increasing, continuous sequence
$\langle A_\alpha \mid \alpha<\lambda^+ \rangle$ of models in $\AECK_\lambda$
whose union is $N$.%
\footnote{A representation is sometimes called a \emph{filtration}.
Note, however, that our definition (adapted from~\cite{JrSh:875}) is more restrictive than the one given in~\cite[Exercise~4.7]{Baldwin-Categoricity},
as we require both that the length of the sequence be $\lambda^+$ (though allowing repetitions)
and that every model in the sequence have cardinality $\lambda$.}
\end{definition}

\begin{fact}\label{representation-exists}
Suppose $\LST(\AECK) \leq\lambda$.  Then:
\begin{enumerate}
\item Every $N \in \AECK_{\lambda^+}$ has a representation.
\item If $N \in \AECK_{\lambda^+}$ and $A \in \AECK_\lambda$ are such that $A \lessK N$,
then there exists a representation $\langle A_\alpha \mid \alpha<\lambda^+ \rangle$ of $N$
such that $A_0 = A$.
\end{enumerate}
\end{fact}

The following purely set-theoretic fact will be useful in our applications:

\begin{fact}\label{equal-on-club}
Suppose that $\langle X_\alpha \mid \alpha\leq\lambda^+ \rangle$ and $\langle Y_\alpha \mid \alpha\leq\lambda^+ \rangle$ 
are two $\subseteq$-increasing, continuous sequences of sets such that 
$\left|X_\alpha\right|, \left|Y_\alpha\right| \leq\lambda$ for every $\alpha<\lambda^+$,
and $X_{\lambda^+} = Y_{\lambda^+}$.%
\footnote{In applications, it will generally be the case that 
$\left|X_{\lambda^+}\right| = \left|Y_{\lambda^+}\right| = \lambda^+$,
but formally we do not require this as a hypothesis;
it is possible that $\left|X_{\lambda^+}\right|, \left|Y_{\lambda^+}\right| \leq\lambda$,
in which case the two sequences 
$\langle X_\alpha \mid \alpha<\lambda^+ \rangle$ and $\langle Y_\alpha \mid \alpha<\lambda^+ \rangle$ 
are eventually constant (below $\lambda^+$),
so that $X_\alpha = Y_\alpha$ for a cofinal set of ordinals $\alpha<\lambda^+$.}
Then the set
\[
\Set{ \alpha<\lambda^+ | X_\alpha = Y_\alpha }
\]
is a club subset of $\lambda^+$.

In particular,
if
$N \in \AECK_{\lambda^+}$,
and $\langle A_\alpha \mid \alpha<\lambda^+ \rangle$ and $\langle B_\alpha \mid \alpha<\lambda^+ \rangle$ 
are two representations of $N$,
then the set
\[
\Set{ \alpha<\lambda^+ | A_\alpha = B_\alpha }
\]
is a club subset of $\lambda^+$.
\end{fact}

\begin{theorem}[cf.~{\cite[Proposition~1.0.30]{JrSh:875}}]\label{image-equal-on-club}
Suppose that
\[
\Sequence{ (A_\alpha, B_\alpha, g_\alpha) | \alpha\leq\lambda^+ }
\]
is a sequence such that
\begin{enumerate}
\item $\left< A_\alpha \mid \alpha\leq\lambda^+ \right>$ and $\left< B_\alpha \mid \alpha\leq\lambda^+ \right>$
are $\leqK$-increasing, continuous sequences of models in $\AECK$,
where $A_\alpha, B_\alpha \in \AECK_\lambda$ for every $\alpha<\lambda^+$;
and
\item $\langle g_\alpha : A_\alpha \embeds B_\alpha \mid\allowbreak \alpha\leq\lambda^+ \rangle$
is a $\subseteq$-increasing, continuous sequence of embeddings.
\end{enumerate}
Then there is a club subset $C \subseteq \lambda^+$ such that for every $\alpha\in C$,
\[
g_\alpha[A_\alpha] = B_\alpha \cap g_{\alpha+1}[A_{\alpha+1}] = B_\alpha \cap g_{\lambda^+}[A_{\lambda^+}].
\]
\end{theorem}

\begin{proof}
Write $N := g_{\lambda^+}[A_{\lambda^+}]$.
We clearly have 
$g_\alpha[A_\alpha] \subseteq B_\alpha \cap g_{\alpha+1}[A_{\alpha+1}] \subseteq B_\alpha \cap N$
for every $\alpha<\lambda^+$.
Thus we are looking for a club subset of ordinals $\alpha<\lambda^+$ such that 
$B_\alpha \cap N \subseteq g_\alpha[A_\alpha]$.
Now, the two sequences 
$\langle g_\alpha[A_\alpha] \mid \alpha\leq\lambda^+ \rangle$ and
$\langle B_\alpha \cap N \mid\allowbreak \alpha\leq\lambda^+ \rangle$ 
satisfy the hypotheses of Fact~\ref{equal-on-club},%
\footnote{It is not necessarily the case that $B_\alpha \cap N \in \AECK$ for all $\alpha<\lambda^+$,
which is why we need the more general set-theoretic formulation of Fact~\ref{equal-on-club}.}
so that the required conclusion follows from that Fact.
%
%
\end{proof}

\begin{lemma}\label{restriction}
Suppose that:
\begin{enumerate}
\item $\delta$ is some ordinal;

\item $\left< C_\alpha \mid \alpha\leq\delta \right>$ and $\left< B_\alpha \mid \alpha\leq\delta \right>$
are $\leqK$-increasing, continuous sequences of models in $\AECK$;

\item $\langle g_\alpha : C_\alpha \embeds B_\alpha \mid\allowbreak \alpha\leq\delta \rangle$
is a $\subseteq$-increasing, continuous sequence of embeddings;
and
\item 
$b$ is some element of $B_0 \setminus g_\delta[C_\delta]$.
\end{enumerate}

For every ordinal $\alpha\leq\delta$,
write $q_\alpha := \tp(b/g_\alpha[C_\alpha]; B_\alpha)$.

Then for all $\alpha\leq\beta\leq\delta$, $q_\alpha \in S^\na(g_\alpha[C_\alpha])$ and
$q_\beta \restriction g_\alpha[C_\alpha] = q_\alpha$.
\end{lemma}

\begin{proof}
First, it follows from the hypotheses that
$g_\alpha[C_\alpha] \lessK B_\alpha$ and $b \in B_\alpha \setminus g_\alpha[C_\alpha]$ for every $\alpha\leq\delta$.
Thus $q_\alpha \in S^\na(g_\alpha[C_\alpha])$ for every $\alpha\leq\delta$.

Fix $\alpha\leq\beta\leq\delta$.
Since $g_\alpha[C_\alpha] = g_\beta[C_\alpha] \leqK g_\beta[C_\beta]$ and $B_\alpha \leqK B_\beta$,
we have
\begin{align*}
q_\beta \restriction g_\alpha[C_\alpha] &= \tp(b/g_\beta[C_\beta]; B_\beta) \restriction g_\alpha[C_\alpha] \\
&= \tp(b/g_\alpha[C_\alpha]; B_\beta) \\
&= \tp(b/g_\alpha[C_\alpha]; B_\alpha) = q_\alpha.
\qedhere
\end{align*}
\end{proof}

\section{Isomorphic amalgamations}
\label{section:isomorphic}

Although the notion of equivalence of amalgamations (to be defined in Section~\ref{section:equivalence} below)
turns out to be the more useful one in our context,
we begin here by exploring the more natural notion of isomorphic amalgamations.
This will allow us to present two interesting connections between isomorphic and equivalent amalgamations
(Facts \ref{E=isom+extensions} and~\ref{isom+equiv} below)
that do not require the amalgamation property.


\begin{definition}
Suppose $A, B, C \in \AECK$, and
$g_B : A \embeds B$ and $g_C : A \embeds C$ are embeddings.
Two amalgamations $(f^D_B, f^D_C, D)$ and $(f^F_B, f^F_C, F)$ of $B$ and $C$ over $(A, g_B, g_C)$
are \emph{isomorphic amalgamations} if there is an isomorphism $\varphi : D \cong F$ such that
$f^F_B = \varphi \circ f^D_B$ and $f^F_C = \varphi \circ f^D_C$.
\end{definition}

\begin{fact}
If $(f_B, f_C, D)$ is an amalgamation of $B$ and $C$ over $(A, g_B, g_C)$
and $\varphi : D \cong F$ is any isomorphism,
then $(\varphi \circ f_B, \varphi \circ f_C, F)$ is also amalgamation of $B$ and $C$ over $(A, g_B, g_C)$
and the two amalgamations are isomorphic.
\end{fact}

\begin{fact}\label{get-amalg-with-id}
Every amalgamation $(f_B, f_C, G)$ of $B$ and $C$ over $(A, g_B, g_C)$
is isomorphic to an amalgamation of the form $(\id_{B}, f^D_C, D)$
and also to one of the form $(f^F_B, \id_{C}, F)$.
\end{fact}

Thus, 
when fixing a particular amalgamation $(f_B, f_C, D)$ of $B$ and $C$ over $(A, g_B, g_C)$,
we may always assume either that $f_B = \id_{B}$ or that $f_C = \id_{C}$.

\begin{fact}\label{AP-versions}
The following are equivalent:
\begin{enumerate}
\item For all $A, B, C \in \AECK_\lambda$ with $A \leqK B$ and $A \leqK C$,
there exists an amalgamation $(f_B, f_C, D)$ of $B$ and $C$ over $A$;
\item For all $A, B, C \in \AECK_\lambda$, and embeddings
$g_B : A \embeds B$ and $g_C : A \embeds C$,
there exists an amalgamation $(f_B, f_C, D)$ of $B$ and $C$ over $(A, g_B, g_C)$;
\item For all $A, B, C \in \AECK_\lambda$, and embeddings
$g_B : A \embeds B$ and $g_C : A \embeds C$,
there exists an amalgamation $(\id_{B}, f_C, D)$ of $B$ and $C$ over $(A, g_B, g_C)$;
\item For all $A, B, C \in \AECK_\lambda$, and embeddings
$g_B : A \embeds B$ and $g_C : A \embeds C$,
there exists an amalgamation $(f_B, \id_{C}, D)$ of $B$ and $C$ over $(A, g_B, g_C)$.
\end{enumerate}
Furthermore, if $\LST(\AECK) \leq\lambda$, then we may assume (in each case) that $D \in \AECK_\lambda$.
\end{fact}

Thus any one of the equivalent conditions in Fact~\ref{AP-versions} (with the additional requirement that $D \in \AECK_\lambda$)
may be taken as the definition of the \emph{amalgamation property}
for $\AECK_\lambda$ (cf.~\cite[Definition~1.0.20(1)]{JrSh:875}).

\begin{remark}\label{extending-embeddings}
\todo{Refer to this remark when used!}
Consider an amalgamation $(f_B, f_C, D)$ of $B$ and $C$ over $(A, g_B, g_C)$.
In the particular case where $g_C = \id_{A}$ and $f_B = \id_{B}$,
the requirement $\id_{B} \circ g_B = f_C \circ \id_{A}$ can be expressed as $g_B \subseteq f_C$ or as $f_C \restriction A = g_B$.
Similarly, in the particular case where $g_B = \id_{A}$ and $f_C = \id_{C}$,
the requirement $f_B \circ \id_{A} = \id_{C} \circ g_C$ can be expressed as $g_C \subseteq f_B$ or as $f_B \restriction A = g_C$.
\end{remark}

\section{Equivalence of amalgamations, extended}
\label{section:equivalence}

\begin{definition}[{cf.~\cite[Definition~4.1.2]{JrSh:875}}]\label{def-E}
\todo{Improve / streamline notation and presentation in this section.}
Suppose:
\begin{enumerate}
\item $A, B, C, D, F \in \AECK$, and
$g_B : A \embeds B$ and $g_C : A \embeds C$ are embeddings;
\item $(f^D_B, f^D_C, D)$ and $(f^F_B, f^F_C, F)$ are two amalgamations of $B$ and $C$ over $(A, g_B, g_C)$.
\end{enumerate}
Then we write $(f^D_B, f^D_C, D) \mathrel{E} (f^F_B, f^F_C, F)$
\todo{Is $E$ alone enough, without a subscript?
Maybe $E^*$, to indicate that it may not be an equiv relation?
\cite{JrSh:875} uses $E_M$, but here we would need $E_{(A, g_B, g_C)}$.}
if there are $G \in \AECK$ and embeddings $f^G_D : D \embeds G$ and $f^G_F : F \embeds G$ such that
$f^G_D \circ f^D_B = f^G_F \circ f^F_B$ and $f^G_D \circ f^D_C = f^G_F \circ f^F_C$.
In such a case, we say that the triple $(f^G_D, f^G_F, G)$ 
\emph{witnesses that $(f^D_B, f^D_C, D) \mathrel{E} (f^F_B, f^F_C, F)$}.
The negation of $E$ is denoted $\centernot{E}$.
\end{definition}

\begin{fact}\label{E-reflexive-symmetric}
The relation $E$ is reflexive and symmetric.
\end{fact}


\begin{fact}\label{E-gives-amalgamation}
Suppose:
\begin{enumerate}
\item $A, B, C, D, F, G \in \AECK$, and
$g_B : A \embeds B$ and $g_C : A \embeds C$ are embeddings;
\item $(f^D_B, f^D_C, D)$ and $(f^F_B, f^F_C, F)$ are two amalgamations of $B$ and $C$ over $(A, g_B, g_C)$;
and
\item $(f^G_D, f^G_F, G)$ witnesses that $(f^D_B, f^D_C, D) \mathrel{E} (f^F_B, f^F_C, F)$.
\end{enumerate}
Then in particular, all of the following are true:
\begin{enumerate}
\item $(f^G_D, f^G_F, G)$ is an amalgamation of $D$ and $F$ over $(B, f^D_B, f^F_B)$.
\item $(f^G_D, f^G_F, G)$ is an amalgamation of $D$ and $F$ over $(C, f^D_C, f^F_C)$.
\item $(f^G_D, f^G_F, G)$ is an amalgamation of $D$ and $F$ over $(A, f^D_B \circ g_B, f^F_B \circ g_B) = (A, f^D_C \circ g_C, f^F_C \circ g_C)$.
\todo{More to add here?}
\end{enumerate}
\end{fact}

\begin{fact}
Suppose:
\begin{enumerate}
\item $A, B, C, D, D^* \in \AECK$ are such that $D \leqK D^*$;
\item $g_B : A \embeds B$ and $g_C : A \embeds C$ are embeddings; and
\item $(f^D_B, f^D_C, D)$ is an amalgamation of $B$ and $C$ over $(A, g_B, g_C)$.
\end{enumerate}
Then $(\id_D \circ f^D_B, \id_D \circ f^D_C, D^*)$ is also an amalgamation of $B$ and $C$ over $(A, g_B, g_C)$,
and $(\id_D, \id_{D^*}, D^*)$ witnesses that $(f^D_B, f^D_C, D) \mathrel{E} (\id_D \circ f^D_B, \id_D \circ f^D_C, D^*)$.
\end{fact}

\begin{fact}[cf.~{\cite[Remark~E.7(2)(a)]{Vasey14}}]
\todo{See also E.7(2)(b) there.}
Suppose:
\begin{enumerate}
\item $A, B, C, D, D^*, F, F^*, G \in \AECK$ are such that $D \leqK D^*$ and $F \leqK F^*$;
\item $g_B : A \embeds B$ and $g_C : A \embeds C$ are embeddings;
\item $(f^{D^*}_B, f^{D^*}_C, D^*)$ and $(f^{F^*}_B, f^{F^*}_C, F^*)$ are two amalgamations of $B$ and $C$ over $(A, g_B, g_C)$;
and
\item $f^{D^*}_B[B] \cup f^{D^*}_C[C]$ is a subset of $D$, and 
$f^{F^*}_B[B] \cup f^{F^*}_C[C]$ is a subset of~$F$.
\end{enumerate}
Then:
\begin{enumerate}
\item $(f^{D^*}_B, f^{D^*}_C, D)$ and $(f^{F^*}_B, f^{F^*}_C, F)$ are also amalgamations of $B$ and $C$ over $(A, g_B, g_C)$;
and
\item If $(f^G_{D^*}, f^G_{F^*}, G)$ 
witnesses that $(f^{D^*}_B, f^{D^*}_C, D^*) \mathrel{E} (f^{F^*}_B, f^{F^*}_C, F^*)$,
then $(f^G_{D^*} \circ \id_D,\allowbreak f^G_{F^*} \circ \id_F, G)$ 
witnesses that $(f^{D^*}_B, f^{D^*}_C, D) \mathrel{E} (f^{F^*}_B, f^{F^*}_C, F)$.
\end{enumerate}
\end{fact}

The following two facts connect the relation $E$ with the notion of isomorphic amalgamations.



\begin{fact}\label{E=isom+extensions}
\todo[color=green]{Just added this today (Sept.\ 6)!}
Suppose:
\begin{enumerate}
\item $A, B, C, D, F \in \AECK$;
\item $g_B : A \embeds B$ and $g_C : A \embeds C$ are embeddings; and
\item $(f^D_B, f^D_C, D)$ and $(f^F_B, f^F_C, F)$ are two amalgamations of $B$ and $C$ over $(A, g_B, g_C)$.
\end{enumerate}
Then the following are equivalent:
\begin{enumerate}
\item $(f^D_B, f^D_C, D) \mathrel{E} (f^F_B, f^F_C, F)$;
\item There exists an amalgamation $(f^{D^*}_B, f^{D^*}_C, D^*)$ of $B$ and $C$ over $(A, g_B, g_C)$
that is isomorphic to $(f^D_B, f^D_C, D)$,
as well as a model $G \in \AECK$ such that $D^* \leqK G$ and $F \leqK G$.
\end{enumerate}
\end{fact}

\begin{fact}\label{isom+equiv}
Suppose:
\begin{enumerate}
\item $A, B, C, D, F, M \in \AECK$;
\item $g_B : A \embeds B$ and $g_C : A \embeds C$ are embeddings; and
\item $(f^D_B, f^D_C, D)$, $(f^F_B, f^F_C, F)$, and $(f^M_B, f^M_C, M)$ are three amalgamations of $B$ and $C$ over $(A, g_B, g_C)$.
\end{enumerate}
If $(f^D_B, f^D_C, D)$ is isomorphic to $(f^F_B, f^F_C, F)$ and $(f^F_B, f^F_C, F) \mathrel{E} (f^M_B, f^M_C, M)$,
then $(f^D_B, f^D_C, D) \mathrel{E} (f^M_B, f^M_C, M)$.
\end{fact}

\begin{lemma}\label{connect-E-big}
Suppose:
\begin{enumerate}
\item $\LST(\AECK) \leq\lambda$ and $\AECK_\lambda$ satisfies the amalgamation property;
\item $A, B, C, D, F \in \AECK_\lambda$ and $N \in \AECK_{\geq\lambda}$;
\item $g_B : A \embeds B$ and $g_C : A \embeds C$ are embeddings; and
\item $(f^D_B, f^D_C, D)$, $(f^N_B, f^N_C, N)$, and $(f^F_B, f^F_C, F)$ are three amalgamations of $B$ and $C$ over $(A, g_B, g_C)$.
\end{enumerate}
If $(f^D_B, f^D_C, D) \mathrel{E} (f^N_B, f^N_C, N)$ and $(f^N_B, f^N_C, N) \mathrel{E} (f^F_B, f^F_C, F)$,
then $(f^D_B, f^D_C, D) \mathrel{E} (f^F_B, f^F_C, F)$.
\todo{There are other possible results where one of the models is in $\AECK_{\lambda^+}$, but maybe we don't need them.}
\end{lemma}

\begin{proof}
By $\LST(\AECK) \leq\lambda$,
fix $M \in \AECK_\lambda$ that includes $f^N_B[B] \cup f^N_C[C]$ and such that $M \leqK N$.
Clearly, $f^N_B[B] \leqK M$ and $f^N_C[C] \leqK M$.
Next, suppose $(f^{N_1}_D, f^{N_1}_N, N_1)$ witnesses that $(f^D_B, f^D_C, D) \mathrel{E} (f^N_B, f^N_C, N)$,
and $(f^{N_2}_N, f^{N_2}_F, N_2)$ witnesses that $(f^N_B, f^N_C, N) \mathrel{E} (f^F_B, f^F_C, F)$.
By $\LST(\AECK) \leq\lambda$,
fix $M_1 \in \AECK_\lambda$ that includes $f^{N_1}_D[D] \cup f^{N_1}_N[M]$ and such that $M_1 \leqK N_1$.
Clearly, $f^{N_1}_D[D] \leqK M_1$ and $f^{N_1}_N[M] \leqK M_1$.
Again by $\LST(\AECK) \leq\lambda$,
fix $M_2 \in \AECK_\lambda$ that includes $f^{N_2}_N[M] \cup f^{N_2}_F[F]$ and such that $M_2 \leqK N_2$.
Clearly, $f^{N_2}_N[M] \leqK M_2$ and $f^{N_2}_F[F] \leqK M_2$.
Finally, since $\AECK_\lambda$ satisfies the amalgamation property,
fix an amalgamation $(f^G_{M_1}, f^G_{M_2}, G)$ 
of $M_1$ and $M_2$ over $(M, f^{N_1}_N \restriction M, f^{N_2}_N \restriction M)$.
Then $(f^G_{M_1} \circ f^{N_1}_D, f^G_{M_2} \circ f^{N_2}_F, G)$ 
witnesses that $(f^D_B, f^D_C, D) \mathrel{E} (f^F_B, f^F_C, F)$,
as required.
\end{proof}

\begin{corollary}[cf.~{\cite[Proposition~4.1.3]{JrSh:875}}]
Suppose $\LST(\AECK) \leq\lambda$.
Then the following are equivalent:%
\footnote{We could omit the constraint on $\LST(\AECK)$ if we change the definition of $E$ over $\lambda$-amalgamations to require 
that the witnessing model $G$ be in $\AECK_\lambda$.  
However, 
that approach is incompatible with 
defining $E$ over arbitrary amalgamations.}
\begin{enumerate}
\item $\AECK_\lambda$ satisfies the amalgamation property;
\item The relation $E$ is transitive 
when restricted to $\lambda$-amalgamations; 
\item The relation $E$ is an equivalence relation
when restricted to $\lambda$-amalgamations. 
\end{enumerate}
\end{corollary}

\begin{proof}\hfill
\begin{description}
\item[$(1) \implies (2)$]
Take $N \in \AECK_\lambda$ in Lemma~\ref{connect-E-big}.

\item[$(2) \implies (1)$]
Consider arbitrary $A, B, C \in \AECK_\lambda$ such that $A \leqK B$ and $A \leqK C$.
Then $(\id_A, \id_A, A)$, $(\id_A, \id_A, B)$, and $(\id_A, \id_A, C)$ are three $\lambda$-amalgamations of $A$ and $A$ over $A$.
Furthermore, $(\id_B, \id_A, B)$ witnesses that $(\id_A, \id_A, B) \mathrel{E} (\id_A, \id_A, A)$, 
and similarly $(\id_A, \id_C, C)$ witnesses that $(\id_A, \id_A, A) \mathrel{E} (\id_A, \id_A, C)$.
By transitivity of $E$ on $\lambda$-amalgamations, it follows that $(\id_A, \id_A, B) \mathrel{E} (\id_A, \id_A, C)$.
Thus, in particular (Fact~\ref{E-gives-amalgamation}),
there exists an amalgamation of $B$ and $C$ over $A$,
and by $\LST(\AECK) \leq\lambda$, we may take it to be a $\lambda$-amalgamation.

\item[$(2) \iff (3)$]
By Fact~\ref{E-reflexive-symmetric}.
\qedhere
\end{description}
\end{proof}

The preceding Corollary allows the following definition:

\begin{definition}
In Definition~\ref{def-E}, if $\LST(\AECK) \leq\lambda$, $\AECK_\lambda$ satisfies the amalgamation property,
and the models $A, B, C, D, F$ are all in $\AECK_\lambda$, 
then we say that $(f^D_B, f^D_C, D)$ and $(f^F_B, f^F_C, F)$ are \emph{equivalent $\lambda$-amalgamations},
\todo{Do we need to say ``over $(A, g_B, g_C)$'' or is it already implied?}
and that the triple $(f^G_D, f^G_F, G)$ 
\emph{witnesses the equivalence of $(f^D_B, f^D_C, D)$ and $(f^F_B, f^F_C, F)$}.
\end{definition}

%

\begin{fact}
If two $\lambda$-amalgamations are isomorphic, then they are equivalent.
\end{fact}





\section{Non-forking frames}
\label{section:frames}

\begin{definition}\label{def-pre-frame}
A \emph{pre-$\lambda$-frame} is a triple $\mathfrak s = (\AECK, {\dnf}, S^\bs)$ satisfying the following properties:
\todo{SEE HISTORY AND REMARKS in \cite[\S2.4]{MR3471143} and compare the notation with other papers.}
\begin{enumerate}
\item $\AECK$ is an AEC with $\lambda \geq \LST(\AECK)$ and $\AECK_\lambda \neq\emptyset$.
\item $S^\bs$ is a function with domain $\AECK_\lambda$ such that $S^\bs(A) \subseteq S^\na(A)$ for every $A \in \AECK_\lambda$.
We call $S^\bs(A)$ the set of \emph{basic types} over $A$.
\item $\dnf$ is a relation on quadruples of the form $(A, B, c, C)$, where $A, B, C \in \AECK_\lambda$, $A \leqK B \lessK C$,
and $c \in C \setminus B$.
\item Invariance under isomorphisms:  
For all models $A, B, C, C' \in \AECK_\lambda$, every isomorphism $\varphi : C \cong C'$, and all $c \in C \setminus B$:
\begin{enumerate}
\item If $\dnf(A, B, c, C)$, then $\dnf(\varphi[A], \varphi[B], \varphi(c), C')$.
\item If $\tp(c/B; C) \in S^\bs(B)$, then $\tp(\varphi(c)/\varphi[B]; C') \in S^\bs(\varphi[B])$.
\end{enumerate}
\item Monotonicity:  For all $A, A', B, B', C, C', C'' \in \AECK_\lambda$ with
$A \leqK A' \leqK B' \leqK B \lessK C' \leqK C \leqK C''$
and all $c \in C' \setminus B$,
if $\dnf(A, B, c, C)$, then $\dnf(A', B', c, C')$ and $\dnf(A', B', c, C'')$.
\item Non-forking types are basic:
For all $A, C \in \AECK_\lambda$ and all $c \in C$,
if $\dnf(A, A, c, C)$ then $\tp(c/A; C) \in S^\bs(A)$.
\end{enumerate}
\end{definition}

\begin{fact}[{cf.~\cite[Proposition~2.1.2]{JrSh:875}}]
Suppose $\mathfrak s = (\AECK, {\dnf}, S^\bs)$ is a pre-$\lambda$-frame.
\todo{Should we assume AP here, so that types are best-behaved? Check this in the references.}
For all $A, B, C, C^* \in \AECK_\lambda$, every $c \in C$ and every $c^* \in C^*$,
if $A \leqK B \leqK C$, $B \leqK C^*$, and $\tp(c/B; C) = \tp(c^*/B; C^*)$, then
\[
\dnf(A, B, c, C) \iff \dnf(A, B, c^*, C^*).
\]
\end{fact}

By the above fact, $\dnf$ is really a relation on Galois-types, justifying the following definition:

\begin{definition}\label{def-dnf}
\todo{Consider showing how the properties in Definition~\ref{def-pre-frame} can be formulated in terms of types.}
Suppose $\mathfrak s = (\AECK, {\dnf}, S^\bs)$ is a pre-$\lambda$-frame,
$A, B \in \AECK_\lambda$ are such that $A \leqK B$, and $p \in S(B)$ is a type.
We say that \emph{$p$ does not fork over $A$} \todo{Or \emph{$\mathfrak s$-fork} if necessary}
if $\dnf(A, B, c, C)$ for some (equivalently, for every) $c$ and $C$ such that $p = \tp(c/B; C)$.
\end{definition}

\begin{definition}\label{frame-properties}
Suppose $\mathfrak s = (\AECK, {\dnf}, S^\bs)$ is a pre-$\lambda$-frame.
We consider several properties that $\mathfrak s$ may satisfy, as follows.
We say that $\mathfrak s$ satisfies:
\todo{Decide which definitions to fill in, and emphasize which properties are not essential for the sequel.
Continue to check all variations of these properties in all papers and books, e.g.\ \cite{MR3729331}.}
\begin{enumerate}
\item \emph{Amalgamation property} (AP), \emph{joint embedding property} (JEP), or \emph{no maximal model} (respectively) if 
the class $\AECK_\lambda$ satisfies the respective property.
\item \emph{Type-fullness} (or \emph{$\mathfrak s$ is type-full}) if 
$S^\bs = S^\na \restriction \AECK_\lambda$.
\item \emph{Basic stability} if $\left| S^\bs(A) \right| \leq \lambda$ for every $A \in \AECK_\lambda$.  
\item \emph{Basic almost stability} \cite[Definition~2.1.3]{JrSh:875} if 
$\left| S^\bs(A) \right| \leq \lambda^+$ for every $A \in \AECK_\lambda$.  
\item \emph{Density of basic types} (or just \emph{density}) if:
For all $A, B \in \AECK_\lambda$ with $A \lessK B$, there is $b \in B \setminus A$ such that $\tp(b/A; B) \in S^\bs(A)$.
\item \emph{Transitivity} if:  For all $A, B, C \in \AECK_\lambda$ satisfying $A \leqK B \leqK C$ and every $p \in S(C)$,
if $p$ does not fork over $B$ and $p \restriction B$ does not fork over $A$, then $p$ does not fork over $A$.
\item \emph{Existence} \cite[Definition~2.21(4)]{MR3471143} if:  For every $A \in \AECK_\lambda$ and every $p \in S^\bs(A)$,
$p$ does not fork over $A$.
Equivalently:  For every $A \in \AECK_\lambda$,
\[
S^\bs(A) = \Set{ p \in S^\na(A) | p \text{ does not fork over } A }.
\]


\item \label{extension-def} \emph{Extension} \cite[Definition~3.2(5)]{Marcos} if:%
\footnote{\label{extension-footnote}For simplicity, we choose a formulation of the extension property that implies the existence property.
Alternative formulations of the extension property appear in \cite[Definition~2.1.1(3)(f)]{JrSh:875}, 
\cite[Definition~2.21(5)]{MR3471143} and~\cite[Definition~3.8(4)]{MR3694338}.
In the presence of the transitivity and existence properties, they are all equivalent to the version given here.}  
For all $B, C \in \AECK_\lambda$ and every $p \in S^\bs(B)$,
if $B \leqK C$, then there is some $q \in S^\bs(C)$ extending%
\footnote{For types $p \in S(B)$ and $q \in S(C)$, where $B \leqK C$, we say that \emph{$q$ extends $p$} if $q \restriction B = p$.} 
$p$
that does not fork over $B$.
\item \emph{Uniqueness} if:
For all $A, B \in \AECK_\lambda$ and all $p, q \in S(B)$, 
if $A \lessK B$, $p \restriction A = q \restriction A$, and both $p$ and $q$ do not fork over $A$, then $p=q$.
\item\label{continuity-def} \emph{Continuity} if: 
\todo{From here onward, haven't checked this with other sources.  
This version is adapted from \cite[Definition 2.1.1(3)(g,c,e)]{JrSh:875}.}
For every nonzero limit ordinal $\delta<\lambda^+$,
every $\leqK$-increasing, continuous sequence $\Sequence{ A_\alpha | \alpha\leq\delta }$ of models in $\AECK_\lambda$,
and every type $p \in S(A_\delta)$,
if for every $\alpha<\delta$, $p \restriction A_\alpha$ does not fork over $A_0$,
then $p \in S^\bs(A_\delta)$ and does not fork over $A_0$.
\item \emph{Local character} if: 
\todo[color=red]{Filling these in on Sept.~7.}
For every nonzero limit ordinal $\delta<\lambda^+$,
every $\leqK$-increasing, continuous sequence $\Sequence{ A_\alpha | \alpha\leq\delta }$ of models in $\AECK_\lambda$,
and every type $p \in S^\bs(A_\delta)$,
there exists some $\alpha<\delta$ such that $p$ does not fork over~$A_\alpha$.
\item \emph{Symmetry} if: 
For all $A, C, D \in \AECK_\lambda$ such that $A \lessK C \lessK D$, every $c \in C$ and every $d \in D$,
if $\tp(c/A; D) \in S^\bs(A)$ and $\tp(d/C; D)$ does not fork over $A$,
then there exist models $B, F \in \AECK_\lambda$ such that
$d \in B$, $A \lessK B \lessK F$, $D \leqK F$, and $\tp(c/B; F)$ does not fork over $A$.
\todo{Add conjugation?}
\end{enumerate}
\end{definition}

We will not require the local character and symmetry properties at all in the applications in this paper.
\todo{SHOULD WE DEFINE GOOD FRAMES (AND OTHER VARIANTS)?}

\begin{remark}
The property $\AECK_\lambda \neq\emptyset$ does not follow from any of the other properties listed in Definition~\ref{def-pre-frame},
nor from any of the properties listed in Definition~\ref{frame-properties}.
Thus, $(\emptyset, \emptyset, \emptyset)$ would be a pre-$\lambda$-frame (for any infinite $\lambda$) and even a \emph{good $\lambda$-frame}
if it were not explicitly excluded by Clause~(1) of Definition~\ref{def-pre-frame}.
In fact, the definition of good $\lambda$-frame given in~\cite[Definition~2.1.1]{JrSh:875} does not exclude $(\emptyset, \emptyset, \emptyset)$.
\end{remark}

\begin{examples}\label{examples-pre-frames}
Consider a given AEC $\AECK$ and any $\lambda \geq \LST(\AECK)$ such that $\AECK_\lambda \neq\emptyset$.
We explore several examples of pre-$\lambda$-frames $\mathfrak s = (\AECK, {\dnf}, S^\bs)$ and examine which properties they satisfy:
\todo{Check all the properties here (especially existence/extension), and check all notes from 21 Iyyar and 7 Sivan.
Ensure the properties are listed in the same order as in Definition \ref{frame-properties}.}
\begin{enumerate}
\item Let ${\dnf} := \emptyset$ and $S^\bs(A) := \emptyset$ for every $A \in \AECK_\lambda$.
Then $(\AECK, {\dnf}, S^\bs)$ is a pre-$\lambda$-frame
satisfying basic stability, transitivity, existence, extension, uniqueness, local character, continuity, and symmetry, but (in general) not density.
(This is a special case of Examples (3), (5), and~(6) below.)

\item 
Let ${\dnf} := \emptyset$ and 
$S^\bs := S^\na \restriction \AECK_\lambda$.
Then $(\AECK, {\dnf}, S^\bs)$ is a type-full pre-$\lambda$-frame
satisfying density, transitivity, uniqueness, continuity, and symmetry, 
but (in general) not existence, extension, or local character.
(This is a special case of Examples (4) and~(5) below.)

\item 
Suppose we are given any relation $\dnf$ satisfying properties (3), (4)(a), and~(5) of Definition~\ref{def-pre-frame}.
Define the minimal $S^\bs$ compatible with property~(6) of Definition~\ref{def-pre-frame}, by setting for all $A \in \AECK_\lambda$:
\[
S^\bs(A) := \Set{ p \in S^\na(A) | p \text{ does not fork over } A }.
\]
Then $(\AECK, {\dnf}, S^\bs)$ is a pre-$\lambda$-frame
satisfying the existence property.

\item 
Suppose we are given any relation $\dnf$ satisfying properties (3), (4)(a), and~(5) of Definition~\ref{def-pre-frame}.
Define $S^\bs := S^\na \restriction \AECK_\lambda$.
Then $(\AECK, {\dnf}, S^\bs)$ is a type-full pre-$\lambda$-frame
satisfying density, 
but (in general) not existence, extension, or uniqueness.

\item 
Suppose we are given $S^\bs$ satisfying properties (2) and~(4)(b) of Definition~\ref{def-pre-frame}.
Then $(\AECK, \emptyset, S^\bs)$ is a pre-$\lambda$-frame
satisfying transitivity, uniqueness, continuity, and symmetry, 
but (in general) not existence, extension, or local character.

\item 
Suppose we are given $S^\bs$ satisfying properties (2) and~(4)(b) of Definition~\ref{def-pre-frame}.
Define the minimal non-forking relation $\dnf$ compatible with the existence property.
That is, we say that $p$ does not fork over $A$ iff $p \in S^\bs(A)$.
Equivalently,
\[
{\dnf} := \Set{ (A, A, c, C) | 
\begin{gathered}
A, C \in \AECK_\lambda, A \lessK C, c \in C \setminus A, \\
\tp(c/A;C) \in S^\bs(A)
\end{gathered}
}.
\]
Then $(\AECK, {\dnf}, S^\bs)$ is a pre-$\lambda$-frame 
satisfying transitivity, existence, uniqueness, continuity, and symmetry, 
but (in general) not extension or local character.
\todo{Find more examples to add here!}

\end{enumerate}
\end{examples}

Several important examples of pre-$\lambda$-frames ---
the \emph{trivial $\lambda$-frame}, and two pre-$\lambda$-frames derived from the \emph{non-splitting} relation ---
will be introduced and explored in Sections \ref{section:trivial+*domination} and~\ref{section:splitting} below.

\begin{lemma}\label{basic-nonempty}
Suppose $\mathfrak s = (\AECK, {\dnf}, S^\bs)$ is a pre-$\lambda$-frame
satisfying no maximal model and density.
Then $S^\bs(A) \neq\emptyset$ for every $A \in \AECK_\lambda$.
\end{lemma}

\begin{proof}
Consider arbitrary $A \in \AECK_\lambda$.
Since $\AECK_\lambda$ has no maximal model, fix $B \in \AECK_\lambda$ with  $A \lessK B$.
Then by the density property,
there is $b \in B \setminus A$ such that $\tp(b/A; B) \in S^\bs(A)$,
as sought.
\end{proof}

\begin{definition}[{\cite[Definition~3.1.1]{JrSh:875}}]
Suppose $\mathfrak s = (\AECK, {\dnf}, S^\bs)$ is a pre-$\lambda$-frame.
\begin{enumerate}
\item
Let $\AECK^{3,\bs}$ denote the class of \emph{basic triples}, that is,
\[
\AECK^{3,\bs} := \Set{ (A, B, b) | A, B \in \AECK_\lambda, A \lessK B, b \in B \setminus A, \tp(b/A; B) \in S^\bs(A) }.
\]
\item
Define a binary relation $\leqbs$ on $\AECK^{3,\bs}$ by setting
$(A, B, b) \leqbs (C, D, d)$ iff
$A \leqK C$, $B \leqK D$, $b=d$, and $\tp(b/C; D)$ does not fork over $A$.
\todo{Use the $\leqbs$ relation in stating the results!}
\end{enumerate}
\end{definition}

\begin{facts}
Suppose $\mathfrak s = (\AECK, {\dnf}, S^\bs)$ is a pre-$\lambda$-frame.
\begin{enumerate}
\item If $(A, B, b) \leqbs (C, D, b)$, then $\tp(b/C; D) \restriction A = \tp(b/A; B)$,
that is, $\tp(b/C; D)$ extends $\tp(b/A; B)$.
\item If $\tp(d/C; D)$ does not fork over $A$, then
$(A, D, d) \leqbs (C, D, d)$.
\item The relation $\leqbs$ is always antisymmetric.
\item
$\mathfrak s$ satisfies the existence property iff $\leqbs$ is reflexive.
\item
$\mathfrak s$ satisfies the transitivity property iff $\leqbs$ is transitive.
\end{enumerate}
\end{facts}

\section{Non-forking relation and basic types over larger models}
\label{section:larger}

Recall that for a given pre-$\lambda$-frame $(\AECK, {\dnf}, S^\bs)$,
the domain of $S^\bs$ is $\AECK_\lambda$,
and the non-forking relation is defined with respect to types over models in $\AECK_\lambda$.
In this section, we expand the non-forking relation and the class of basic types to include types over models of cardinality $>\lambda$.

\begin{definition}[cf.~{\cite[Definition~2.6.2]{JrSh:875}}]\label{big-dnf}
Suppose $(\AECK, {\dnf}, S^\bs)$ is a pre-$\lambda$-frame,
$A \in \AECK_\lambda$ and $N \in \AECK_{>\lambda}$ are such that $A \leqK N$,
and $p \in S(N)$.
We say that \emph{$p$ does not fork over $A$} if
for every $B \in \AECK_\lambda$,
if $A \leqK B \leqK N$ then $p \restriction B$ does not fork over $A$.
\end{definition}

\begin{definition}[{\cite[Definition~2.6.4]{JrSh:875}}]\label{big-basic}
Suppose $(\AECK, {\dnf}, S^\bs)$ is a pre-$\lambda$-frame and $N \in \AECK_{>\lambda}$.
The type $p \in S(N)$ is said to be \emph{basic} if
there is some $A \in \AECK_\lambda$ such that $A \leqK N$ and 
$p$ does not fork over $A$.
The collection 
\todo{In fact $S(N)$ is always a set of size $\leq 2^{\left|N\right|}$, assuming $\LST(\AECK) \leq \left|N\right|$ 
(Fact \ref{bounded-card-of-types}).}
of basic types over $N$ is denoted $S^\bs_{>\lambda}(N)$.
\end{definition}

Notice that $S^\bs_{>\lambda}(N) \subseteq S^\na(N)$ for all $N \in \AECK_{>\lambda}$.

\begin{remark}
Definition~\ref{big-dnf} is a natural parallel to the definition of non-forking of types over models in $\AECK_\lambda$.
That is, if we apply the defining condition in Definition~\ref{big-dnf} to a model $C$ in $\AECK_\lambda$ rather than to $N$ in $\AECK_{>\lambda}$,
we recover exactly the non-forking relation of Definition~\ref{def-dnf}.

Similarly, assuming that $(\AECK, {\dnf}, S^\bs)$ satisfies the existence property,
the function $S^\bs_{>\lambda}$ is a natural parallel to $S^\bs$.
That is, if we apply the defining condition in Definition~\ref{big-basic} to a model $C$ in $\AECK_\lambda$ rather than to $N$ in $\AECK_{>\lambda}$,
we recover exactly the collection of types in $S^\bs(C)$.
\end{remark}

\begin{lemma}\label{basic-triple-from-big}
Suppose $(\AECK, {\dnf}, S^\bs)$ is a pre-$\lambda$-frame, $N \in \AECK_{>\lambda}$,
and $p \in S^\bs_{>\lambda}(N)$.
Then there is a triple $(A, B, b) \in \AECK^{3,\bs}$ such that 
$A \lessK N$, 
$p$ does not fork over $A$,
and $\tp(b/A; B) = p \restriction A$.
\todo{We can ensure that $B$, $N$ can be amalgamated via $\id$,
so that in particular, $b \notin N$;
should this be explicit?}
\end{lemma}

\begin{proof}
As, in particular, $p \in S^\na(N)$, we fix
some ambient model $N^* \in \AECK_{>\lambda}$ with $N \lessK N^*$ and some $b \in N^* \setminus N$
that realizes $p$ in $N^*$, that is,
such that $p = \tp(b/N; N^*)$.

Since $p \in S^\bs_{>\lambda}(N)$ (see Definition~\ref{big-basic}),
we can fix $A \in \AECK_\lambda$ such that $A \lessK N$ and $p$ does not fork over $A$.
In particular (see Definition~\ref{big-dnf}), $p \restriction A$ does not fork over $A$, so that $p \restriction A \in S^\bs(A)$.

As $A \lessK N \lessK N^*$ and $b \in N^* \setminus N$, 
we apply $\LST(\AECK) \leq\lambda$ and Remark~\ref{LST-remark} to fix
$B \in \AECK_\lambda$ containing $b$ and every point of $A$, and such that $A \lessK B \lessK N^*$,
where clearly $b \in B \setminus A$.  
Then $\tp(b/A; B) = \tp(b/A; N^*) = \tp(b/N; N^*) \restriction A = p \restriction A \in S^\bs(A)$,
meaning that $(A, B, b) \in \AECK^{3,\bs}$.
\end{proof}

In the particular case where $N \in \AECK_{\lambda^+}$,
the following Lemma allows us to restrict our attention to models in a particular representation, when applying Definition~\ref{big-dnf}.

\begin{lemma}\label{big-dnf-equiv}
Suppose $(\AECK, {\dnf}, S^\bs)$ is a pre-$\lambda$-frame, $N \in \AECK_{\lambda^+}$,
$p \in S(N)$,
and $\langle A_\alpha \mid \alpha<\lambda^+ \rangle$ is a representation of $N$.
Then the following are equivalent:
\begin{enumerate}
\item $p$ does not fork over $A_0$;
\item For every $\alpha<\lambda^+$, $p \restriction A_\alpha$ does not fork over $A_0$.
\end{enumerate}
\end{lemma}

\begin{proof}\hfill
\begin{description}
\item[$(1) \implies (2)$]
Clear from the definition, since $A_0 \leqK A_\alpha \lessK N$ for every $\alpha<\lambda^+$.

\item[$(2) \implies (1)$]
Consider any $B \in \AECK_\lambda$ such that $A_0 \leqK B \leqK N$.
As $B \in [N]^\lambda$ and $\langle A_\alpha \mid \alpha<\lambda^+ \rangle$ is a representation of $N$,
we can fix some $\alpha<\lambda^+$ such that every point of $B$ is in $A_\alpha$.
As $B \leqK N$ and $A_\alpha \leqK N$,
it follows that
$B$ is necessarily a submodel of $A_\alpha$, so that the coherence property of the AEC 
(Definition~\ref{AEC-def}(5))
guarantees that in fact $B \leqK A_\alpha$.
Since $p \restriction A_\alpha$ does not fork over $A_0$,
it follows by monotonicity that
$p \restriction B$ does not fork over $A_0$,
as required.
\qedhere
\end{description}
\end{proof}

\begin{corollary}\label{big-basic-equiv}
Suppose $(\AECK, {\dnf}, S^\bs)$ is a pre-$\lambda$-frame, $N \in \AECK_{\lambda^+}$,
and $p \in S(N)$.
Then the following are equivalent:
\begin{enumerate}
\item $p \in S^\bs_{>\lambda}(N)$;
\item There is some representation $\langle A_\alpha \mid \alpha<\lambda^+ \rangle$ of $N$
such that for every $\alpha<\lambda^+$, $p \restriction A_\alpha$ does not fork over $A_0$.
\end{enumerate}
\end{corollary}

\begin{proof}\hfill
\begin{description}
\item[$(1) \implies (2)$]
By~(1), fix $A \in \AECK_\lambda$ such that $A \leqK N$ and $p$ does not fork over $A$.
By $\LST(\AECK) \leq\lambda$ and Fact~\ref{representation-exists}(2), fix a representation $\langle A_\alpha \mid \alpha<\lambda^+ \rangle$ of $N$
such that $A_0 = A$.
Then (2) follows from Lemma~\ref{big-dnf-equiv}.

\item[$(2) \implies (1)$]
Clearly $A_0 \leqK N$, and by Lemma~\ref{big-dnf-equiv} $p$ does not fork over $A_0$, 
showing that $p$ is basic.
\qedhere
\end{description}
\end{proof}

\begin{fact}[Invariance under isomorphisms]\label{invariance-big}
Suppose $(\AECK, {\dnf}, S^\bs)$ is a pre-$\lambda$-frame,
$A \in \AECK_\lambda$, $N, N' \in \AECK_{>\lambda}$, 
$p \in S(N)$, and $\varphi : N \cong N'$ is an isomorphism.
Using the notation $\varphi(p)$ from~\cite[Definition~2.5.3]{JrSh:875}: 
\begin{enumerate}
\item If $p$ does not fork over $A$, then $\varphi(p)$ does not fork over $\varphi[A]$.
\item If $p \in S^\bs_{>\lambda}(N)$, then $\varphi(p) \in S^\bs_{>\lambda}(N')$.
\end{enumerate}
\end{fact}

\section{Applying the continuity property}
\label{section:continuity}

\todo{Decide where to put this.}
The following Lemma formalizes how the continuity property of a pre-$\lambda$-frame
(Definition~\ref{frame-properties}(\ref{continuity-def}))
is typically used at the limit stages in our recursive constructions:

\begin{lemma}\label{continuous-union}
Suppose that:
\begin{enumerate}
\item $\mathfrak s = (\AECK, {\dnf}, S^\bs)$ is a pre-$\lambda$-frame satisfying the continuity property;

\item $\delta\leq\lambda^+$ is some nonzero limit ordinal;

\item $\left< C_\alpha \mid \alpha<\delta \right>$ and $\left< B_\alpha \mid \alpha<\delta \right>$
are $\leqK$-increasing, continuous sequences of models in $\AECK_\lambda$;

\item $\langle g_\alpha : C_\alpha \embeds B_\alpha \mid\allowbreak \alpha<\delta \rangle$
is a $\subseteq$-increasing, continuous sequence of embeddings;

\item $b$ is some element of $B_\alpha \setminus g_\alpha[C_\alpha]$
for all $\alpha <\delta$; 
and
\item For every $\alpha<\delta$, 
$
\tp(b/g_\alpha[C_\alpha]; B_\alpha)
$ 
does not fork over $g_0[C_0]$. 
\end{enumerate}

Define
\[
C_\delta := \bigcup_{\alpha<\delta} C_\alpha; \quad
B_\delta := \bigcup_{\alpha<\delta} B_\alpha; \quad
g_\delta := \bigcup_{\alpha<\delta} g_\alpha,
\]
and write $q_\delta := \tp(b/g_\delta[C_\delta]; B_\delta)$.

Then:
\begin{enumerate}
\item $C_\delta, B_\delta \in \AECK$, and
$\left< C_\alpha \mid \alpha\leq\delta \right>$ and $\left< B_\alpha \mid \alpha\leq\delta \right>$
are $\leqK$-increasing, continuous sequences of models;

\item If $N \in \AECK$ is some model such that $C_\alpha \leqK N$ for all $\alpha<\delta$, then $C_\delta \leqK N$;

\item $\langle g_\alpha : C_\alpha \embeds B_\alpha \mid\allowbreak \alpha\leq\delta \rangle$
is a $\subseteq$-increasing, continuous sequence of embeddings,
so that 
$\left< g_\alpha[C_\alpha] \mid \alpha\leq\delta \right>$ 
is a $\leqK$-increasing, continuous sequences of models in $\AECK$;

\item $b \in B_\delta \setminus g_\delta[C_\delta]$;
and

\item For every $\gamma<\delta$,
\footnote{In fact this is true for $\gamma=\delta$ as well,
provided that $\left|C_\delta\right|=\lambda$.
(Even in Definition \ref{big-dnf} we defined non-forking only over models $A \in \AECK_\lambda$.)}
$q_\delta$
does not fork over $g_\gamma[C_\gamma]$,
so that, in particular, either $q_\delta \in S^\bs(g_\delta[C_\delta])$ (if $\left|C_\delta \right| = \lambda$)
or $q_\delta \in S^\bs_{>\lambda}(g_\delta[C_\delta])$ (if $\left|C_\delta \right| = \lambda^+$).
\end{enumerate}
\end{lemma}

\begin{proof}
Clause~(1) is simply an application of 
Definition~\ref{AEC-def}(3).
Clause~(2) is simply the smoothness property of the AEC
(Definition~\ref{AEC-def}(4)).
Clauses (3) and~(4) are immediate.

To prove (5):
Write $q_\alpha := \tp(b/g_\alpha[C_\alpha]; B_\alpha)$ for every $\alpha\leq\delta$.
Consider $q_\delta$, which is an element of $S(g_\delta[C_\delta])$.
For every $\alpha<\delta$,
we obtain $q_\delta \restriction g_\alpha[C_\alpha] = q_\alpha$ by Lemma~\ref{restriction},
and Clause~(6) of the hypothesis gives that
$q_\alpha$ does not fork over $g_0[C_0]$.
We now consider three cases:
\begin{itemize}
\item Suppose $\delta<\lambda^+$.
Since $\langle g_\alpha[C_\alpha] \mid\allowbreak \alpha\leq\delta \rangle$ is a
$\leqK$-increasing, continuous sequence of models in $\AECK_\lambda$,
the continuity property of $\mathfrak s$ (Definition~\ref{frame-properties}(\ref{continuity-def})) gives
$q_\delta \in S^\bs(g_\delta[C_\delta])$ and does not fork over $g_0[C_0]$.

\item Suppose $\delta=\lambda^+$ but $\left|C_{\lambda^+} \right| = \lambda$,
meaning that $\left< C_\alpha \mid \alpha<\lambda^+ \right>$ is eventually constant.
Fix some nonzero limit $\beta<\lambda^+$ such that $C_\beta = C_{\lambda^+}$.
Then also $g_\beta = g_{\lambda^+}$, so that $g_{\lambda^+}[C_{\lambda^+}] = g_\beta[C_\beta] \in \AECK_\lambda$,
and 
\[
q_{\lambda^+} = \tp(b/g_{\lambda^+}[C_{\lambda^+}]; B_{\lambda^+}) = \tp(b/g_\beta[C_\beta]; B_{\lambda^+}) = \tp(b/g_\beta[C_\beta]; B_\beta) = q_\beta,
\]
which we already know does not fork over $g_0[C_0]$
(and is therefore in $S^\bs(g_\beta[C_\beta])$).

\item Suppose $\delta=\lambda^+$ and $\left|C_{\lambda^+} \right| = \lambda^+$.
Since $\langle g_\alpha[C_\alpha] \mid\allowbreak \alpha<\delta \rangle$ is a representation of 
$g_\delta[C_\delta] \in \AECK_{\lambda^+}$, 
Lemma~\ref{big-dnf-equiv} gives that 
$q_\delta$ does not fork over $g_0[C_0]$,
so that in particular $q_\delta \in S^\bs_{>\lambda}(g_\delta[C_\delta])$.
\end{itemize}
Then, in all cases, by monotonicity, 
$q_\delta$ does not fork over $g_\gamma[C_\gamma]$
for all $\gamma<\delta$.
\end{proof}

\section{Universal, saturated, and homogeneous models}
\label{section:universal+}

\begin{notation}
\todo{Organize these sections.}
Throughout this section, $\mu$ denotes an infinite cardinal.
\end{notation}




We recall the following definitions:
\todo{See also various notes from 2 Tevet.}

\begin{definition}[cf.~{\cite[Definition~10.4]{Baldwin-Categoricity}}]
Suppose 
$N, N' \in \AECK_\mu$ are such that $N \lessK N'$.
We say that \emph{$N'$ is universal over $N$}, \todo{Or $\mu$-universal} and we write $N \lessKuniv N'$, provided that
for every $N^+ \in \AECK_\mu$,
if $N \leqK N^+$ then there is an embedding $f : N^+ \embeds N'$ such that $f \restriction N = \id_N$.
\end{definition}

\begin{definition}
\todo{COMPARE HAVING A UNIVERSAL MODEL OVER A SINGLE MODEL,
as in the hypothesis of the Theorem~\ref{repres-uq-extend-triple} and its corollaries.
Is this strictly weaker than having a universal model over EVERY model?}
Suppose $N \in \AECK_{\lambda^+}$.
We say that $N$ is:
\begin{enumerate}
\item
\cite[Definition~1.0.25]{JrSh:875}
\emph{saturated in $\lambda^+$ over $\lambda$}
if for every model $M \in \AECK_\lambda$ with $M \lessK N$, $N$ is full over~$M$;
\item 
\cite[Definition~1.0.28]{JrSh:875}
\emph{homogeneous in $\lambda^+$ over $\lambda$}
if for all $A, B \in \AECK_\lambda$ with $A \lessK N$ and $A \leqK B$,
there is an embedding $f : B \embeds N$ with $f \restriction A = \id_A$.
\end{enumerate}
We denote by $\AECK^\sat_{\lambda^+}$ the class of models in $\AECK_{\lambda^+}$ that are saturated in $\lambda^+$ over $\lambda$.
\end{definition}

\begin{definition}
A model $A \in \AECK_\mu$ is called an \emph{amalgamation base in $\AECK_\mu$} if:
For all $B, C \in \AECK_\mu$ with $A \leqK B$ and $A \leqK C$,
there exists a $\mu$-amalgamation $(f_B, f_C, D)$ of $B$ and $C$ over $A$.
\todo{Note counterexample:  
Let $\AECK$ be the class of sets with not both red and green points.
Then there is AP over a model with (e.g.)\ green points, but not over a model with only black points.}
\end{definition}

The following is immediate:

\begin{fact}\label{AP=AB}
$\AECK_\mu$ satisfies the amalgamation property iff every $A \in \AECK_\mu$ is an amalgamation base in $\AECK_\mu$.
\end{fact}

\begin{lemma}\label{embedding-gives-full&AB}
Suppose that $M \in \AECK_\mu$ and $N \in \AECK_{\geq\mu}$ satisfy $M \leqK N$.
Suppose that for every $B \in \AECK_\mu$ with $M \leqK B$ there is some embedding $f : B \embeds N$ such that $f \restriction M = \id_M$.
Then:
\begin{enumerate}
\item If $\min\{\LST(\AECK), \left|N\right|\} \leq\mu$, 
then $M$ is an amalgamation base in $\AECK_\mu$; and 
\item If $\LST(\AECK) \leq\mu$, 
then $N$ is full over $M$, so that
$\left|S(M)\right| \leq \left|N\right|$.
\item If $M^- \in \AECK_\mu$ satisfies $M^- \leqK M$ and $M^-$ is an amalgamation base in $\AECK_\mu$,
then for every $C \in \AECK_\mu$ satisfying $M^- \leqK C$ there is some embedding $g : C \embeds N$ such that $g \restriction M^- = \id_{M^-}$.
\end{enumerate}
\end{lemma}

\begin{proof}\hfill
\begin{enumerate}
\item 
To see that $M$ is an amalgamation base in $\AECK_\mu$,
consider arbitrary $B, C \in \AECK_\mu$ with $M \leqK B$ and $M \leqK C$.
By the hypothesis,
we can fix embeddings $f_B : B \embeds N$ and $f_C : C \embeds N$ such that $f_B \restriction M = f_C \restriction M = \id_M$.
Then $(f_B, f_C, N)$ is an amalgamation of $B$ and $C$ over $M$.
If $\left|N\right| = \mu$, then this amalgamation is already a $\mu$-amalgamation.
Otherwise, 
apply $\LST(\AECK) \leq\mu$ and Remark~\ref{LST-remark} to fix
$D \in \AECK_\mu$ containing every point of $f_B[B] \cup f_C[C]$, and such that $f_B[B] \leqK D \lessK N$ and $f_C[C] \leqK D$.
Then $(f_B, f_C, D)$ is a $\mu$-amalgamation of $B$ and $C$ over $M$.

\item 
Consider arbitrary $p \in S(M)$, and we shall show that $N$ realizes $p$.
As $\LST(\AECK) \leq\mu$,
by Lemma~\ref{get-ambient-lambda} we can fix some $B \in \AECK_\mu$ that realizes $p$.
%
%
Since 
$M \leqK B$, 
by the hypothesis
we fix
an embedding $f : B \embeds N$ such that $f \restriction M = \id_M$.
Then by Lemma~\ref{embedding-realizes-type},
it follows that $N$ realizes $p$. 

Thus $N$ realizes $S(M)$, and it follows from fact~\ref{full-gives-bounded} that
$\left|S(M)\right| \leq \left|N\right|$.

\item
Fix $M^- \in \AECK_\mu$ such that $M^- \leqK M$ and $M^-$ is an amalgamation base in $\AECK_\mu$,
and consider any $C \in \AECK_\mu$ such that $M^- \leqK C$.
Since $M^-$ is an amalgamation base in $\AECK_\mu$,
we can fix a $\mu$-amalgamation $(h, \id_M, B)$
of $C$ and $M$ over $M^-$.
In particular, $B \in \AECK_\mu$ and $M \leqK B$,
so that by the hypothesis 
there is an embedding $f : B \embeds N$ such that $f \restriction M = \id_M$.
Setting $g := f \circ h$,
it is clear that $g : C \embeds N$ is an embedding such that $g \restriction M^- = \id_{M^-}$,
as sought.
\qedhere
\end{enumerate}
\end{proof}

\begin{corollary}\label{univ->AB+stab}
Suppose $N, N' \in \AECK_\mu$ are models such that $N \lessKuniv N'$.
Then:
\begin{enumerate}
\item $N$ is an amalgamation base in $\AECK_\mu$; and 
\item If $\LST(\AECK) \leq\mu$, 
then $N'$ realizes $S(N)$, so that
$\left|S(N)\right| \leq \mu$.
\end{enumerate}
\end{corollary}

%
%


\begin{corollary}\label{homogeneous->saturated}
Suppose $\LST(\AECK) \leq\lambda$, and
$N \in \AECK_{\lambda^+}$ 
is homogeneous in $\lambda^+$ over $\lambda$.
Then:
\begin{enumerate}
\item Every $M \in \AECK_\lambda$ satisfying $M \lessK N$ is an amalgamation base in $\AECK_\lambda$; and 
\item 
$N \in \AECK^\sat_{\lambda^+}$,  
so that in particular,
$\left|S(M)\right| \leq \lambda^+$ for every $M \in \AECK_\lambda$ satisfying $M \lessK N$.
\todo{See also Fact \ref{saturated=homogeneous}.}
\end{enumerate}
\end{corollary}

\begin{corollary}\label{smaller-univ-from-AB}
Suppose $N^-, N, N' \in \AECK_\mu$ are models such that $N^- \leqK N \lessKuniv N'$,
and $N^-$ is an amalgamation base in $\AECK_\mu$.
Then $N^- \lessKuniv N'$.
\todo{This is what we will use to remove the extension axiom!}
\end{corollary}




\begin{lemma}\label{univ-embedding}
Suppose 
$N, N', N^\bullet \in \AECK_\mu$ are models such that $N \lessKuniv N'$, 
and $g : N \embeds N^\bullet$ is an embedding.
Then there is an embedding $h : N^\bullet \embeds N'$ such that $h \circ g = \id_N$.
\end{lemma}

\begin{proof}
Fix a model $N^+ \in \AECK_\mu$ with $N \leqK N^+$ and an isomorphism $\varphi : N^+ \cong N^\bullet$ extending $g$.
Since $N \lessKuniv N'$, 
we can fix an embedding $f : N^+ \embeds N'$ with $f \restriction N = \id_N$.
Then $h := f \circ \varphi^{-1}$ is as sought.
\end{proof}

\begin{fact}\label{stab+AB->universal}
\todo{Does this need $\LST$?  Or any other hypothesis?
See my handwritten notes in notebook from 2 Tevet for the proof,
but it also uses some bookkeeping.}
Suppose that $\AECK_\mu$ satisfies stability, and $N \in \AECK_\mu$ is an amalgamation base in $\AECK_\mu$.
Then there is some $N' \in \AECK_\mu$ such that $N \lessKuniv N'$.
\end{fact}



\begin{theorem}\label{stab+amalg<->universal}
Suppose $\LST(\AECK) \leq\mu$.
Then the following are equivalent:
\begin{enumerate}
\item For every $N \in \AECK_{\mu}$ there is some $N' \in \AECK_{\mu}$ such that $N \lessKuniv N'$;
\item $\AECK_{\mu}$ satisfies amalgamation and stability.
\end{enumerate}
\end{theorem}

\begin{proof}
Combine Corollary~\ref{univ->AB+stab} with Fact~\ref{stab+AB->universal}.
\todo{Clarify this.}
\end{proof}

\begin{fact}\label{universal-saturated}
Suppose that $\AECK_{\lambda^+}$ satisfies amalgamation and stability,
and $\AECK_\lambda$ satisfies amalgamation.
Then for every $N \in \AECK_{\lambda^+}$ there is some $N' \in \AECK^\sat_{\lambda^+}$
such that $N \lessKuniv N'$.
\todo{We don't need this result!}
\end{fact}


\begin{fact}[{\cite[Proposition~1.0.31]{JrSh:875}}]\label{saturated=homogeneous}
Suppose that $\AECK_\lambda$ satisfies the amalgamation property, and $\LST(\AECK) \leq\lambda$.
\todo{Check this proof to see whether there is anything interesting to say.  
Also relate to Lemma \ref{embedding-gives-full&AB} and see Corollary \ref{homogeneous->saturated}.}
Let $N \in \AECK_{\lambda^+}$.
Then $N \in \AECK^\sat_{\lambda^+}$ 
iff $N$ is homogeneous in $\lambda^+$ over $\lambda$.
\end{fact}


\begin{lemma}
Suppose that $\AECK_\lambda$ satisfies JEP, 
and $\LST(\AECK) \leq\lambda$.
If there exists a model that is homogeneous in $\lambda^+$ over $\lambda$,
then for every $M \in \AECK_\lambda$ there exists some $N \in \AECK_{\lambda^+}$ such that
$M \lessK N$ and $N$ is homogeneous in $\lambda^+$ over $\lambda$.
\todo{And therefore saturated; see Corollary \ref{homogeneous->saturated}.
Don't need AP here.}
\end{lemma}

\begin{proof}
Fix $N^\bullet \in \AECK_{\lambda^+}$ that is homogeneous in $\lambda^+$ over $\lambda$,
and let $M \in \AECK_\lambda$ be arbitrary.
By $\LST(\AECK) \leq\lambda$,
fix some model $A \in \AECK_\lambda$ such that $A \lessK N^\bullet$.
Using the JEP, we can fix a joint embedding $(g, \id_A, B)$ of $M$ and $A$.
That is, $B \in \AECK_\lambda$ satisfies $A \leqK B$, and $g : M \embeds B$ is an embedding.
By homogeneity of $N$, we can fix an embedding $f : B \embeds N^\bullet$ such that $f \restriction A = \id_A$.
Considering the embedding $f \circ g : M \embeds N^\bullet$,
we fix some model $N \in \AECK_{\lambda^+}$ that is isomorphic to $N^\bullet$ and such that $M \lessK N$.  
Then $N$ is as sought.
\end{proof}

However, the JEP is not needed for the following Theorem and its corollaries:

\begin{theorem}[cf.~{\cite[Theorem~2.5.8]{JrSh:875}}]\label{get-saturated}
Suppose $\mathfrak s = (\AECK, {\dnf}, S^\bs)$ is a pre-$\lambda$-frame%
\footnote{All we really need for this Theorem and its corollaries is $(\AECK, S^\bs)$ satisfying the stated hypotheses.
The non-forking relation $\dnf$ is irrelevant here, but can be added arbitrarily if desired; 
cf.\ Examples~\ref{examples-pre-frames}((5)\&(6)).}
satisfying 
amalgamation, no maximal model, basic almost stability, and density.
\todo{Verify that this is the optimal list of hypotheses!
Also notice that there are no requirements on $\AECK_{\lambda^+}$.}

Then for every $M \in \AECK_\lambda$ there exists some $N \in \AECK^\sat_{\lambda^+}$ such that
$M \lessK N$.  
\end{theorem}

\begin{proof}
Begin by fixing a bookkeeping function (bijection) $\psi : \lambda^+ \times \lambda^+ \leftrightarrow \lambda^+$
such that $\psi(\gamma,\beta) \geq\gamma$ for all $\gamma,\beta<\lambda^+$.
Let $M \in \AECK_\lambda$ be given.

We shall build, recursively over $\alpha\leq\lambda^+$,
sequences
\[
\Sequence{ M_\alpha | \alpha\leq\lambda^+ } \text{ and }
\Sequence{ \Sequence{ p_{\alpha,\beta} | \beta<\lambda^+ } | \alpha<\lambda^+ } 
\]
such that:
\begin{enumerate}
\item $\left< M_\alpha \mid \alpha\leq\lambda^+ \right>$ 
is a $\lessK$-increasing, continuous sequences of models,
with $M_0 = M$;
\setcounter{condition}{\value{enumi}}
\end{enumerate}
and satisfying the following for every $\alpha <\lambda^+$:
\begin{enumerate}
\setcounter{enumi}{\value{condition}}
\item $M_\alpha \in \AECK_\lambda$; 
\item $\Set{ p_{\alpha,\beta} | \beta<\lambda^+ } = S^\bs(M_\alpha)$;
\item $M_{\alpha+1}$ realizes $p_{\gamma,\beta}$, where $(\gamma,\beta) = \psi^{-1}(\alpha)$.
\end{enumerate}

We carry out the recursive construction as follows:
\begin{itemize}
\item For $\alpha=0$:
Fix $M_0 := M$.

\item For a nonzero limit ordinal $\alpha\leq\lambda^+$:
Define
\[
M_\alpha := \bigcup_{\eta<\alpha} M_\eta.
\]

\item For a successor ordinal $\alpha+1 < \lambda^+$, 
where 
$\Sequence{ M_\eta | \eta\leq\alpha }$ and 
$\langle \langle p_{\eta,\beta} \mid\allowbreak \beta<\lambda^+ \rangle \mid\allowbreak \eta<\alpha \rangle$ 
have already been constructed:

By basic almost stability, we have $\left| S^\bs(M_\alpha) \right| \leq\lambda^+$.
By density and no maximal model (see Lemma~\ref{basic-nonempty}),
we have $S^\bs(M_\alpha) \neq\emptyset$.
Thus we can enumerate $S^\bs(M_\alpha)$ as $\Set{ p_{\alpha,\beta} | \beta<\lambda^+ }$
(possibly with repetition).

Now, fix $(\gamma,\beta) := \psi^{-1}(\alpha)$.
We have $\gamma\leq\alpha$, so that $p_{\gamma,\beta} \in S^\bs(M_\gamma)$ is defined.
By $\LST(\AECK)\leq\lambda$ and Lemma~\ref{get-ambient-lambda},
we fix some $B \in \AECK_\lambda$ that realizes $p_{\gamma,\beta}$.
In particular, $M_\gamma \lessK B$.
As $\AECK_\lambda$ satisfies the amalgamation property, 
we can fix an amalgamation $(f, \id_{M_\alpha}, M_{\alpha+1})$ of $B$ and $M_\alpha$ over $M_\gamma$.
Then $M_{\alpha+1}$ also realizes%
\footnote{It is possible that $p_{\gamma,\beta}$ is already realized in $M_\eta$ for some $\eta\leq\alpha$,
but this does not matter.}
$p_{\gamma,\beta}$ (see Lemma~\ref{embedding-realizes-type}).
We have $M_\alpha \leqK M_{\alpha+1}$, but as $\AECK_\lambda$ has no maximal model,
we may expand $M_{\alpha+1}$ to ensure that $M_\alpha \lessK M_{\alpha+1}$.
\end{itemize}
This completes the recursive construction.

Write $N := M_{\lambda^+}$, so that
$N = \bigcup_{\alpha<\lambda^+} M_\alpha \in \AECK_{\lambda^+}$,
and in particular, $M \lessK N$.
It remains to show that $N \in \AECK^\sat_{\lambda^+}$.

\begin{claim}
For every $\alpha<\lambda^+$ and every $D \in \AECK_\lambda$ satisfying $M_\alpha \leqK D$,
there exists some embedding $h : D \embeds N$ such that $h \restriction M_\alpha = \id_{M_\alpha}$.
\end{claim}

\begin{proof}
Consider arbitrary $\alpha<\lambda^+$ and $D \in \AECK_\lambda$ satisfying $M_\alpha \leqK D$.
We shall attempt to build, 
recursively over $\varepsilon\leq\lambda^+$,
a sequence 
\[
\Sequence{ (\gamma_\varepsilon, D_\varepsilon, g_\varepsilon) | \varepsilon\leq\lambda^+ }
\]
such that:
\begin{enumerate}
\item $\langle \gamma_\varepsilon \mid \varepsilon\leq\lambda^+ \rangle$ 
is a strictly increasing, continuous sequence of ordinals
$\leq\lambda^+$, with $\gamma_0 = \alpha$;

\item $\left< D_\varepsilon \mid \varepsilon\leq\lambda^+ \right>$ 
is a $\leqK$-increasing, continuous sequence of models,
with $D_0 = D$;

\item $\langle g_\varepsilon : M_{\gamma_\varepsilon} \embeds D_\varepsilon \mid\allowbreak \varepsilon\leq\lambda^+ \rangle$
is a $\subseteq$-increasing, continuous sequence of embeddings,
with $g_0 = \id_{M_\alpha}$;
\setcounter{condition}{\value{enumi}}
\end{enumerate}
and satisfying the following for all $\varepsilon <\lambda^+$:
\begin{enumerate}
\setcounter{enumi}{\value{condition}}
\item $\gamma_\varepsilon <\lambda^+$;
\item $D_\varepsilon \in \AECK_\lambda$; 
and
\item $g_\varepsilon[M_{\gamma_\varepsilon}] \neq g_{\varepsilon+1}[M_{\gamma_{\varepsilon+1}}] \cap D_\varepsilon$.
\end{enumerate}

We carry out the recursive construction as follows:
\begin{itemize}
\item For $\varepsilon=0$:
Set $\gamma_0 = \alpha$, $D_0 = D$, and $g_0 = \id_{M_\alpha}$.

\item For a nonzero limit ordinal $\varepsilon\leq\lambda^+$:
Define
\[
\gamma_\varepsilon := \sup_{\eta<\varepsilon} \gamma_\eta; \quad
D_\varepsilon := \bigcup_{\eta<\varepsilon} D_\eta; \quad
g_\varepsilon := \bigcup_{\eta<\varepsilon} g_\eta.
\]

\item For a successor ordinal $\varepsilon+1 < \lambda^+$,
where $(\gamma_\varepsilon, D_\varepsilon, g_\varepsilon)$
has already been constructed,
we consider two cases:
\begin{itemize}
\item Suppose $g_\varepsilon[M_{\gamma_\varepsilon}] = D_\varepsilon$.
That is, $g_\varepsilon : M_{\gamma_\varepsilon} \cong D_\varepsilon$  is an isomorphism,
so that 
$h := g_\varepsilon^{-1} \restriction D$ is an embedding from $D$ into $N$
such that $h \restriction M_\alpha = g_\varepsilon^{-1} \restriction M_\alpha = g_0^{-1} \restriction M_\alpha = \id_{M_\alpha}$.
Thus the Claim is proven in this case, and we terminate the recursive construction here.

\item
Otherwise, $g_\varepsilon[M_{\gamma_\varepsilon}] \lessK D_\varepsilon$.
By the density property, we can fix some $d \in D_\varepsilon \setminus g_\varepsilon[M_{\gamma_\varepsilon}]$ such that 
$\tp(d/g_\varepsilon[M_{\gamma_\varepsilon}]; D_\varepsilon) \in S^\bs(g_\varepsilon[M_{\gamma_\varepsilon}])$.
Fix a model $B \in \AECK_\lambda$ with $M_{\gamma_\varepsilon} \lessK B$ together with an isomorphism $\varphi : B \cong D_\varepsilon$ extending $g_\varepsilon$.
Let $b := \varphi^{-1}(d)$, so that $b \in B$.
Then by invariance of $S^\bs$ under the isomorphism $\varphi^{-1} : D_\varepsilon \cong B$,
it follows that 
\[
\tp(b/M_{\gamma_\varepsilon}; B) = 
\tp(\varphi^{-1}(d)/\varphi^{-1}[g_\varepsilon[M_{\gamma_\varepsilon}]]; B) 
\in S^\bs(\varphi^{-1}[g_\varepsilon[M_{\gamma_\varepsilon}]]) = S^\bs(M_{\gamma_\varepsilon}).
\]
Thus we can fix some $\beta<\lambda^+$ such that $\tp(b/M_{\gamma_\varepsilon}; B) = p_{\gamma_\varepsilon, \beta}$.
Let $\gamma_{\varepsilon+1} := \psi(\gamma_\varepsilon, \beta)+1$.
Then $\gamma_\varepsilon < \gamma_{\varepsilon+1} <\lambda^+$ and
$M_{\gamma_{\varepsilon+1}}$ realizes $p_{\gamma_\varepsilon, \beta}$.
Fix $x \in M_{\gamma_{\varepsilon+1}}$ such that 
$\tp(x/M_{\gamma_\varepsilon}; M_{\gamma_{\varepsilon+1}}) = p_{\gamma_\varepsilon, \beta}$.
Then $\tp(b/M_{\gamma_\varepsilon}; B) = \tp(x/M_{\gamma_\varepsilon}; M_{\gamma_{\varepsilon+1}})$.
Thus by the amalgamation property (so that types are ``best-behaved''),
we can fix a $\lambda$-amalgamation 
$(\id_B, f, C)$ of $B$ and $M_{\gamma_{\varepsilon+1}}$ over $M_{\gamma_\varepsilon}$
such that $f(x) = b$.
In particular, $f : M_{\gamma_{\varepsilon+1}} \embeds C$ is an embedding satisfying 
$f \restriction M_{\gamma_\varepsilon} = \id_{M_{\gamma_\varepsilon}}$.
Then, as $B \leqK C$ and $\varphi : B \cong D_\varepsilon$ is an isomorphism,
we fix some model $D_{\varepsilon+1} \in \AECK_\lambda$ such that $D_\varepsilon \leqK D_{\varepsilon+1}$
together with an isomorphism $\varphi^+ : C \cong D_{\varepsilon+1}$ extending $\varphi$.
Then $g_\varepsilon \subseteq \varphi \subseteq \varphi^+$. 
Let $g_{\varepsilon+1} := \varphi^+ \circ f$.
Then $g_{\varepsilon+1} : M_{\gamma_{\varepsilon+1}} \embeds D_{\varepsilon+1}$ is an embedding
satisfying 
$g_{\varepsilon+1} \restriction M_{\gamma_\varepsilon} = \varphi^+ \circ f \restriction M_{\gamma_\varepsilon} = \varphi^+ \restriction M_{\gamma_\varepsilon} = g_\varepsilon$.
Furthermore, $g_{\varepsilon+1}(x) = \varphi^+(f(x)) = \varphi^+(b) = \varphi(b) = d$,
so that $d \in g_{\varepsilon+1}[M_{\gamma_{\varepsilon+1}}] \cap D_\varepsilon \setminus g_\varepsilon[M_{\gamma_\varepsilon}]$,
so that all the properties of the recursive construction are satisfied at stage $\varepsilon+1$ in this case.
\end{itemize}
\end{itemize}

If we have managed to complete the recursive construction,
then the sequence 
\[
\Sequence{ \left( M_{\gamma_\varepsilon}, D_\varepsilon, g_\varepsilon \right) | \varepsilon\leq\lambda^+ }
\]
violates the statement of Theorem~\ref{image-equal-on-club}.
This contradiction shows that we must have halted the recursive construction at some $\varepsilon<\lambda^+$,
and thus the Claim is proven.
\end{proof}

\begin{claim}
$N$ is homogeneous in $\lambda^+$ over $\lambda$.
\end{claim}

\begin{proof}
Consider arbitrary $A \in \AECK_\lambda$ satisfying $A \lessK N$.
As $A \in [N]^\lambda$ and $\langle M_\alpha \mid\allowbreak \alpha<\lambda^+ \rangle$ is a representation of $N$,
we can fix some $\alpha<\lambda^+$ such that every point of $A$ is in $M_\alpha$.
Since $A \lessK N$ and $M_\alpha \lessK N$,
it follows that
$A$ is necessarily a submodel of $M_\alpha$, so that the coherence property of the AEC 
(Definition~\ref{AEC-def}(5))
guarantees that in fact $A \leqK M_\alpha$.

By the amalgamation property in $\AECK_\lambda$, 
we have in particular (Fact~\ref{AP=AB}) that $A$ is an amalgamation base in $\AECK_\lambda$,
so that by the previous Claim together with Lemma~\ref{embedding-gives-full&AB}(3)
we obtain that for every 
$B \in \AECK_\lambda$ satisfying $A \leqK B$,
there exists some embedding $f : B \embeds N$ such that $f \restriction A = \id_{A}$.

As $N \in \AECK_{\lambda^+}$,
this shows that $N$ is homogeneous in $\lambda^+$ over $\lambda$.
\end{proof}

Thus, by $\LST(\AECK) \leq\lambda$ and Corollary~\ref{homogeneous->saturated},
$N \in \AECK^\sat_{\lambda^+}$, 
as sought.
\end{proof}

\begin{corollary}\label{equiv-basic-AS}
Suppose $\mathfrak s = (\AECK, {\dnf}, S^\bs)$ is a pre-$\lambda$-frame
satisfying 
amalgamation, no maximal model, and density.
Then the following are equivalent:
\begin{enumerate}
\item $\mathfrak s$ satisfies basic almost stability;
\item $\AECK_\lambda$ satisfies almost stability;
\item $\left|S(M)\right| \leq\lambda^+$ for every $M \in \AECK_\lambda$;
\item
For every $M \in \AECK_\lambda$ there exists some $N \in \AECK^\sat_{\lambda^+}$ such that
$M \lessK N$.  
\end{enumerate}
\end{corollary}

\begin{proof}\hfill
\begin{description}
\item[$(3) \implies (2) \implies (1)$]
$S^\bs(A) \subseteq S^\na(A) \subseteq S(A)$ for every $A \in \AECK_\lambda$.

\item[$(1) \implies (4)$]
This is Theorem~\ref{get-saturated}.

\item[$(4) \implies (3)$]
Consider arbitrary $M \in \AECK_\lambda$.
By~(4), fix $N \in \AECK^\sat_{\lambda^+}$ such that $M \lessK N$.  
In particular, $N$ is full over $M$.
Thus, by Fact~\ref{full-gives-bounded},
$\left|S(M)\right| \leq \left|N\right| \leq\lambda^+$.
\qedhere
\end{description}
\end{proof}

\begin{corollary}
Suppose $\lambda\geq\LST(\AECK)$ is such that $\AECK_\lambda$ satisfies amalgamation and no maximal model.
If $\mathfrak s = (\AECK, {\dnf}, S^\bs)$ is a pre-$\lambda$-frame
satisfying density and basic almost stability,
then every pre-$\lambda$-frame that satisfies density must also satisfy basic almost stability.
\end{corollary}

\begin{proof}
By Corollary~\ref{equiv-basic-AS}, since Clause~(2) there depends only on the class $\AECK_\lambda$ and not on the pre-$\lambda$-frame.
\end{proof}

\begin{corollary}
Suppose $\lambda\geq\LST(\AECK)$ is such that $\AECK_\lambda$ satisfies the amalgamation property.
Then the following are equivalent:
\todo{Maybe this should be in Section \ref{section:trivial+*domination} as it relies on results there.}
\begin{enumerate}
\item $\AECK_\lambda$ satisfies no maximal model and almost stability;
\item
For every $M \in \AECK_\lambda$ there exists some $N \in \AECK^\sat_{\lambda^+}$ such that
$M \lessK N$.  
\end{enumerate}
\end{corollary}

\begin{proof}
If $\AECK_\lambda = \emptyset$, then the equivalence is vacuously true.
Thus we assume $\AECK_\lambda \neq \emptyset$,
and let $\mathfrak s = (\AECK, {\dnf}, S^\bs)$ be the trivial $\lambda$-frame of $\AECK$ (see Definition~\ref{trivial-frame-def}).
Then by Fact~\ref{properties-of-trivial}, $\mathfrak s$ is a pre-$\lambda$-frame satisfying the density property.
\begin{description}
\item[$(1) \implies (2)$]
This implication now follows from Corollary~\ref{equiv-basic-AS}($(2) \implies (4)$).

\item[$(2) \implies (1)$]
First, consider arbitrary $M \in \AECK_\lambda$, and we will show that it is not maximal.
By~(2), fix $N \in \AECK_{\lambda^+}$ such that $M \lessK N$.  
Choose some point $b \in N \setminus M$.
Then we apply $\LST(\AECK) \leq\lambda$ and Remark~\ref{LST-remark} to fix
$B \in \AECK_\lambda$ containing $b$ and every point of $M$, and such that $M \lessK B \lessK N$,
where clearly $b \in B \setminus M$.  

Hence there is no maximal model in $\AECK_\lambda$.
Almost stability then follows from Corollary~\ref{equiv-basic-AS}($(4) \implies (2)$).
\qedhere
\end{description}
\end{proof}

\section{Uniqueness triples}
\label{section:uniqueness}


\begin{notation}
Throughout this section, we
suppose that $\mathfrak s = (\AECK, {\dnf}, S^\bs)$ is a pre-$\lambda$-frame satisfying the amalgamation property (in $\lambda$),
so that Galois-types are ``best-behaved'' \cite[p.~64]{Baldwin-Categoricity} and equivalence of $\lambda$-amalgamations is well-defined (see Section~\ref{section:equivalence}).
\end{notation}


We begin this section with the following Lemma, showing that an extension of type can always be amalgamated with the given ambient model that realizes the restricted type:

\begin{lemma}\label{amalg-ext-with-ambient}
Suppose that 
$A, B, C \in \AECK_\lambda$ are such that $A \leqK B$ and $A \leqK C$,
$b \in B$, and $q \in S(C)$.
If $q \restriction A = \tp(b/A; B)$, then there exists a $\lambda$-amalgamation
$(f^F_B, \id_C, F)$ of $B$ and $C$ over $A$ such that $q = \tp(f^F_B(b)/C; F)$.
\end{lemma}

\begin{proof}
Fix an ambient model $D \in \AECK_\lambda$ and $d \in D$ that realizes $q$ in $D$,
that is, such that $C \leqK D$ and $q = \tp(d/C; D)$.
Then $\tp(b/A; B) = q \restriction A = \tp(d/A; D)$.
Thus, since $\AECK_\lambda$ satisfies the amalgamation property,
we can fix a $\lambda$-amalgamation $(f^F_B, \id_D, F)$ of $B$ and $D$ over $A$ such that $f^F_B(b) = d$.
In particular, $F \in \AECK_\lambda$ and $D \leqK F$.
Then $(f^F_B, \id_C, F)$ is a $\lambda$-amalgamation of $B$ and $C$ over $A$,
and $q = \tp(d/C; D) = \tp(d/C; F) = \tp(f^F_B(b)/C; F)$,
as required.
\end{proof}

Our applications of Lemma~\ref{amalg-ext-with-ambient} will typically be 
in conjunction with the extension property, as follows:
\todo{Is this the correct place for this result?}

\begin{corollary}\label{nonfork-amalg-extension}
\todo{Should this be a Lemma?}
Suppose that $(A, B, b) \in \AECK^{3,\bs}$
and $C \in \AECK_\lambda$ are such that $A \leqK C$.
Suppose that $\mathfrak s$ satisfies the extension property. 

Then there exists a $\lambda$-amalgamation 
$(\id_{B}, f^D_{C}, D)$ 
of $B$ and $C$ over $A$ such that 
$\tp(b/f^D_{C}[C]; D)$ does not fork over $A$.
\end{corollary}

\begin{remark}
In the conclusion of Corollary~\ref{nonfork-amalg-extension},
the fact that $(\id_{B}, f^D_{C}, D)$ 
is an amalgamation
of $B$ and $C$ over $A$ implies that
$f^D_{C} \restriction A = \id_A$,
so that in fact $A = f^D_{C}[A] \leqK f^D_{C}[C]$ and
$\tp(b/f^D_{C}[C]; D) \restriction A = \tp(b/A; D) = \tp(b/A; B)$,
meaning that $\tp(b/f^D_{C}[C]; D)$ extends $\tp(b/A; B)$.
\end{remark}

\begin{proof}[Proof of Corollary~\ref{nonfork-amalg-extension}]
Write $p := \tp(b/A; B)$, so that $p \in S^\bs(A)$.
By the extension property,
we can find some $q \in S^\bs(C)$ extending $p$ that does not fork over $A$.
Then, by Lemma~\ref{amalg-ext-with-ambient}, we fix
a $\lambda$-amalgamation
$(f^F_B, \id_C, F)$ of $B$ and $C$ over $A$ such that $q = \tp(f^F_B(b)/C; F)$.

Considering the embedding $f^F_{B} : B \embeds F$,
we fix a model $D \in \AECK_\lambda$ with $B \leqK D$ 
and an isomorphism $\psi : D \cong F$
extending $f^F_{B}$.
As $C \leqK F$, define $f^D_C := \psi^{-1} \restriction C$,
so that $f^D_C : C \embeds D$ is an embedding.
As $f^F_{B} \restriction A = \id_{A}$,
it follows that $f^D_C \restriction A = \psi^{-1} \restriction A = \id_A$,
so that 
$(\id_{B}, f^D_C, D)$ is a $\lambda$-amalgamation of 
$B$ and $C$ over $A$
(that is, in fact, isomorphic to $(f^F_{B}, \id_{C}, F)$).
Furthermore,
since $q = \tp(f^F_{B}(b)/C; F)$
does not fork over $A$,
it follows by invariance of the non-forking relation under the isomorphism $\psi^{-1}$ that
$\tp(b/f^D_C[C]; D) = 
\tp(\psi^{-1}(f^F_{B}(b))/\psi^{-1}[C]; \psi^{-1}[F])$
does not fork over $\psi^{-1}[A] = A$,
as sought.
\end{proof}

We will need a variant of Corollary~\ref{nonfork-amalg-extension} involving isomorphisms:

\begin{corollary}\label{nonfork-isom-extension}
\todo{Should this be a Lemma?}
Suppose that $(A, B, b) \in \AECK^{3,\bs}$, 
$A^*, C^* \in \AECK_\lambda$ are such that $A^* \leqK C^*$,
and $\varphi : A^* \cong A$ is an isomorphism.
Suppose that $\mathfrak s$ satisfies the extension property. 

Then there exists a $\lambda$-amalgamation 
$(\id_{B}, f^D_{C^*}, D)$ 
of $B$ and $C^*$ over $(A^*, \varphi, \id_{A^*})$ such that 
$\tp(b/f^D_{C^*}[C^*]; D)$ does not fork over $A$.

\end{corollary}

\begin{remark}
\todo{Refer to this remark where useful.}
In the conclusion of Corollary~\ref{nonfork-isom-extension},
the fact that $(\id_{B}, f^D_{C^*}, D)$ 
is an amalgamation
of $B$ and $C^*$ over $(A^*, \varphi, \id_{A^*})$ implies that
$\varphi \subseteq f^D_{C^*}$,
so that in fact $A = \varphi[A^*] = f^D_{C^*}[A^*] \leqK f^D_{C^*}[C^*]$ and
$\tp(b/f^D_{C^*}[C^*]; D) \restriction A = \tp(b/A; D) = \tp(b/A; B)$,
meaning that $\tp(b/f^D_{C^*}[C^*]; D)$ extends $\tp(b/A; B)$.
\end{remark}

\begin{proof}[Proof of Corollary~\ref{nonfork-isom-extension}]
Fix some isomorphism $\varphi^+ : C^* \cong C$ extending $\varphi$,
so that $C \in \AECK_\lambda$ and $A \leqK C$.
Then by Corollary~\ref{nonfork-amalg-extension}, we fix
a $\lambda$-amalgamation 
$(\id_{B}, f^D_{C}, D)$ 
of $B$ and $C$ over $A$ such that 
$\tp(b/f^D_{C}[C]; D)$ does not fork over $A$.
Define $f^D_{C^*} := f^D_C \circ \varphi^+$,
so that $f^D_{C^*} : C^* \embeds D$ is an embedding.
Since $f^D_C \restriction A = \id_A$ and $\varphi[A^*] = A$,
it follows that $f^D_{C^*} \restriction A^* = f^D_C \circ \varphi^+ \restriction A^* = f^D_C \circ \varphi = \varphi$,
showing that 
$(\id_{B}, f^D_{C^*}, D)$ 
is a $\lambda$-amalgamation 
of $B$ and $C^*$ over $(A^*, \varphi, \id_{A^*})$.
Furthermore, $f^D_{C^*}[C^*] = f^D_C \circ \varphi^+ [C^*] = f^D_C[C]$,
so that in fact 
$\tp(b/f^D_{C^*}[C^*]; D)$ does not fork over $A$,
as sought.
\end{proof}

Lemma~\ref{amalg-ext-with-ambient} and its corollaries prompt the following question:
Under what circumstances does a non-forking extension uniquely determine the $\lambda$-amalgamation
(up to equivalence of $\lambda$-amalgamations)?

\begin{definition}\label{non-uniqueness}
Suppose that 
$(A, B, b) \in \AECK^{3,\bs}$, and
$C \in \AECK_\lambda$ is such that $A \lessK C$.
We say that \emph{$C$ witnesses the non-uniqueness of $(A, B, b)$} if 
there exist two $\lambda$-amalgamations
$(f^D_B, \id_C, D)$ and $(f^F_B, \id_C, F)$ 
of $B$ and $C$ over $A$ that are not equivalent over $A$ and such that 
$\tp(f^D_B(b)/C; D) = \tp(f^F_B(b)/C; F)$
and this type does not fork over $A$.
\end{definition}

\begin{definition}
\todo{Consider strong uniqueness and weak uniqueness,
where strong uniqueness means, in addition, that there aren't two different extending types.}
A triple $(A, B, b) \in \AECK^{3,\bs}$ is a \emph{uniqueness triple} if there is no $C \in \AECK_\lambda$ witnessing the non-uniqueness
of $(A, B, b)$.
The class of uniqueness triples is denoted $\AECK^{3,\uq}$.
\end{definition}

From \cite[\S4.1]{JrSh:875}:  ``The element $b$ represents the extension $B$ over $A$.''
\todo{Should this refer to uniqueness triples or $*$-domination triples?}

\begin{remark}

If $\mathfrak s$ satisfies the extension and uniqueness properties, 
then by Lemma~\ref{amalg-ext-with-ambient}, our definition of uniqueness triples is equivalent to the one given in~\cite[Definition~4.1.5]{JrSh:875}.
\end{remark}

Intuitively, we may think of $(A,B,b) \in \AECK^{3,\uq}$ as saying that $B = \cl(A \cup \{b\})$.
\todo{Explain this closure concept!
But there are uniqueness triples that are not literal closures; see \cite[Def.\ 2.2.5 and Prop.\ 2.2.6]{JrSh:875} and notes of 28 Shevat.
See also 23 Shevat re closure not always existing.}
However, this intuition breaks down in some cases, as we see from the following example:

\begin{example}
Let the vocabulary $\tau$ consist of a single binary-relation symbol $E$.
Let $\AECK$ be the class of all $\tau$-models $M$ such that $E$ is interpreted in $M$ as an equivalence relation,
and let $\leqK$ be the submodel relation $\subseteq$.
It is clear that $\AECK$ is an AEC with $\LST(\AECK) = \aleph_0$
(in fact every subset is a submodel),
and that $\AECK$ satisfies DAP, JEP, and no maximal model.

Fix any infinite cardinal $\lambda$, 
and define a non-forking relation on $\AECK_\lambda$ as follows.
For $A, C, D \in \AECK_\lambda$ with $A \leqK C \lessK D$ and $d \in D \setminus C$,
we say $\tp(d/C; D)$ does not fork over $A$ provided that:
if there is some $c \in C$ such that $c \mathrel{E^D} d$,
then there is some $a \in A$ such that $a \mathrel{E^D} d$.
Define $S^\bs$ as in Example~\ref{examples-pre-frames}((3) or (4)).%
\footnote{The two definitions of $S^\bs$ coincide in this case!}
Then $(\AECK, {\dnf}, S^\bs)$ is a type-full pre-$\lambda$-frame
satisfying density, transitivity, existence, extension, uniqueness, continuity, local character, and symmetry.

Now, let $C := \{ a, b, c\}$ be a model, 
with the equivalence relation $E^C$ dividing $C$ into two equivalence classes, $\{a\}$ and $\{b,c\}$.
Let $A := \{a\}$ and $B := \{a,b\}$ be submodels of $C$.
It is easy to see that both $(A, B, b)$ and $(A, C, b)$ are uniqueness triples, 
\todo{Add some details here?}
even though $B \lessK C$.
In particular, $C$ is not the closure of $A \cup \{b\}$, even though 
$(A,C,b) \in \AECK^{3,\uq}$.
\end{example}

We broaden Definition~\ref{non-uniqueness} in order to include witnesses to non-uniqueness via isomorphism:

\begin{definition}\label{def-non-uniqueness-via}
Suppose that $(A, B, b) \in \AECK^{3,\bs}$, 
$A^*, C^* \in \AECK_\lambda$ are such that $A^* \lessK C^*$,
and $\varphi : A^* \cong A$ is an isomorphism.
We say that \emph{$C^*$ witnesses the non-uniqueness of $(A, B, b)$ via~$\varphi$} if 
there are some $C \in \AECK_\lambda$ and isomorphism $\varphi^+ : C^* \cong C$ extending $\varphi$,
where $C$ witnesses the non-uniqueness of $(A, B, b)$.
\end{definition}

\begin{lemma}\label{unique-equiv-via}
Suppose that $(A, B, b) \in \AECK^{3,\bs}$, 
$A^* \in \AECK_\lambda$,
and $\varphi : A^* \cong A$ is an isomorphism.
Then the following are equivalent:
\begin{enumerate}
\item $(A, B, b) \in \AECK^{3,\uq}$;
\item There is no $C^* \in \AECK_\lambda$ witnessing the non-uniqueness
of $(A, B, b)$ via~$\varphi$.
\end{enumerate}
\end{lemma}

\begin{proof}\hfill
\begin{description}
\item[$(1) \implies (2)$]
Clear from the definitions.

\item[$\neg(1) \implies \neg(2)$]
Suppose $C$ witnesses the non-uniqueness of $(A, B, b)$.
In particular, $C \in \AECK_\lambda$ and $A \lessK C$.
Let $\psi : C \cong C^*$ be some isomorphism extending $\varphi^{-1}$,
so that $C^* \in \AECK_\lambda$ and $A^* \lessK C^*$.
Then $\psi^{-1} : C^* \cong C$ is an isomorphism  extending $\varphi$.
It follows that $C^* \in \AECK_\lambda$ witnesses the non-uniqueness
of $(A, B, b)$ via~$\varphi$.
\qedhere
\end{description}
\end{proof}

\begin{lemma}\label{non-uniqueness-isom}
Suppose that $(A, B, b) \in \AECK^{3,\bs}$, 
$A^*, C^*, C^\bullet \in \AECK_\lambda$ are such that $A^* \lessK C^*$,
and $\varphi : A^* \cong A$ and $\psi : C^* \cong C^\bullet$ are isomorphisms.
If $C^*$ witnesses the non-uniqueness of $(A, B, b)$ via~$\varphi$,
then $C^\bullet$ witnesses the non-uniqueness of $(A, B, b)$ via $\varphi \circ (\psi \restriction A^*)^{-1}$.
\todo{Should we highlight some special cases as corollaries, such as $\psi \restriction A^* = \id_{A^*}$ or $\varphi = \id_A$?
Compare Lemma \ref{non-uniqueness-embedding}.}
\end{lemma}

\begin{proof}
Suppose $C^*$ witnesses the non-uniqueness of $(A, B, b)$ via~$\varphi$.
Then we can fix $C \in \AECK_\lambda$ and an isomorphism $\varphi^+ : C^* \cong C$ extending $\varphi$,
where $C$ witnesses the non-uniqueness of $(A, B, b)$.
Let $A^\bullet := \psi[A^*]$.
Then $A^\bullet \lessK C^\bullet$, $\varphi \circ (\psi \restriction A^*)^{-1} : A^\bullet \cong A$ is an isomorphism, 
and $\varphi^+ \circ \psi^{-1} : C^\bullet \cong C$ is an isomorphism extending $\varphi \circ (\psi \restriction A^*)^{-1}$.
Thus $C^\bullet$ witnesses the non-uniqueness of $(A, B, b)$ via $\varphi \circ (\psi \restriction A^*)^{-1}$.
\end{proof}

\begin{lemma}\label{equiv-via}
Suppose that $(A, B, b) \in \AECK^{3,\bs}$, 
$A^*, C^* \in \AECK_\lambda$ are such that $A^* \lessK C^*$,
and $\varphi : A^* \cong A$ is an isomorphism.
Then the following are equivalent:
\begin{enumerate}
\item $C^*$ witnesses the non-uniqueness of $(A, B, b)$ via~$\varphi$;
\item There exist two $\lambda$-amalgamations
$(f^{D^*}_B, \id_{C^*}, D^*)$ and $(f^{F^*}_B, \id_{C^*}, F^*)$ 
of $B$ and $C^*$ over $(A^*, \varphi, \id_{A^*})$ that are not equivalent and such that 
$\tp(f^{D^*}_B(b)/C^*; D^*) = \tp(f^{F^*}_B(b)/C^*; F^*)$
and this type does not fork over $A^*$.
\end{enumerate}
\end{lemma}

\begin{proof}\hfill
\begin{description}
\item[$(1) \implies (2)$]
Fix $C \in \AECK_\lambda$ and an isomorphism $\varphi^+ : C^* \cong C$ extending $\varphi$,
where $C$ witnesses the non-uniqueness of $(A, B, b)$.
Then we can fix two $\lambda$-amalgamations
$(f^D_B, \id_C, D)$ and $(f^F_B, \id_C, F)$ 
of $B$ and $C$ over $A$ that are not equivalent over $A$ and such that 
$\tp(f^D_B(b)/C; D) = \tp(f^F_B(b)/C; F)$
and this type does not fork over $A$.
In particular, $D, F \in \AECK_\lambda$, $C \leqK D$, and $C \leqK F$.
Let $\psi_D : D \cong D^*$ and $\psi_F : F \cong F^*$ be isomorphisms each extending $(\varphi^+)^{-1}$,
so that $D^*, F^* \in \AECK_\lambda$, $C^* \leqK D^*$, and $C^* \leqK F^*$.
Then $(\psi_D \circ f^D_B,\allowbreak \id_{C^*}, D^*)$ and $(\psi_F \circ f^F_B, \id_{C^*}, F^*)$ 
are two $\lambda$-amalgamations of $B$ and $C^*$ over $(A^*, \varphi, \id_{A^*})$.

If there exists $(f^G_{D^*}, f^G_{F^*}, G)$ witnessing the equivalence of the two $\lambda$-amalgamations 
$(\psi_D \circ f^D_B, \id_{C^*}, D^*)$ and $(\psi_F \circ f^F_B, \id_{C^*}, F^*)$ 
of $B$ and $C^*$ over $(A^*, \varphi, \id_{A^*})$,
then $(f^G_{D^*} \circ \psi_D, f^G_{F^*} \circ \psi_F, G)$ would witness the equivalence of the two $\lambda$-amalgamations 
$(f^D_B, \id_{C}, D)$ and $(f^F_B, \id_{C}, F)$ 
of $B$ and $C$ over $A$, a contradiction.
Thus, the two $\lambda$-amalgamations of $B$ and $C^*$ over $(A^*, \varphi, \id_{A^*})$ are not equivalent.

Furthermore, since $\tp(f^D_B(b)/C; D) = \tp(f^F_B(b)/C; F)$,
we can fix a $\lambda$-amalgamation $(f^G_{D}, f^G_{F}, G)$ of $D$ and $F$ over $C$
such that $f^G_D(f^D_B(b)) = f^G_F(f^F_B(b))$.
But then $(f^G_{D} \circ \psi_D^{-1}, f^G_{F} \circ \psi_F^{-1}, G)$ is a $\lambda$-amalgamation of $D^*$ and $F^*$ over $C^*$
such that $f^G_D \circ \psi_D^{-1}(\psi_D \circ f^D_B(b)) = f^G_F \circ \psi_F^{-1}(\psi_F \circ f^F_B(b))$,
showing that $\tp(\psi_D \circ f^D_B (b)/C^*; D^*) = \tp(\psi_F \circ f^F_B (b)/C^*; F^*)$.

Finally, since $\tp(f^D_B(b)/C; D)$ does not fork over $A$,
it follows from invariance of the non-forking relation under the isomorphism $\psi_D : D \cong D^*$ that
$\tp(\psi_D \circ f^D_B(b)/C^*; D^*)$ does not fork over $A^*$,
completing the verification of~(2).

\item[$(2) \implies (1)$]
Fix two $\lambda$-amalgamations
$(f^{D^*}_B, \id_{C^*}, D^*)$ and $(f^{F^*}_B, \id_{C^*}, F^*)$ 
of $B$ and $C^*$ over $(A^*, \varphi, \id_{A^*})$ satisfying (2).
In particular, $D^*, F^* \in \AECK_\lambda$, $C^* \leqK D^*$, and $C^* \leqK F^*$.
Let $\varphi^+ : C^* \cong C$ be some isomorphism extending $\varphi$,
so that $C \in \AECK_\lambda$ and $A \lessK C$.
Let $\psi_D : D^* \cong D$ and $\psi_F : F^* \cong F$ be isomorphisms each extending $\varphi^+$,
so that $D, F \in \AECK_\lambda$, $C \leqK D$, and $C \leqK F$.
Then $(\psi_D \circ f^{D^*}_B, \id_{C}, D)$ and $(\psi_F \circ f^{F^*}_B, \id_{C}, F)$ 
are two $\lambda$-amalgamations of $B$ and $C$ over $A$.
By arguments similar to those in the proof of $(1) \implies (2)$,
these two $\lambda$-amalgamations show that $C$ witnesses the non-uniqueness of $(A, B, b)$,
thereby verifying (1).
\qedhere
\end{description}
\end{proof}

The following corollary shows that 
Definition~\ref{non-uniqueness} is the particular case of Definition~\ref{def-non-uniqueness-via} where $\varphi = \id_A$.
That is, if $\varphi = \id_A$ in Definition~\ref{def-non-uniqueness-via},
then we may omit ``via~$\varphi$'':

\begin{corollary}\label{omit-via-phi}
Suppose that $(A, B, b) \in \AECK^{3,\bs}$, and
$C \in \AECK_\lambda$ is such that $A \lessK C$.
Then the following are equivalent:
\begin{enumerate}
\item $C$ witnesses the non-uniqueness of $(A, B, b)$ via~$\id_A$;
\item $C$ witnesses the non-uniqueness of $(A, B, b)$.
\end{enumerate}
\end{corollary}

\begin{proof}
This is the case $\varphi := \id_A$ of Lemma~\ref{equiv-via}.
\end{proof}

\begin{lemma}\label{uniqueness-saturated}
Suppose that $(A, B, b) \in \AECK^{3,\bs}$,
$A^* \in \AECK_\lambda$ and $N \in \AECK^\sat_{\lambda^+}$ are such that $A^* \lessK N$,
and $\varphi : A^* \cong A$ is an isomorphism.
Then the following are equivalent:
\begin{enumerate}
\item $(A, B, b) \in \AECK^{3,\uq}$;
\item There is no $C^\bullet \in \AECK_\lambda$ witnessing the non-uniqueness
of $(A, B, b)$ via~$\varphi$, with $A^\bullet \lessK C^\bullet \lessK N$.
\end{enumerate}
\end{lemma}

\begin{proof}\hfill
\begin{description}
\item[$(1) \implies (2)$]
By Lemma~\ref{unique-equiv-via}.


\item[$\neg(1) \implies \neg(2)$]
Suppose $(A, B, b) \notin \AECK^{3,\uq}$.
By Lemma~\ref{unique-equiv-via},
we fix some $C^* \in \AECK_\lambda$ witnessing the non-uniqueness of $(A, B, b)$ via~$\varphi$.
In particular, 
$A^* \lessK C^*$.
By Fact~\ref{saturated=homogeneous},  
since $\LST(\AECK) \leq\lambda$ and $\AECK_\lambda$ satisfies the amalgamation property,
$N \in \AECK^\sat_{\lambda^+}$ is equivalent to the statement that
$N$ is a homogeneous model in $\lambda^+$ over $\lambda$.
Thus, since $A^* \lessK N$ and $A^* \lessK C^*$,
there is an embedding $f : C^* \embeds N$ with $f \restriction A^* = \id_{A^*}$.
Let $C^\bullet := f[C^*]$, so that $C^\bullet \in \AECK_\lambda$ and $A^* \lessK C^\bullet \lessK N$.
Viewing $f : C^* \cong C^\bullet$ as an isomorphism extending $\id_{A^\bullet}$,
we see by Lemma~\ref{non-uniqueness-isom} that in fact 
$C^\bullet$ witnesses the non-uniqueness of $(A, B, b)$ via~$\varphi$,
as sought.
\qedhere
\end{description}
\end{proof}

\begin{lemma}\label{non-uniqueness-extension}
Suppose that $(A, B, b) \in \AECK^{3,\bs}$, 
$A^*, C^*, C^{**} \in \AECK_\lambda$ are such that $A^* \lessK C^* \lessK C^{**}$,
and $\varphi : A^* \cong A$ is an isomorphism.
Suppose that $\mathfrak s$ satisfies transitivity and extension. 
If $C^*$ witnesses the non-uniqueness of $(A, B, b)$ via~$\varphi$,
then so does $C^{**}$.
\end{lemma}

\begin{proof}
Suppose $C^*$ witnesses the non-uniqueness of $(A, B, b)$ via~$\varphi$.
Then we can fix $C \in \AECK_\lambda$ and an isomorphism $\varphi^+ : C^* \cong C$ extending $\varphi$,
where $C$ witnesses the non-uniqueness of $(A, B, b)$.
Fix some isomorphism $\varphi^{++} : C^{**} \cong C^\bullet$ extending $\varphi^+$,
so that $C^\bullet \in \AECK_\lambda$ and $C \lessK C^\bullet$.
We will show that $C^\bullet$ witnesses the non-uniqueness of $(A, B, b)$.

Since $C$ witnesses the non-uniqueness of $(A, B, b)$,
in particular, $A \lessK C$, and
we can fix two $\lambda$-amalgamations
$(f^D_B, \id_C, D)$ and $(f^F_B, \id_C, F)$ 
of $B$ and $C$ over $A$ that are not equivalent over $A$ and such that 
$\tp(f^D_B(b)/C; D) = \tp(f^F_B(b)/C; F)$
and this type does not fork over $A$.
Write $p := \tp(b/A; B)$ and $q := \tp(f^D_B(b)/C; D)$, so that $p \in S^\bs(A)$, $q \in S^\bs(C)$, and $q \restriction A = p$.

By the extension property,
we can find some $r \in S^\bs(C^\bullet)$ extending $q$ that does not fork over $C$.
Since 
$r \restriction C = q$ does not fork over $A$,
it follows by the transitivity property that $r$ does not fork over $A$.

Since $C \leqK D$, $C \leqK C^\bullet$, and $r \restriction C = \tp(f^D_B(b)/C; D)$, 
Lemma~\ref{amalg-ext-with-ambient} gives us a $\lambda$-amalgamation
$(f^{D^*}_D, \id_{C^\bullet}, D^*)$ of $D$ and $C^\bullet$ over $C$ such that 
$\tp(f^{D^*}_D(f^D_B(b))/C^\bullet; D^*) = r$.
Then $(f^{D^*}_D \circ f^D_B, \id_{C^\bullet}, D^*)$ is a $\lambda$-amalgamation of $B$ and $C^\bullet$ over $A$.

Similarly, since $C \leqK F$, $C \leqK C^\bullet$, and $r \restriction C = \tp(f^F_B(b)/C; F)$, 
Lemma~\ref{amalg-ext-with-ambient} gives us a $\lambda$-amalgamation
$(f^{F^*}_F, \id_{C^\bullet}, F^*)$ of $F$ and $C^\bullet$ over $C$ such that $r = \tp(f^{F^*}_F(f^F_B(b))/C^\bullet; F^*)$,
so that $(f^{F^*}_F \circ f^F_B, \id_{C^\bullet}, F^*)$ is a $\lambda$-amalgamation of $B$ and $C^\bullet$ over $A$.

If there exists $(f^G_{D^*}, f^G_{F^*}, G)$ witnessing the equivalence of the two $\lambda$-amalgamations 
$(f^{D^*}_D \circ f^D_B, \id_{C^\bullet}, D^*)$ and $(f^{F^*}_F \circ f^F_B, \id_{C^\bullet}, F^*)$ 
of $B$ and $C^\bullet$ over $A$,
then $(f^G_{D^*} \circ f^{D^*}_D,\allowbreak f^G_{F^*} \circ f^{F^*}_F, G)$ would witness the equivalence of the two $\lambda$-amalgamations 
$(f^D_B, \id_{C}, D)$ and $(f^F_B, \id_{C}, F)$ 
of $B$ and $C$ over $A$, a contradiction.
Thus, the two $\lambda$-amalgamations of $B$ and $C^\bullet$ over $A$ are not equivalent.
Furthermore, $\tp(f^{D^*}_D \circ f^D_B (b)/C^\bullet; D^*) = r = \tp(f^{F^*}_F \circ f^F_B (b)/C^\bullet; F^*)$,
and we have already seen that $r$ does not fork over $A$.
Altogether, this means that $C^\bullet$ witnesses the non-uniqueness of $(A, B, b)$.
Since $\varphi^{++} : C^{**} \cong C^\bullet$ extends $\varphi$,
it follows that $C^{**}$ witnesses the non-uniqueness of $(A, B, b)$ via~$\varphi$,
as sought.
\end{proof}

While non-uniqueness guarantees a non-forking extension realized in two different amalgamations,
the extension property together with Corollary~\ref{nonfork-isom-extension}
guarantees the existence of at least one amalgamation with a non-forking extension.
Thus we obtain the following result, which will be crucial in our applications
(see the proofs of Theorems \ref{main-thm} and~\ref{repres-uq-extend-triple}).

\begin{lemma}\label{get-desired-amalg}
\todo[color=green]{Added on Sept.~6; Improve this by focusing on the witnessing type.}
Suppose that $(A, B, b) \in \AECK^{3,\bs}$,
$A^*, C^* \in \AECK_\lambda$ and $N^\bullet \in \AECK_{\geq\lambda}$ are such that $A^* \leqK C^* \lessK N^\bullet$,
and $\varphi : A^* \cong A$ is an isomorphism.
Suppose that
$\Omega : B \embeds N^\bullet$ is an embedding 
such that $\Omega \circ \varphi = \id_{A^*}$.
Suppose either that $\mathfrak s$ satisfies the extension property,
or that $C^*$ witnesses the non-uniqueness of $(A, B, b)$ via~$\varphi$.

Then there exists a $\lambda$-amalgamation
$(\id_B, f^D_{C^*}, D)$ of $B$ and $C^*$ over $(A^*, \varphi, \id_{A^*})$ such that:
\begin{enumerate}
\item $\tp(b/f^D_{C^*}[C^*]; D)$ does not fork over $A$;
and
\item If $C^*$ witnesses the non-uniqueness of $(A, B, b)$ via~$\varphi$, 
then $(\id_B, f^D_{C^*}, D) \centernot{E} (\Omega, \id_{C^*}, N^\bullet)$
(where $(\Omega, \id_{C^*}, N^\bullet)$ is an
amalgamation of 
$B$ and $C^*$ over $(A^*, \varphi, \id_{A^*})$ by Fact~\ref{predicted-is-amalgamation};
and $\centernot{E}$ is the negation of the relation $E$ between amalgamations of 
$B$ and $C^*$ over $(A^*, \varphi, \id_{A^*})$, as in Definition~\ref{def-E}).
\end{enumerate}
\end{lemma}

\begin{proof}
We consider two cases:
\begin{itemize}
\item
Suppose that
$C^*$ witnesses the non-uniqueness of $(A, B, b)$ via~$\varphi$.
(In particular, $A^* \lessK C^*$.)
Then, by Lemma~\ref{equiv-via},
there exist two $\lambda$-amalgamations
$(f^{D^*}_{B}, \id_{C^*}, D^*)$ and $(f^{F^*}_{B}, \id_{C^*}, F^*)$ 
of ${B}$ and $C^*$ over $(A^*, \varphi, \id_{A^*})$ that are not equivalent and such that 
\todo{Investigate:  How is this equality used?}
$\tp(f^{D^*}_{B}(b)/C^*; D^*) = \tp(f^{F^*}_{B}(b)/C^*; F^*)$
and this type does not fork over $A^*$.

By Fact~\ref{predicted-is-amalgamation}
we may consider
the 
amalgamation 
$(\Omega, \id_{C^*}, N^\bullet)$ of 
$B$ and $C^*$ over $(A^*, \varphi, \id_{A^*})$.
Since $(f^{D^*}_{B}, \id_{C^*}, D^*) \centernot{E} (f^{F^*}_{B}, \id_{C^*}, F^*)$, 
by Lemma~\ref{connect-E-big} 
(since $D^*, F^* \in \AECK_\lambda$, even though possibly $\left|N^\bullet\right| > \lambda$)
it cannot be that both
$(f^{D^*}_{B}, \id_{C^*}, D^*) \mathrel{E} (\Omega, \id_{C^*}, N^\bullet)$ and
$(f^{F^*}_{B}, \id_{C^*}, F^*) \mathrel{E} (\Omega, \id_{C^*}, N^\bullet)$.
Thus, without loss of generality,
we may assume that 
$(f^{D^*}_{B}, \id_{C^*}, D^*) \centernot{E} (\Omega, \id_{C^*}, N^\bullet)$.

Considering the embedding $f^{D^*}_{B} : B \embeds D^*$,
we fix a model $D \in \AECK_\lambda$ with $B \leqK D$ 
and an isomorphism $\psi : D \cong D^*$
extending $f^{D^*}_{B}$.
Define $f^D_{C^*} := \psi^{-1} \restriction C^*$,
so that $f^D_{C^*} : C^* \embeds D$ is an embedding.
As $f^{D^*}_{B} \circ \varphi = \id_{A^*}$,
it follows that $f^D_{C^*} \restriction A^* = \psi^{-1} \restriction A^* = \varphi$,
so that in particular, $\varphi \subset f^D_{C^*}$.
Since $\tp(f^{D^*}_{B}(b)/C^*; D^*)$
does not fork over $A^*$,
it follows by invariance of the non-forking relation under the isomorphism $\psi^{-1}$ that
$\tp(b/f^D_{C^*}[C^*]; D) = 
\tp(\psi^{-1}(f^{D^*}_{B}(b))/\psi^{-1}[C^*]; \psi^{-1}[D^*])$
does not fork over $\psi^{-1}[A^*] = A$.
Furthermore,
$(\id_{B}, f^D_{C^*}, D)$ is a $\lambda$-amalgamation of 
$B$ and $C^*$ over $(A^*, \varphi, \id_{A^*})$
that is isomorphic to $(f^{D^*}_{B}, \id_{C^*}, D^*)$,
so that 
by Fact~\ref{isom+equiv} we obtain 
$(\id_{B}, f^D_{C^*}, D) \centernot{E} (\Omega, \id_{C^*}, N^\bullet)$,
as required in this case.

\item 
Otherwise,
by our hypothesis, $\mathfrak s$ must satisfy the extension property,
and the result follows from Corollary~\ref{nonfork-isom-extension},
as Clause~(2) of our conclusion is irrelevant.
\qedhere
\end{itemize}
\end{proof}

\begin{fact}\label{invariance-uniqueness}
If $(A, B, b) \in \AECK^{3,\uq}$, $B^* \in \AECK_\lambda$,
and $\varphi : B \cong B^*$ is an isomorphism,
then $(\varphi[A], B^*, \varphi(b)) \in \AECK^{3,\uq}$.
\end{fact}

That is,
uniqueness is preserved by isomorphisms.

Unlike the class of basic triples,
the class of uniqueness triples is not closed under expansion of the ambient model.
However,
the following result ensures that uniqueness is preserved when shrinking the ambient model:

\begin{lemma}
Suppose that $A, B, B^*, C \in \AECK_\lambda$ and $b \in B \setminus A$ are such that
$A \lessK B \leqK B^*$, $A \lessK C$, and $(A, B, b) \in \AECK^{3,\bs}$.
\todo{Optimize the hypotheses here.}
If $C$ witnesses the non-uniqueness of $(A, B, b)$, then it also witnesses the non-uniqueness of $(A, B^*, b)$.
\end{lemma}

\begin{proof}
First, notice that $\tp(b/A; B) = \tp(b/A; B^*)$,
so that also $(A, B^*, b) \in \AECK^{3,\bs}$.

Suppose that $C$ witnesses the non-uniqueness of $(A, B, b)$.
Then we can fix two $\lambda$-amalgamations
$(f^D_B, \id_C, D)$ and $(f^F_B, \id_C, F)$ 
of $B$ and $C$ over $A$ that are not equivalent over $A$ and such that 
$\tp(f^D_B(b)/C; D) = \tp(f^F_B(b)/C; F)$
and this type does not fork over $A$.
In particular, $D, F \in \AECK_\lambda$, $C \leqK D$, and $C \leqK F$.
By the amalgamation property, we can fix $\lambda$-amalgamations
$(f^{D^*}_{B^*}, \id_D, D^*)$ of $B^*$ and $D$ over $(B, \id_B, f^D_B)$, and
$(f^{F^*}_{B^*}, \id_F, F^*)$ of $B^*$ and $F$ over $(B, \id_B, f^F_B)$.
Then $(f^{D^*}_{B^*}, \id_C, D^*)$ and $(f^{F^*}_{B^*}, \id_C, F^*)$ 
are two $\lambda$-amalgamations of $B^*$ and $C$ over $A$.

If there exists $(f^G_{D^*}, f^G_{F^*}, G)$ witnessing the equivalence of the two $\lambda$-amalgamations 
$(f^{D^*}_{B^*}, \id_C, D^*)$ and $(f^{F^*}_{B^*}, \id_C, F^*)$ 
of $B^*$ and $C$ over $A$,
then $(f^G_{D^*} \circ \id_D,\allowbreak f^G_{F^*} \circ \id_F,\allowbreak G)$ 
would witness the equivalence of the two $\lambda$-amalgamations 
$(f^D_B, \id_{C}, D)$ and $(f^F_B, \id_{C}, F)$ 
of $B$ and $C$ over $A$, a contradiction.
Thus, the two $\lambda$-amalgamations of $B^*$ and $C$ over $A$ are not equivalent.

Furthermore, since $b \in B$, $f^{D^*}_{B^*} \restriction B = f^D_B$, and $f^{F^*}_{B^*} \restriction B = f^F_B$, 
it follows that
$\tp(f^{D^*}_{B^*}(b)/C; D^*) = \tp(f^{D}_{B}(b)/C; D^*) = \tp(f^{D}_{B}(b)/C; D) = 
\tp(f^F_B(b)/C; F) = \tp(f^F_B(b)/C; F^*) = \tp(f^{F^*}_{B^*}(b)/C; F^*)$,
and by hypothesis, this type does not fork over $A$.
Altogether, this means that $C$ witnesses the non-uniqueness of $(A, B^*, b)$,
as sought.
\end{proof}

\begin{corollary}
Suppose that $A, B, B^* \in \AECK_\lambda$ are such that
$A \lessK B \leqK B^*$, and $b \in B \setminus A$. 
If $(A, B^*, b) \in \AECK^{3,\uq}$, then $(A, B, b) \in \AECK^{3,\uq}$.
\end{corollary}

\begin{lemma}
Suppose that $\mathfrak s$ satisfies transitivity.
Suppose that $A^*, A, B, C \in \AECK_\lambda$ and $b \in B \setminus A$ are such that
$A^* \leqK A \lessK B$, $A \lessK C$, and $\dnf(A^*, A, b, B)$. 
If $C$ witnesses the non-uniqueness of $(A, B, b)$, then it also witnesses the non-uniqueness of $(A^*, B, b)$.
\end{lemma}

\begin{proof}
First, since $\dnf(A^*, A, b, B)$,
it follows that $\tp(b/A; B) \in S^\bs(A)$ and also $\tp(b/A^*; B) \in S^\bs(A^*)$.

Suppose that $C$ witnesses the non-uniqueness of $(A, B, b)$.
Then we can fix two $\lambda$-amalgamations
$(f^D_B, \id_C, D)$ and $(f^F_B, \id_C, F)$ 
of $B$ and $C$ over $A$ that are not equivalent over $A$ and such that 
$\tp(f^D_B(b)/C; D) = \tp(f^F_B(b)/C; F)$
and this type does not fork over $A$.
As $A^* \leqK A$, the two $\lambda$-amalgamations are also over $A^*$,
and remain non-equivalent over $A^*$.
Letting $r := \tp(f^D_B(b)/C; D)$,
we have $r \restriction A = \tp(f^D_B(b)/A; D) = \tp(b/A; B)$.
Since $r$ does not fork over $A$ and by hypothesis $r \restriction A$ does not fork over $A^*$,
we conclude by transitivity that $r$ does not fork over $A^*$.
Altogether, $C$ witnesses the non-uniqueness of $(A^*, B, b)$.
\end{proof}

\begin{corollary}
Suppose that $\mathfrak s$ satisfies transitivity,
$A^*, A, B \in \AECK_\lambda$ are such that
$A^*  \leqK A \lessK B$, and $b \in B \setminus A$. 
If $(A^*, B, b) \in \AECK^{3,\uq}$ and $\dnf(A^*, A, b, B)$, then $(A, B, b) \in \AECK^{3,\uq}$.
\end{corollary}


\section{A more useful form of \texorpdfstring{$\diamondsuit$}{the diamond axiom}}
\label{section:diamond-H-kappa}

Throughout this section, $\kappa$ denotes a regular uncountable cardinal.
(In our intended applications in this paper, $\kappa$ will be $\lambda^+$.)

Jensen~\cite{MR0309729} introduced the diamond axiom $\diamondsuit(\kappa)$ axiom to predict subsets of $\kappa$.
But usually what we want to guess are subsets of some structure of size $\kappa$,
not necessarily sets of ordinals.
\todo{In our application, we will want to predict a function $\Omega : \kappa \to \kappa$
using approximating functions $\Omega_\beta$.}
Encoding the desired sets as sets of ordinals is cumbersome, 
and distracts us from properly applying the guessing/predicting power of $\diamondsuit$.
\todo{Expand this explanation.}
Thus, we appeal to a more versatile formulation of $\diamondsuit$, introduced by Assaf Rinot and the first author:

\begin{definition}[{\cite[Definition~2]{Souslin-note}, cf.~\cite[Definition~2.1]{paper22}}]
$\diamondsuit^-(\mathbf H_\kappa)$
asserts the existence of
a sequence $\langle \Omega_\beta \mid \beta < \kappa \rangle$ of elements of $\mathbf H_\kappa$
such that for every parameter $z\in \mathbf H_{\kappa^+}$ and every subset $\Omega \subseteq \mathbf H_\kappa$,
there exists an elementary submodel $\mathcal M \esm \mathbf H_{\kappa^+}$ with $z \in \mathcal M$, such that
$\kappa^{\mathcal M} := \mathcal M\cap\kappa$ is an ordinal $<\kappa$ and
$\mathcal M\cap \Omega=\Omega_{\kappa^{\mathcal M}}$.
\end{definition}

Here, $\mathbf H_\theta$ denotes the collection of all sets of hereditary cardinality less than $\theta$,
and $\esm$ denotes the first-order elementary-submodel relation between standard models of the vocabulary $\{ {\in} \}$.

\begin{fact}[{\cite[Lemma~2.2]{paper22}}]\label{diamond-H_kappa-equiv}
For any regular uncountable cardinal $\kappa$:
\[
\diamondsuit(\kappa) \iff \diamondsuit^-(\mathbf H_\kappa).
\]
\end{fact}

The proof of \cite[Lemma~2.2]{paper22} \todo{Decide whether this is the best reference.} also shows the following:

\begin{fact}\label{diamond-H_kappa-stat}
If $\langle \Omega_\beta \mid \beta < \kappa \rangle$ is a sequence witnessing $\diamondsuit^-(\mathbf H_\kappa)$,
then for every parameter $z\in \mathbf H_{\kappa^+}$ and every subset $\Omega \subseteq \mathbf H_\kappa$,
the following set is stationary in~$\kappa$:
\[
\Set{ \beta<\kappa | \exists \mathcal M \esm \mathbf H_{\kappa^+} \left[ z \in \mathcal M, 
\beta = \mathcal M \cap \kappa, \mathcal M \cap \Omega = \Omega_\beta \right] }
\]
\end{fact}


\section{Density of uniqueness triples}
\label{section:density}

Here we state and prove the Main Theorem, that assuming $\diamondsuit(\lambda^+)$ and reasonably minimal hypotheses on the pre-$\lambda$-frame,
the class of uniqueness triples is dense among the basic triples.
Significantly, we do not require the pre-$\lambda$-frame to satisfy any of the extension, uniqueness, stability, symmetry, or local character properties.
As we shall see,
\todo{Are there other such examples?
Look for more applications of the Main Theorem.}
the \emph{trivial $\lambda$-frame} to be introduced in Section~\ref{section:trivial+*domination}
is an important example of a pre-$\lambda$-frame in which the extension and uniqueness axioms may fail.
The weak hypotheses in our Main Theorem expand its versatility to apply to such examples.






\begin{definition}[cf.~{\cite[Definition~3.2.1(1)]{JrSh:875}}]
Suppose $\mathfrak s = (\AECK, {\dnf}, S^\bs)$ is a pre-$\lambda$-frame.
We say $\AECK^{3,\uq}$ is \emph{dense with respect to $\leqbs$}
\todo{Apply this definition throughout.}
(or \emph{the class of uniqueness triples is dense among the basic triples}) if
for every $(A, B, b) \in \AECK^{3,\bs}$ there is some $(C, D, b) \in \AECK^{3,\uq}$
such that $(A, B, b) \leqbs (C, D, b)$.
\end{definition}

\begin{theorem}[Main Theorem]\label{main-thm}
Suppose that:
\todo{Add comments about:  
it's no big deal to assume AP in $\AECK_{\lambda^+}$, which allows us to avoid extension;
getting the saturated and universal models from the axioms on $\AECK_\lambda$ and $\AECK_{\lambda^+}$ rather than positing their existence.}
\begin{enumerate}
\item $\lambda$ is an infinite cardinal such that $\diamondsuit(\lambda^+)$ holds;
\item $\mathfrak s = (\AECK, {\dnf}, S^\bs)$ is a pre-$\lambda$-frame satisfying 
amalgamation, no maximal model, 
density, transitivity,
existence, and continuity;
\todo{Can we remove ``no maximal model''?}
\item $\AECK_{\lambda^+}$ satisfies amalgamation and stability.
\end{enumerate}
Then $\AECK^{3,\uq}$ is dense with respect to $\leqbs$.

\end{theorem}

Before giving a formal proof, we describe its argument.
Given a basic triple $(A, B, b)$,
we begin by fixing a model $N$ extending $A$ that is saturated in $\lambda^+$ over $\lambda$,
and a model $N'$ that is universal over $N$.
Then we attempt to construct a $\lambda^+$-length sequence of basic triples $(C_\alpha, B_\alpha, b)$,
each non-forking over $A$,
with the goal of encountering a uniqueness triple somewhere along the way, 
at which point we are done.
We consider $\Omega_\alpha$ (provided by the $\diamondsuit$ sequence)
as a prediction of an embedding of $B_\alpha$ into $N'$,
which determines an amalgamation of $B_\alpha$ and $N$ over
(some submodel $A_\alpha$ of $N$ that is isomorphic to) $C_\alpha$.
At any stage where we have not encountered a uniqueness triple,
we arrange things so that the amalgamation of $B_\alpha$ and $N$ over $A_\alpha$ that we choose 
will be different from the predicted amalgamation.
If we manage to complete the construction up to $\lambda^+$,
then by the diamond principle, the prediction will be realized on a stationary subset $W$ of $\lambda^+$,
so that we get a uniqueness triple for each $\alpha \in W$.

\begin{proof}[Proof of Theorem~\ref{main-thm}]
Consider arbitrary $(A, B, b) \in \AECK^{3,\bs}$.
Our goal is to find $C, D \in \AECK_\lambda$ such that
$(C, D, b) \in \AECK^{3,\uq}$ and $(A, B, b) \leqbs (C, D, b)$.

Step 1:

By $\diamondsuit(\lambda^+)$ we have $2^\lambda = \lambda^+$, 
so that by Corollary~\ref{almost-stability-from-CH}, $\AECK_\lambda$ satisfies almost stability.
Then, as the pre-$\lambda$-frame $\mathfrak s$ satisfies amalgamation, no maximal model, 
and density, 
and $A \in \AECK_\lambda$,
we apply Corllary~\ref{equiv-basic-AS} 
to obtain
$N \in \AECK^\sat_{\lambda^+}$ satisfying $A \lessK N$.

As $\AECK_{\lambda^+}$ satisfies amalgamation and stability,
by Theorem~\ref{stab+amalg<->universal} we obtain 
$N' \in \AECK_{\lambda^+}$ such that $N \lessKuniv N'$.


\todo{Here is where the proof of Theorem \ref{no-extension-old} would start.}
Without loss of generality we may assume that
the respective universes of $A$, $B$, $N$, and $N'$ are all subsets of $\lambda^+$.


By $\diamondsuit(\lambda^+)$ and Fact~\ref{diamond-H_kappa-equiv}, 
fix a sequence $\left< \Omega_\alpha \mid \alpha<\lambda^+ \right>$ witnessing $\diamondsuit^-(\mathbf H_{\lambda^+})$.


The following definition is justified by Fact~\ref{predicted-is-amalgamation}:

\begin{definition}[$\alpha$-level predicted amalgamation]\label{def-predicted-amalg-variant}
Given an ordinal $\alpha<\lambda^+$,
models $A_\alpha, B_\alpha \in \AECK_\lambda$ such that $A_\alpha \lessK N$,
and an embedding $g_\alpha : A_\alpha \embeds B_\alpha$,
if $\Omega_\alpha$ is an embedding from $B_\alpha$ into $N'$ such that $\Omega_\alpha \circ g_\alpha = \id_{A_\alpha}$,
then:
\begin{itemize}
\item We say ``\emph{the $\alpha$-level predicted amalgamation exists}'';
\item We refer to $(\Omega_\alpha, \id_N, N')$ as the 
\emph{$\alpha$-level predicted amalgamation};
and
\item For every $M \in \AECK_\lambda$ such that
$A_\alpha \lessK M \lessK N$,
\todo{Should we define it this way at all?  Maybe restrict this to $A_{\alpha+1}$ instead of arbitrary $M$?}
we refer to $(\Omega_\alpha, \id_{M}, N')$ as the
\emph{$\alpha$-level predicted amalgamation of $B_\alpha$ and $M$ over
$(A_\alpha, g_\alpha, \id_{A_\alpha})$}.
\end{itemize}
\end{definition}

Step 2:

We shall attempt to build, 
recursively over $\alpha\leq\lambda^+$,
a sequence
\[
\Sequence{ (A_\alpha, B_\alpha, g_\alpha) | \alpha\leq\lambda^+ }
\]
such that:
\begin{enumerate}
\item $\left< A_\alpha \mid \alpha\leq\lambda^+ \right>$ 
is a $\lessK$-increasing, continuous sequence of models,
with $A_0 = A$;

\item $\left< B_\alpha \mid \alpha\leq\lambda^+ \right>$
is a $\leqK$-increasing, continuous sequence of models,
with $B_0 = B$;

\item $\langle g_\alpha : A_\alpha \embeds B_\alpha \mid\allowbreak \alpha\leq\lambda^+ \rangle$
is a $\subseteq$-increasing, continuous sequence of embeddings,
with $g_0 = \id_A$;
\setcounter{condition}{\value{enumi}}
\end{enumerate}
and satisfying the following for all $\alpha \leq\lambda^+$:
\begin{enumerate}
\setcounter{enumi}{\value{condition}}
\item $A_\alpha, B_\alpha \in \AECK_\lambda$ if $\alpha<\lambda^+$; 
\item The universes of $A_\alpha$ and $B_\alpha$ are subsets of $\lambda^+$;
\item $A_\alpha \leqK N$;
\todo{NOTE  WE ARE NOT ENSURING 
$N \cap \alpha$ is a subset of $A_\alpha$,
so we will not obtain $A_{\lambda^+} = N$ at the end.}
\item \label{b-in-difference-variant} 
\todo{Ensure references to here are correct.}
$b \in B_\alpha \setminus g_\alpha[A_\alpha]$; 
\item $
\tp(b/g_\alpha[A_\alpha]; B_\alpha)$
does not fork over $g_0[A_0]$; 

\item If $\alpha<\lambda^+$, $(g_\alpha[A_\alpha], B_\alpha, b) \notin \AECK^{3,\uq}$,
and the $\alpha$-level predicted amalgamation exists,
then $(\id_{B_\alpha}, g_{\alpha+1}, B_{\alpha+1}) \centernot{E} (\Omega_\alpha, \id_{A_{\alpha+1}}, N')$
\todo{Consider writing the contrapositive?}
(where $(\Omega_\alpha, \id_{A_{\alpha+1}}, N')$ is the 
$\alpha$-level predicted amalgamation of 
$B_\alpha$ and $A_{\alpha+1}$ over $(A_\alpha, g_\alpha, \id_{A_\alpha})$, 
as in Definition~\ref{def-predicted-amalg-variant};
and $\centernot{E}$ is the negation of the relation $E$ between amalgamations of 
$B_\alpha$ and $A_{\alpha+1}$ over $(A_\alpha, g_\alpha, \id_{A_\alpha})$, as in Definition~\ref{def-E}).

\end{enumerate}

\begin{notation}
Given an ordinal $\alpha\leq\lambda^+$ and $(A_\alpha, B_\alpha, g_\alpha)$ satisfying
Clauses (1), (2), and~(\ref{b-in-difference-variant}) above,
we write $q_\alpha := \tp(b/g_\alpha[A_\alpha]; B_\alpha)$.
\end{notation}

We now carry out the recursive construction:
\begin{itemize}
\item For $\alpha=0$:
Set $A_0 := A$, $B_0 := B$, and $g_0 := \id_A$.
As $(A, B, b) \in \AECK^{3,\bs}$,
the existence property guarantees that $\tp(b/A; B)$ does not fork over $A$,
and all of the requirements are satisfied.

\item For a nonzero limit ordinal $\alpha\leq\lambda^+$:
Define
\[
A_\alpha := \bigcup_{\gamma<\alpha} A_\gamma; \quad
B_\alpha := \bigcup_{\gamma<\alpha} B_\gamma; \quad
g_\alpha := \bigcup_{\gamma<\alpha} g_\gamma.
\]
Using the continuity property, Lemma~\ref{continuous-union} guarantees that all of the requirements are satisfied.

\item For $\beta = \alpha+1$, we consider two possibilities:
\begin{itemize}
\item
First, suppose $(g_\alpha[A_\alpha], B_\alpha, b) \in \AECK^{3,\uq}$.
In this case, we terminate the recursive construction here and
set $C := g_\alpha[A_\alpha]$ and $D := B_\alpha$.
From the induction hypothesis, we obtain $A = A_0 = g_0[A_0] \leqK g_\alpha[A_\alpha] = C$, $B = B_0 \leqK B_\alpha = D$,
and $\tp(b/C; D) = \tp(b/g_\alpha[A_\alpha]; B_\alpha)$ does not fork over $g_0[A_0] = A$,
so that $C$ and $D$ satisfy the conclusion of the Theorem in this case.

\item
Thus, from now on, we assume that
$(g_\alpha[A_\alpha], B_\alpha, b) \notin \AECK^{3,\uq}$.
By the induction hypothesis, $q_\alpha \in S^\bs(g_\alpha[A_\alpha])$,
so that $(g_\alpha[A_\alpha], B_\alpha, b) \in \AECK^{3,\bs} \setminus \AECK^{3,\uq}$.
Viewing $g_\alpha : A_\alpha \cong g_\alpha[A_\alpha]$ as an isomorphism,
since $N \in \AECK^\sat_{\lambda^+}$,
we apply Lemma~\ref{uniqueness-saturated} to obtain
$A_{\alpha+1} \in \AECK_\lambda$ witnessing the non-uniqueness
of $(g_\alpha[A_\alpha], B_\alpha, b)$ via~$g_\alpha$, with $A_\alpha \lessK A_{\alpha+1} \lessK N$.
\todo[color=green]{Updated on Sept.~6}
We now apply Lemma~\ref{get-desired-amalg}
with 
\[
\left( A, B, A^*, C^*, N^\bullet, \varphi \right) := 
\left( g_\alpha[A_\alpha], B_\alpha, A_\alpha, A_{\alpha+1}, N', g_\alpha \right).
\]
If the $\alpha$-level predicted amalgamation exists, then let $\Omega := \Omega_\alpha$;
otherwise $\Omega$ can be chosen arbitrarily (since $A_\alpha \lessK N \lessK N'$ and $N \in \AECK^\sat_{\lambda^+}$)
but is actually irrelevant.
We thus obtain
a $\lambda$-amalgamation
$(\id_{B_\alpha}, g_{\alpha+1}, B_{\alpha+1})$ of $B_\alpha$ and $A_{\alpha+1}$ over $(A_\alpha, g_\alpha, \id_{A_\alpha})$ such that
$q_{\alpha+1} = \tp(b/g_{\alpha+1}[A_{\alpha+1}]; B_{\alpha+1})$ does not fork over $g_\alpha[A_\alpha]$,
and 
such that Clause~(9) of the recursive construction is satisfied.
In particular, 
$B_{\alpha+1} \in \AECK_\lambda$, $B_\alpha \leqK B_{\alpha+1}$, 
$g_{\alpha+1} : A_{\alpha+1} \embeds B_{\alpha+1}$ is an embedding such that $g_\alpha \subseteq g_{\alpha+1}$,
$b \in B_{\alpha+1} \setminus g_{\alpha+1}[A_{\alpha+1}]$,
and $q_{\alpha+1} \in S^\bs(g_{\alpha+1}[A_{\alpha+1}])$.
By the induction hypothesis, we know that $q_\alpha$ does not fork over $g_0[A_0]$,
so that by Lemma~\ref{restriction} and the transitivity property
it follows that $q_{\alpha+1}$ does not fork over $g_0[A_0]$.
We can easily ensure that the universe of $B_{\alpha+1}$ is a subset of $\lambda^+$.

\end{itemize}
\end{itemize}


Step 3:
Suppose that we have managed to complete the recursive construction as described above.

Let $N^- := A_{\lambda^+}$ and 
$N^\bullet := B_{\lambda^+}$.
Since $N^-$ is the union of a $\lambda^+$-length $\lessK$-increasing sequence of models each of cardinality $\lambda$,
it follows that $N^- \in \AECK_{\lambda^+}$.
As $\AECK_{\lambda^+}$ satisfies the amalgamation property, 
in particular $N^- \in \AECK_{\lambda^+}$ is an amalgamation base (Fact~\ref{AP=AB}).
Since $N^- \lessK N$ by Clause~(6) of the recursive construction, 
we thus obtain from Corollary~\ref{smaller-univ-from-AB} that $N^- \lessKuniv N'$.

Since $N^\bullet = \bigcup_{\alpha<\lambda^+} B_\alpha$, 
where $\left|B_\alpha\right| = \lambda$ for every $\alpha<\lambda^+$,
it follows that $\left| N^\bullet \right| \leq \lambda^+$.
But $g_{\lambda^+} : N^- \embeds N^\bullet$ is an embedding, where $\left|N^-\right| = \lambda^+$,
so that in fact $N^\bullet \in \AECK_{\lambda^+}$.


Then, since $N^- \lessKuniv N'$, 
apply Lemma~\ref{univ-embedding} to obtain
an embedding $h : N^\bullet \embeds N'$ such that $h \circ g_{\lambda^+} = \id_{N^-}$.
\todo{Do we need to define $c := h(b)$ and $C_\alpha := h[B_\alpha]$?}
For every $\alpha\leq\lambda^+$, let 
$h_\alpha := h \restriction B_\alpha$,
so that $h = h_{\lambda^+} = \bigcup_{\alpha<\lambda^+} h_\alpha$.

\begin{claim}\label{claim-after-recursion-variant}
For every $\alpha<\lambda^+$:
\begin{enumerate}
\item $h_\alpha \circ g_\alpha = \id_{A_\alpha}$.
\item $A_\alpha \leqK h_\alpha[B_\alpha] \lessK N'$.
\todo{Do we need this?}
\item
$(\id_{B_\alpha}, g_{\alpha+1}, B_{\alpha+1})$ and $(h_\alpha, \id_{A_{\alpha+1}}, N')$
are two amalgamations of $B_\alpha$ and $A_{\alpha+1}$ over $(A_\alpha, g_\alpha, \id_{A_\alpha})$.
\item
$(\id_{B_\alpha}, g_{\alpha+1}, B_{\alpha+1}) \mathrel{E} (h_\alpha, \id_{A_{\alpha+1}}, N')$.
\item If $\Omega_\alpha = h_\alpha$ then $(g_\alpha[A_\alpha], B_\alpha, b) \in \AECK^{3,\uq}$.
\end{enumerate}
\end{claim}

\begin{proof}\hfill
\begin{enumerate}
\item Since $g_\alpha \subseteq g_{\lambda^+}$, $h_\alpha \subseteq h$, and $h \circ g_{\lambda^+} = \id_{A_{\lambda^+}}$.
\item $h_\alpha = h \restriction B_\alpha$ is an embedding from $B_\alpha$ into $N'$, and by the previous clause,
$A_\alpha = h_\alpha \circ g_\alpha[A_\alpha] \leqK h_\alpha[B_\alpha] \lessK N'$.



\item
The fact that 
$(\id_{B_\alpha}, g_{\alpha+1}, B_{\alpha+1})$ 
is an amalgamation of $B_\alpha$ and $A_{\alpha+1}$ over $(A_\alpha, g_\alpha, \id_{A_\alpha})$ 
follows from $g_\alpha \subseteq g_{\alpha+1}$.
The fact that 
$(h_\alpha, \id_{A_{\alpha+1}}, N')$
is an amalgamation of $B_\alpha$ and $A_{\alpha+1}$ over $(A_\alpha, g_\alpha, \id_{A_\alpha})$ 
follows from $h_\alpha \circ g_\alpha = \id_{A_\alpha}$.

\item $h_{\alpha+1} : B_{\alpha+1} \embeds N'$ is an embedding satisfying
$h_{\alpha+1} \circ \id_{B_\alpha} = h_\alpha$ and
$h_{\alpha+1} \circ g_{\alpha+1} = \id_{A_{\alpha+1}}$,
so that $(h_{\alpha+1}, \id_{N'}, N')$ witnesses that
$(\id_{B_\alpha}, g_{\alpha+1}, B_{\alpha+1}) \mathrel{E} (h_\alpha, \id_{A_{\alpha+1}}, N')$,
as sought.

\item
Suppose $\Omega_\alpha = h_\alpha$.
In particular, $\Omega_\alpha : B_\alpha \embeds N'$ is an embedding satisfying
$\Omega_\alpha \circ g_\alpha = h_\alpha \circ g_\alpha = \id_{A_\alpha}$,
meaning that the $\alpha$-level predicted amalgamation exists, and
$(\Omega_\alpha, \id_{A_{\alpha+1}}, N')$
is the $\alpha$-level predicted amalgamation of 
$B_\alpha$ and $A_{\alpha+1}$ over $(A_\alpha, g_\alpha, \id_{A_\alpha})$.
Furthermore, by Clause~(4) above we obtain
$(\id_{B_\alpha}, g_{\alpha+1}, B_{\alpha+1}) \mathrel{E} (\Omega_\alpha, \id_{A_{\alpha+1}}, N')$.
Thus, by Clause~(9) of the recursion, it must be that $(g_\alpha[A_\alpha], B_\alpha, b) \in \AECK^{3,\uq}$.
\qedhere
\end{enumerate}
\end{proof}

\begin{claim}\label{get-stationary-beta}
There are stationarily many $\beta<\lambda^+$ such that $\Omega_\beta = h_\beta$.
\todo{We really only need one!}
\end{claim}

\begin{proof}
Define the function $\rank_h$ as follows:
For any ordered pair $(x,y) \in h$,
let $\rank_h(x,y)$ be the smallest ordinal $\alpha$ such that $(x,y) \in h_\alpha$.

Let $\Omega := h$ and $z := \{ \langle h_\alpha \mid \alpha\leq\lambda^+ \rangle, {\rank_h} \}$.
\todo{Insert more parameters into $z$ if necessary.}
The respective universes of $N^\bullet$ and $N'$ are subsets of $\lambda^+$, so that
clearly $\Omega \subseteq \mathbf H_{\lambda^+}$ and $z \in \mathbf H_{\lambda^{++}}$.
Thus, by Fact~\ref{diamond-H_kappa-stat} and our choice of the sequence $\left< \Omega_\alpha \mid \alpha<\lambda^+ \right>$,
the following set is stationary in $\lambda^+$:
\[
W := \Set{ \beta<\lambda^+ | \exists \mathcal M \esm \mathbf H_{\lambda^{++}} [ z \in \mathcal M, 
\beta = \mathcal M \cap \lambda^+, \mathcal M \cap \Omega = \Omega_\beta ] }
\]
Consider any $\beta \in W$, and we must show that 
$\Omega_\beta = h_\beta$.

Let $\mathcal M$ witness that $\beta \in W$.
That is, $\mathcal M \esm \mathbf H_{\lambda^{++}}$, $z \in \mathcal M$,
$\beta = \mathcal M \cap \lambda^+$, and $\mathcal M \cap \Omega = \Omega_\beta$.
In particular, $\beta$ is a nonzero limit ordinal $<\lambda^+$.

For all $\alpha<\beta$, by $\alpha, \langle h_\alpha \mid \alpha\leq\lambda^+ \rangle \in \mathcal M$,
it follows that $h_\alpha \in \mathcal M$,
and by $\mathcal M \models |h_\alpha| < \lambda^+$ we have $h_\alpha \subseteq \mathcal M$.
Since $\beta$ is a nonzero limit ordinal, it follows that
$h_\beta = \bigcup_{\alpha<\beta} h_\alpha \subseteq \mathcal M \cap \Omega = \Omega_\beta$.

Conversely, for any ordered pair $(x,y) \in \Omega_\beta = \mathcal M \cap \Omega$,
we have $\rank_h(x,y) \in \mathcal M$, so that $\rank_h(x,y)$ is an ordinal $<\beta$,
and it follows that $(x,y) \in h_\beta$.
\end{proof}

Fix some $\beta<\lambda^+$ such that $\Omega_\beta = h_\beta$,
and set $C := g_\beta[A_\beta]$ and $D := B_\beta$.
By Clause~(4) of the recursive construction we have $C, D \in \AECK_\lambda$.
Since $\Omega_\beta = h_\beta$,
Claim~\ref{claim-after-recursion-variant}(5) gives
$(C, D, b) = (g_\beta[A_\beta], B_\beta, b) \in \AECK^{3,\uq}$.
By Clauses (1) and~(3) of the recursive construction, we have $A = A_0 = g_0[A_0] \leqK g_\beta[A_\beta] = C$.
By Clause~(2) we have $B = B_0 \leqK B_\beta = D$.
Finally, by Clause~(8) we obtain
$
\tp(b/C; D) = \tp(b/g_\beta[A_\beta]; B_\beta)$
does not fork over $g_0[A_0] = A$, 
completing the proof of the Theorem.
\end{proof}

\section{*domination triples and the trivial frame}
\label{section:trivial+*domination}

Uniqueness triples were defined with respect to a given pre-$\lambda$-frame.
However, given any AEC $\AECK$, 
there is a natural form of \emph{domination} by which,
given two models $A \lessK B$ both in $\AECK$,
we can think of an element $b \in B \setminus A$ as representing $B$ as an extension over $A$.
\todo{From 30 Shevat.
Is this intuition correct?
Formalize this?  Does it help here?  
In \cite[\S4.1]{JrSh:875} we find:  ``The element $b$ represents the extension $B$ over $A$.''
But that is referring to uniqueness triples, not $*$-domination triples.
Which one is correct?}
We explore this concept here.


Throughout this section, we assume that $\AECK_\lambda$ satisfies the amalgamation property,
so that Galois-types are ``best-behaved'' \cite[p.~64]{Baldwin-Categoricity} and equivalence of $\lambda$-amalgamations is well-defined (see Section~\ref{section:equivalence}).

\begin{definition}
Suppose that $\AECK$ is an AEC, $A, B, C \in \AECK_\lambda$ with $A \lessK B$ and $A \leqK C$,
and $b \in B \setminus A$.
We say that \emph{$C$ witnesses the non-$*$-domination of $(A, B, b)$} if there exist two $\lambda$-amalgamations
$(f^D_B, \id_C, D)$ and $(f^F_B, \id_C, F)$ of $B$ and $C$ over $A$ that are not equivalent over $A$, 
but such that 
$\tp(f^D_B(b)/C; D) = \tp(f^F_B(b)/C; F) \in S^\na(C)$.

\end{definition}

\begin{definition}
\todo{Show that relevant classes of triples are downward closed.}
Suppose $A, B \in \AECK_\lambda$ with $A \lessK B$, and $b \in B \setminus A$.
We say that \emph{$b$ $*$-dominates $B$ over $A$} and that $(A, B, b)$ is a \emph{$*$-domination triple in $\AECK_\lambda$} if
there is no $C \in \AECK_\lambda$ witnessing the non-$*$-domination of $(A, B, b)$.
The class of $*$-domination triples in $\AECK_\lambda$ is denoted $\AECK^{3, \dom}_\lambda$.
\todo{$\AECK^{3, \dom}_\lambda$ doesn't seem to conflict with any prior use.}
\end{definition}

\begin{remark}
We use the * symbol in $*$-dominates and $*$-domination in order to avoid conflict
with the definition of \emph{$a$ dominates $N$ over $M$} given in~\cite[Definition~4.1.6]{JrSh:875} and~\cite[Definition~6.8]{JrSitton},
as well as with the definition of \emph{domination triple} given in \cite[Definition~3.4]{Vasey14}.
\end{remark}

\begin{remark}
If $(A, B, b) \in \AECK^{3, \dom}_\lambda$ and $C \in \AECK_\lambda$ with $A \leqK C$,
there may or may not exist any $\lambda$-amalgamation $(f^D_B, \id_C, D)$ of $B$ and $C$ over $A$ such that
$f^D_B(b) \notin C$.
\end{remark}

In order to obtain the density of $*$-domination triples from our main result involving uniqueness triples, 
we introduce the trivial $\lambda$-frame:



\begin{definition}[{\cite[Definition~2.2.2]{JrSh:875}}]\label{trivial-frame-def}
Consider a given AEC $\AECK$ and any $\lambda \geq \LST(\AECK)$ such that $\AECK_\lambda \neq\emptyset$.
Define 
\[
{\dnf} := \Set{ (A, B, c, C) | A, B, C \in \AECK_\lambda, A \leqK B \lessK C, c \in C \setminus B }
\]
(that is, \emph{every non-algebraic extension does not fork})
and $S^\bs := S^\na \restriction \AECK_\lambda$.
Then $(\AECK, {\dnf}, S^\bs)$ is called the \emph{trivial $\lambda$-frame of $\AECK$}.
\end{definition}

\begin{fact}[cf.~{\cite[Proposition~2.2.3]{JrSh:875}}]\label{properties-of-trivial}
Suppose $\AECK$ is an AEC and $\lambda \geq \LST(\AECK)$ is such that $\AECK_\lambda \neq\emptyset$.
Then:
\todo{Verify all properties here.}
\begin{enumerate}
\item The trivial $\lambda$-frame of $\AECK$ is a type-full pre-$\lambda$-frame
satisfying density, transitivity, existence, local character, and continuity.
\item
If $\AECK_\lambda$ 
satisfies the \emph{disjoint amalgamation property}%
\footnote{This property is known in Fra\"\i ss\'e theory as the \emph{strong amalgamation property},
but of course the name \emph{disjoint amalgamation property} used in abstract elementary classes is more descriptive.
Notice that \cite[Exercise~5.7]{Sh:838} provides an example of a good $\lambda$-frame satisfying AP but not DAP.}
(DAP),
then the trivial $\lambda$-frame satisfies extension.
\todo{Is the converse true?}
\item
The trivial $\lambda$-frame satisfies basic stability (respectively, basic almost stability) iff the class $\AECK_\lambda$ satisfies stability (respectively, almost stability).
\todo{Move definitions of stability and almost stability to be before this.}
\end{enumerate}
\end{fact}

\begin{remark}\label{trivial-no-uniqueness}
The trivial $\lambda$-frame usually does not satisfy uniqueness.
\end{remark}

\begin{remark}
The trivial $\lambda$-frame is a special case of Examples~\ref{examples-pre-frames}((3)\&(4)) above.
\end{remark}

The following example (inspired by~\cite[Exercise~5.7]{Sh:838}) shows that
the amalgamation property is not sufficient to guarantee that the trivial $\lambda$-frame satisfies the extension property:

\begin{example}\label{trivial-no-extension}
Let the vocabulary $\tau$ consist of a single unary-relation symbol $U$.
For any $\tau$-model $M$, 
we consider $U^M$ (the interpretation of $U$ in $M$) as a subset of $M$.
Let $\AECK$ be the class of all $\tau$-models $M$ such that $\left|U^M\right| \leq1$,
and let $\leqK$ be the submodel relation $\subseteq$.

It is clear that $\AECK$ is an AEC with $\LST(\AECK) = \aleph_0$
(in fact every subset is a submodel),
and that $\AECK$ satisfies amalgamation (cf.\ the hint in~\cite[Exercise~5.7]{Sh:838}).
Fix any infinite cardinal $\lambda$, 
and we show that the trivial $\lambda$-frame of $\AECK$ does not satisfy the extension property
(so that $\AECK_\lambda$ does not satisfy the DAP),
as follows:

Fix models $A, B, C \in \AECK_\lambda$ such that $A \lessK B$, $A \lessK C$,
$U^A = \emptyset$, $U^B = \{b\}$, and $U^C = \{c\}$.
Consider the type $p := (b/A; B)$, which is in $S^\bs(A)$.
Suppose $q \in S(C)$ extends $p$.
Then we can fix an amalgamation $(f, \id_C, D)$ of $B$ and $C$ over $A$ such that
$q = \tp(f(b)/C; D)$ (see Lemma~\ref{amalg-ext-with-ambient}).
Since $f : B \embeds D$ is an embedding, $U^B(b)$ implies $U^D(f(b))$.
Since $C \leqK D$, $U^C(c)$ implies $U^D(c)$.
But $D \in \AECK$, so that $\left|U^D\right| \leq1$,
meaning that $f(b) = c \in C$,
and it follows that $q = \tp(c/C; D) \notin S^\na(C)$.
\end{example}


The trivial $\lambda$-frame provides the crucial link between $*$-domination triples and uniqueness triples:


\begin{fact}\label{dom=uq-in-trivial}
Suppose $\AECK$ is an AEC and $\lambda \geq \LST(\AECK)$ is such that $\AECK_\lambda \neq\emptyset$.
Let $\mathfrak s = (\AECK, {\dnf}, S^\bs)$ be the trivial $\lambda$-frame of $\AECK$.
Then the $*$-domination triples in $\AECK_\lambda$ 
are exactly the uniqueness triples in $\mathfrak s$;
that is, $\AECK^{3,\uq} = \AECK^{3,\dom}_\lambda$.
\end{fact}

We are now able to prove the density of $*$-domination triples from the diamond axiom.
In doing so, we see the importance of the weak set of hypotheses on the pre-$\lambda$-frame
in the statement of the Main Theorem (Theorem~\ref{main-thm}).
In particular, the fact that the Main Theorem does not require the extension or uniqueness properties
is crucial, in light of Example~\ref{trivial-no-extension} and Remark~\ref{trivial-no-uniqueness} above.

\begin{corollary}\label{*domination-triples-dense}
Suppose that:
\begin{enumerate}
\item $\lambda$ is an infinite cardinal such that $\diamondsuit(\lambda^+)$ holds;
\item $\LST(\AECK) \leq\lambda$, and $\AECK_\lambda$ satisfies 
amalgamation and no maximal model; 
\item $\AECK_{\lambda^+}$ satisfies amalgamation and stability;
\item $A, B \in \AECK_\lambda$ satisfy $A \lessK B$, and $b \in B \setminus A$.  
\todo{Should we define $(A, B, b) \in \AECK^{3,\na}_\lambda$?}
\end{enumerate}
Then there exist models $C, D \in \AECK_\lambda$ such that:
\begin{enumerate}
\item $(C, D, b) \in \AECK^{3,\dom}_\lambda$; 
and
\item $A \leqK C$ and $B \leqK D$ (so that, in particular, $\tp(b/C; D)$ extends $\tp(b/A; B)$).
\end{enumerate}
\end{corollary}

\begin{proof}
As $\LST(\AECK) \leq\lambda$ and clearly $\AECK_\lambda \neq\emptyset$,
we may define the trivial $\lambda$-frame $\mathfrak s = (\AECK, {\dnf}, S^\bs)$ of $\AECK$ (Definition~\ref{trivial-frame-def}).
By Fact~\ref{properties-of-trivial}(1), the trivial $\lambda$-frame is a pre-$\lambda$-frame satisfying
density, transitivity, existence, and continuity.
Furthermore, we have assumed that $\AECK_\lambda$ satisfies 
amalgamation and no maximal model, 
and $\AECK_{\lambda^+}$ satisfies amalgamation and stability.
Finally, $\tp(b/A;B) \in S^\na(A) = S^\bs(A)$.
Thus we may apply Theorem~\ref{main-thm} to obtain $C, D \in \AECK_\lambda$ such that
$(C, D, b) \in \AECK^{3,\uq}$,
$A \leqK C$, and $B \leqK D$.
But $\AECK^{3,\uq} = \AECK^{3,\dom}_\lambda$ in the trivial $\lambda$-frame 
by Fact~\ref{dom=uq-in-trivial},
and we are done.
\todo{Possibly add more corollaries assuming DAP so that we get extension and can apply the other theorems?}
\end{proof}

\section{Non-splitting}
\label{section:splitting}

In this section,
we apply our Main Theorem to 
the \emph{non-splitting relation},
obtaining the density of uniqueness triples assuming the transitivity and continuity properties only.

Recall:

\begin{definition}[cf.~{\cite[Definition~2.4]{MR2225893}}, {\cite[Definition~3.2]{MR3471143}}]
For $\lambda \geq\LST(\AECK)$, $A, B \in \AECK$ with $A \leqK B$, and $p \in S^\na(B)$, 
we say that \emph{$p$ does not $\lambda$-split over $A$} if:
For all $M_1, M_2 \in \AECK_\lambda$ and every isomorphism $f:M_1 \cong M_2$, 
if $A \leqK M_1 \leqK B$, $A \leqK M_2 \leqK B$, and $f \restriction A = \id_A$, 
then $f(p \restriction M_1)=p \restriction M_2$.
\end{definition}

Let $\dnsplit$ denote the non-$\lambda$-splitting relation.

\begin{fact}[cf.~{\cite[Remark~3.3]{MR3471143}}]\label{non-splitting-pre-frame}
Suppose $\lambda \geq \LST(\AECK)$ is such that $\AECK_\lambda \neq\emptyset$.
Then $\mathfrak s_{\ns}=(\AECK, {\dnsplit}, {S^{\na} \restriction \AECK_\lambda})$ is a type-full pre-$\lambda$-frame
satisfying density and existence.
\end{fact}

The extension, uniqueness and transitivity properties for $\mathfrak s_{\ns}$ are very problematic. 
This was a central difficulty addressed, for example, by Grossberg and VanDieren in~\cite{MR2225893}. 
Vasey proved a limited version of extension for non-splitting in~\cite[Lemma~4.11]{Vasey33}. 
Since our Main Theorem does not require the extension and uniqueness properties, we avoid the need to deal with those two properties. 
However we do assume transitivity and continuity.

Let $\AECK^{3,\uq}_{\ns}$ denote the classs of uniqueness triples in the pre-$\lambda$-frame $\mathfrak s_{\ns}$.

\begin{corollary}\label{non-splitting-cor}
Suppose that:
\begin{enumerate}
\item $\lambda \geq \LST(\AECK)$ is such that $\AECK_\lambda \neq\emptyset$, and $\diamondsuit(\lambda^+)$ holds;
\item $\mathfrak s_{\ns}$ satisfies the transitivity and continuity properties;
\item $\AECK_\lambda$ satisfies the amalgamation property and has no maximal models;
\item $\AECK_{\lambda^+}$ satisfies amalgamation and stability.
\end{enumerate}
Then $\AECK^{3,\uq}_{\ns}$ is dense with respect to non-splitting. 
\end{corollary}

\begin{proof}
We have already seen (Fact~\ref{non-splitting-pre-frame}) that
$\mathfrak s_{\ns}$ is a pre-$\lambda$-frame
satisfying density and existence.
Thus the result follows immediately from Theorem~\ref{main-thm}.
\end{proof}

We now compare our approach toward proving Corollary~\ref{non-splitting-cor}
with Vasey's approach toward obtaining the existence of uniqueness triples in~\cite[Corollary~6.5]{Vasey31}.
Vasey assumes superstability and shifts the non-splitting relation by a universal extension, 
in order to obtain a good $\lambda$-frame 
(in particular, satisfying the stability and extension axioms used in the proof of \cite[Fact~6.4]{Vasey31})
and thereby derive that it is weakly successful, that is, satisfies the existence of uniqueness triples.
In contrast, because our Main Theorem (Theorem~\ref{main-thm}) does not require extension and stability in the pre-$\lambda$-frame,
we are able to obtain the density of uniqueness triples with respect to the non-splitting relation itself,
and without assuming superstability.

%
%
%
%
%
%
%
%

\begin{remark}
The \emph{coheir relation} was generalized to the context of AECs by Boney and Grossberg in~\cite{MR3650352} 
and further examined by Vasey in~\cite[Theorem~5.15]{MR3490921}.
In both papers, the extension property of coheir is not proven.
Therefore we believe the coheir relation to be another appropriate application of our Main Theorem.
\end{remark}


\section{Obtaining representations with many uniqueness triples extending a given triple}
\label{section:representations}

In the proof of Theorem~\ref{main-thm},
we attempted to construct $\lambda^+$-length sequences of models,
but we aborted the recursive construction upon encountering the first uniqueness triple along the way.
This raises the question:
Is there some way we can ensure the existence of an increasing $\lambda^+$-length sequence of basic triples
that contains many uniqueness triples?

The essential difficulty is that without assuming the extension property, we cannot necessarily obtain a non-forking extension at the successor stages.
By contrast, in the following Theorem, at the cost of requiring the extension property as an additional hypothesis,
we obtain representations extending the given basic triple and containing stationarily many uniqueness triples.
The extension property guarantees that we can complete the construction to obtain a sequence of length $\lambda^+$,
with non-forking extensions at every successor stage.
As in the proof of Theorem~\ref{main-thm},
the $\diamondsuit(\lambda^+)$ axiom ensures that
we obtain stationarily many uniqueness triples once the construction is complete,
but unlike in that proof, here the completed construction is actually realized rather than considered hypothetically.

As an added bonus resulting from the extension axiom, 
if we begin with a given saturated model $N^\circ$ that extends $A$,
then by tweaking the construction slightly 
(see Clause~\ref{tweak-from-extension} of the recursive construction in the proof of Theorem~\ref{repres-uq-extend-triple} below),
we can ensure that the sequence $\langle A_\alpha \mid \alpha<\lambda^+ \rangle$ of models that we construct converges to some model $N$ isomorphic to $N^\circ$.

\begin{theorem}\label{repres-uq-extend-triple}
Suppose that:
\begin{enumerate}
\item $\lambda$ is an infinite cardinal such that $\diamondsuit(\lambda^+)$ holds;
\item $\mathfrak s = (\AECK, {\dnf}, S^\bs)$ is a pre-$\lambda$-frame satisfying amalgamation, transitivity, 
extension, and continuity;
\item 
$(A, B, b) \in \AECK^{3,\bs}$;
\item $N^\circ, N^\bullet \in \AECK_{\lambda^+}$ satisfy $N^\circ \in \AECK^\sat_{\lambda^+}$ 
and $A \lessK N^\circ \lessKuniv N^\bullet$.
\end{enumerate}

Then there 
exist models $N, N^+ \in \AECK_{\lambda^+}$, 
representations $\langle A_\alpha \mid \alpha<\lambda^+ \rangle$ of $N$
and $\langle B_\alpha \mid \alpha<\lambda^+ \rangle$ of 
$N^+$,
and an embedding $h : N^+ \embeds N^\bullet$ 
\todo{Should we bother mentioning $h$ at all in the statement?
Yes, it's used in Corollary \ref{start-from-top}.}
such that:
\begin{enumerate}
\item $N \lessK N^+$ and $b \in N^+ \setminus N$;
\item $N \in \AECK^\sat_{\lambda^+}$;
\item $A_0 = A$ and $B_0 = B$;
\item $\tp(b/N; N^+)$ does not fork over $A$;
\item $(A_\alpha, B_\alpha, b) \in \AECK^{3,\bs}$
for every $\alpha<\lambda^+$;
\item the set
\[
\Set{ \alpha<\lambda^+ | (A_\alpha, B_\alpha, b) \in \AECK^{3,\uq} }
\]
is stationary in $\lambda^+$;
\todo{Can we express this conclusion in terms of weakly primeness triples?
We need to specify the representation in order to connect it to $A$ and $B$.}
and
\item $h \restriction N : N \cong N^\circ$ is an isomorphism and $h \restriction A = \id_A$.
\end{enumerate}

\end{theorem}

\begin{remarks}\hfill
\begin{enumerate}
\item The existence of $N^\circ$ and $N^\bullet$ as in Clause~(4) of the hypotheses of Theorem~\ref{repres-uq-extend-triple}
follows (respectively) from Clauses (2) and~(3) of the hypotheses of Theorem~\ref{main-thm}.
Thus the extension property is the only new requirement in the hypotheses of Theorem~\ref{repres-uq-extend-triple}
that was not assumed in Theorem~\ref{main-thm}.

\item In particular, it is significant that the uniqueness property is still not assumed.

\item By Lemma~\ref{big-dnf-equiv},
Clause~(4) of the conclusion of Theorem~\ref{repres-uq-extend-triple} implies that
for every $\alpha<\lambda^+$:
$\tp(b/A_\alpha; B_\alpha)$ does not fork over $A$,
so that $(A, B, b) \leqbs (A_\alpha, B_\alpha, b)$ .
Thus, the conclusion of Theorem~\ref{repres-uq-extend-triple} is clearly stronger than that of Theorem~\ref{main-thm},
in that we obtain stationarily many uniqueness triples extending the given basic triple rather than just one
(at the cost of requiring the extension axiom in the hypotheses).
%
\todo{Do these remarks deserve to become a corollary?}
\end{enumerate}
\end{remarks}

\begin{proof}[Proof of Theorem~\ref{repres-uq-extend-triple}]


Step 1:

Without loss of generality we may assume that
the respective universes of $A$, $B$, $N^\circ$, and $N^\bullet$ are all subsets of $\lambda^+$.

By $\diamondsuit(\lambda^+)$ and Fact~\ref{diamond-H_kappa-equiv}, 
fix a sequence $\left< \Omega_\alpha \mid \alpha<\lambda^+ \right>$ witnessing $\diamondsuit^-(\mathbf H_{\lambda^+})$.

The following definition is justified by Fact~\ref{predicted-is-amalgamation}:

\begin{definition}[$\alpha$-level predicted amalgamation]\label{def-predicted-amalg}
\todo{Decide whether this definition is needed here at all, or a reference to Definition \ref{def-predicted-amalg-variant} is good enough.}
Given an ordinal $\alpha<\lambda^+$,
models $C_\alpha, B_\alpha \in \AECK_\lambda$ such that $C_\alpha \lessK N^\circ$,
and an embedding $g_\alpha : C_\alpha \embeds B_\alpha$,
if $\Omega_\alpha$ is an embedding from $B_\alpha$ into $N^\bullet$ such that $\Omega_\alpha \circ g_\alpha = \id_{C_\alpha}$,
then:
\begin{itemize}
\item We say ``\emph{the $\alpha$-level predicted amalgamation exists}'';
\item We refer to $(\Omega_\alpha, \id_{N^\circ}, N^\bullet)$ as the 
\emph{$\alpha$-level predicted amalgamation};
and
\item For every $M \in \AECK_\lambda$ such that
$C_\alpha \lessK M \lessK N^\circ$,
\todo{Should we define it this way at all?  Maybe restrict this to $C_{\alpha+1}$ instead of arbitrary $M$?}
we refer to $(\Omega_\alpha, \id_{M}, N^\bullet)$ as the
\emph{$\alpha$-level predicted amalgamation of $B_\alpha$ and $M$ over
$(C_\alpha, g_\alpha, \id_{C_\alpha})$}.
\end{itemize}
\end{definition}

Step 2:

We shall build, recursively over $\alpha\leq\lambda^+$,
a sequence
\[
\Sequence{ (C_\alpha, B_\alpha, g_\alpha) | \alpha\leq\lambda^+ }
\]
such that:
\todo{Consider mentioning $\left< A_\alpha \mid \alpha\leq\lambda^+ \right>$ here to simplify subsequent statements,
rather than introducing it later.}
\begin{enumerate}
\item $\left< C_\alpha \mid \alpha\leq\lambda^+ \right>$ and $\left< B_\alpha \mid \alpha\leq\lambda^+ \right>$
are $\leqK$-increasing, continuous sequences of models,
with $C_0 = A$ and $B_0 = B$;

\item $\langle g_\alpha : C_\alpha \embeds B_\alpha \mid\allowbreak \alpha\leq\lambda^+ \rangle$
is a $\subseteq$-increasing, continuous sequence of embeddings,
with $g_0 = \id_A$;
\setcounter{condition}{\value{enumi}}
\end{enumerate}
and satisfying the following for all $\alpha \leq\lambda^+$:
\begin{enumerate}
\setcounter{enumi}{\value{condition}}
\item $C_\alpha, B_\alpha \in \AECK_\lambda$ if $\alpha<\lambda^+$; 
\item The universes of $C_\alpha$ and $B_\alpha$ are subsets of $\lambda^+$;
\item $C_\alpha \leqK N^\circ$;
\item \label{tweak-from-extension}
$N^\circ \cap \alpha$ is a subset of $C_\alpha$; 
\item \label{b-in-difference} 
\todo{Ensure references to here are correct.}
$b \in B_\alpha \setminus g_\alpha[C_\alpha]$; 
\item $
\tp(b/g_\alpha[C_\alpha]; B_\alpha)$
does not fork over $A$; 

\item If $\alpha<\lambda^+$, $(g_\alpha[C_\alpha], B_\alpha, b) \notin \AECK^{3,\uq}$,
and the $\alpha$-level predicted amalgamation exists,
%
then $(\id_{B_\alpha}, g_{\alpha+1}, B_{\alpha+1}) \centernot{E} (\Omega_\alpha, \id_{C_{\alpha+1}}, N^\bullet)$
\todo{Consider writing the contrapositive?}
(where $(\Omega_\alpha, \id_{C_{\alpha+1}}, N^\bullet)$ is the 
$\alpha$-level predicted amalgamation of 
$B_\alpha$ and $C_{\alpha+1}$ over $(C_\alpha, g_\alpha, \id_{C_\alpha})$, 
as in Definition~\ref{def-predicted-amalg};
and $\centernot{E}$ is the negation of the relation $E$ between amalgamations of 
$B_\alpha$ and $C_{\alpha+1}$ over $(C_\alpha, g_\alpha, \id_{C_\alpha})$, as in Definition~\ref{def-E}).


\end{enumerate}

\begin{notation}
Given an ordinal $\alpha\leq\lambda^+$ and $(C_\alpha, B_\alpha, g_\alpha)$ satisfying
Clauses (1), (2), and~(\ref{b-in-difference}) above,
we write $q_\alpha := \tp(b/g_\alpha[C_\alpha]; B_\alpha)$.
\end{notation}


We now carry out the recursive construction:
\begin{itemize}
\item For $\alpha=0$:
Set $C_0 := A$, $B_0 := B$, and $g_0 := \id_A$.
As $(A, B, b) \in \AECK^{3,\bs}$,
the existence property 
(which is a consequence of the extension property; see Definition~\ref{frame-properties} and footnote~\ref{extension-footnote} there)
guarantees that $q_0 = \tp(b/A; B)$ does not fork over $A$,
and all of the requirements are satisfied.

\item For a nonzero limit ordinal $\alpha\leq\lambda^+$:
Define
\[
C_\alpha := \bigcup_{\gamma<\alpha} C_\gamma; \quad
B_\alpha := \bigcup_{\gamma<\alpha} B_\gamma; \quad
g_\alpha := \bigcup_{\gamma<\alpha} g_\gamma.
\]
Using the continuity property, Lemma~\ref{continuous-union} guarantees that all of the requirements are satisfied,
recalling that $g_0[C_0] = \id_A[A] = A$.

\item For $\beta = \alpha+1$, we consider two cases:
\todo[color=green]{Updated on Sept.~6}
\begin{itemize}
\item
Suppose $(g_\alpha[C_\alpha], B_\alpha, b) \in \AECK^{3,\uq}$.
Since $C_\alpha \lessK N^\circ$, 
begin by using $\LST(\AECK) \leq\lambda$ and Remark~\ref{LST-remark} to fix a model $C_{\alpha+1} \in \AECK_\lambda$
that includes $C_\alpha \cup (N^\circ \cap (\alpha+1))$ and such that $C_\alpha \leqK C_{\alpha+1} \lessK N^\circ$.

\item
Otherwise:
By the induction hypothesis, $q_\alpha \in S^\bs(g_\alpha[C_\alpha])$,
so that $(g_\alpha[C_\alpha], B_\alpha, b) \in \AECK^{3,\bs} \setminus \AECK^{3,\uq}$.
Viewing $g_\alpha : C_\alpha \cong g_\alpha[C_\alpha]$ as an isomorphism,
since $N^\circ \in \AECK^\sat_{\lambda^+}$,
we apply Lemma~\ref{uniqueness-saturated} to obtain
$M \in \AECK_\lambda$ witnessing the non-uniqueness
of $(g_\alpha[C_\alpha], B_\alpha, b)$ via~$g_\alpha$, with $C_\alpha \lessK M \lessK N^\circ$.
Then, by $\LST(\AECK) \leq\lambda$ and Remark~\ref{LST-remark},
we can fix a model $C_{\alpha+1} \in \AECK_\lambda$ that includes
$M \cup (N^\circ \cap (\alpha+1))$ and such that $M \leqK C_{\alpha+1} \lessK N^\circ$.
Since $\mathfrak s$ satisfies transitivity and extension,
by Lemma~\ref{non-uniqueness-extension},
$C_{\alpha+1}$ also witnesses the the non-uniqueness
of $(g_\alpha[C_\alpha], B_\alpha, b)$ via~$g_\alpha$.
\end{itemize}

In both cases, 
since $\mathfrak s$ satisfies the extension property,
we continue by applying 
Lemma~\ref{get-desired-amalg}
with 
\[
\left( A, B, A^*, C^*, N^\bullet, \varphi \right) := 
\left( g_\alpha[C_\alpha], B_\alpha, C_\alpha, C_{\alpha+1}, N^\bullet, g_\alpha \right).
\]
If the $\alpha$-level predicted amalgamation exists, then let $\Omega := \Omega_\alpha$;
otherwise $\Omega$ can be chosen arbitrarily (since $C_\alpha \lessK N^\circ \lessK N^\bullet$ and $N^\circ \in \AECK^\sat_{\lambda^+}$)
but is actually irrelevant.
We thus obtain
a $\lambda$-amalgamation
$(\id_{B_\alpha}, g_{\alpha+1}, B_{\alpha+1})$ of $B_\alpha$ and $C_{\alpha+1}$ over $(C_\alpha, g_\alpha, \id_{C_\alpha})$ such that
$q_{\alpha+1} = \tp(b/g_{\alpha+1}[C_{\alpha+1}]; B_{\alpha+1})$ does not fork over $g_\alpha[C_\alpha]$,
and 
such that Clause~(9) of the recursive construction is satisfied.
In particular, 
$B_{\alpha+1} \in \AECK_\lambda$, $B_\alpha \leqK B_{\alpha+1}$, 
$g_{\alpha+1} : C_{\alpha+1} \embeds B_{\alpha+1}$ is an embedding such that $g_\alpha \subseteq g_{\alpha+1}$,
$b \in B_{\alpha+1} \setminus g_{\alpha+1}[C_{\alpha+1}]$,
and $q_{\alpha+1} \in S^\bs(g_{\alpha+1}[C_{\alpha+1}])$.
By the induction hypothesis, we know that $q_\alpha$ does not fork over $A$,
so that by Lemma~\ref{restriction} and the transitivity property
it follows that $q_{\alpha+1}$ does not fork over $A$.
Of course,
we can ensure that the universe of $B_{\alpha+1}$ is a subset of $\lambda^+$.

\end{itemize}

This completes the recursive construction.

Step 3:

It is clear from Clauses (5) and~(6) of the recursive construction that $C_{\lambda^+} = N^\circ$,
so that $\langle C_\alpha \mid \alpha<\lambda^+ \rangle$ is a representation of $N^\circ$.

For every $\alpha\leq\lambda^+$, let $A_\alpha := g_\alpha[C_\alpha]$.
Let $N := A_{\lambda^+}$.
Then $N = g_{\lambda^+}[N^\circ] \in \AECK^\sat_{\lambda^+}$ and
$\left< A_\alpha \mid \alpha<\lambda^+ \right>$ is a representation of $N$.

Let $N^+ := B_{\lambda^+}$.
Since $N^+ = \bigcup_{\alpha<\lambda^+} B_\alpha$, 
where $\left|B_\alpha\right| = \lambda$ for every $\alpha<\lambda^+$,
it follows that $\left| N^+ \right| \leq \lambda^+$.
But $N \lessK N^+$ (since $g_{\lambda^+} : N^\circ \embeds N^+$ is an embedding), where $\left|N\right| = \lambda^+$,
so that in fact $N^+ \in \AECK_{\lambda^+}$
and $\left< B_\alpha \mid \alpha<\lambda^+ \right>$ is a representation of $N^+$.

Since $N^\circ \lessKuniv N^\bullet$, 
apply Lemma~\ref{univ-embedding} to obtain
an embedding $h : N^+ \embeds N^\bullet$ such that $h \circ g_{\lambda^+} = \id_{N^\circ}$.
Since $N = g_{\lambda^+}[N^\circ]$, it follows that $h \restriction N : N \cong N^\circ$ is an isomorphism.

For every $\alpha\leq\lambda^+$, let 
$h_\alpha := h \restriction B_\alpha$,
so that $h = h_{\lambda^+} = \bigcup_{\alpha<\lambda^+} h_\alpha$.

\begin{claim}\label{claim-after-recursion}
For every $\alpha<\lambda^+$:
\begin{enumerate}
\item $g_\alpha : C_\alpha \cong A_\alpha$ is an isomorphism and $A_\alpha \leqK B_\alpha$.
\item $A_\alpha \in \AECK_\lambda$.
\item $h_\alpha \circ g_\alpha = \id_{C_\alpha}$.
\item
$(\id_{B_\alpha}, g_{\alpha+1}, B_{\alpha+1})$ and $(h_\alpha, \id_{C_{\alpha+1}}, N^\bullet)$
are two amalgamations of $B_\alpha$ and $C_{\alpha+1}$ over $(C_\alpha, g_\alpha, \id_{C_\alpha})$.
\item
$(\id_{B_\alpha}, g_{\alpha+1}, B_{\alpha+1}) \mathrel{E} (h_\alpha, \id_{C_{\alpha+1}}, N^\bullet)$.
\item If $\Omega_\alpha = h_\alpha$ then $(A_\alpha, B_\alpha, b) \in \AECK^{3,\uq}$.
\end{enumerate}
\end{claim}

\begin{proof}\hfill
\begin{enumerate}
\item $g_\alpha$ is an embedding from $C_\alpha$ into $B_\alpha$, and $A_\alpha$ is its image.
\item By the closure of $\AECK_\lambda$ 
under the isomorphism $g_\alpha$.
\item Since $g_\alpha \subseteq g_{\lambda^+}$, $h_\alpha \subseteq h$, and $h \circ g_{\lambda^+} = \id_{C_{\lambda^+}}$.
\item
The fact that 
$(\id_{B_\alpha}, g_{\alpha+1}, B_{\alpha+1})$ 
is an amalgamation of $B_\alpha$ and $C_{\alpha+1}$ over $(C_\alpha, g_\alpha, \id_{C_\alpha})$ 
follows from $g_\alpha \subseteq g_{\alpha+1}$.
The fact that 
$(h_\alpha, \id_{C_{\alpha+1}}, N^\bullet)$
is an amalgamation of $B_\alpha$ and $C_{\alpha+1}$ over $(C_\alpha, g_\alpha, \id_{C_\alpha})$ 
follows from $h_\alpha \circ g_\alpha = \id_{C_\alpha}$.

\item $h_{\alpha+1} : B_{\alpha+1} \embeds N^\bullet$ is an embedding satisfying
$h_{\alpha+1} \circ \id_{B_\alpha} = h_\alpha$ and
$h_{\alpha+1} \circ g_{\alpha+1} = \id_{C_{\alpha+1}}$,
so that $(h_{\alpha+1}, \id_{N^\bullet}, N^\bullet)$ witnesses that
$(\id_{B_\alpha}, g_{\alpha+1}, B_{\alpha+1}) \mathrel{E} (h_\alpha, \id_{C_{\alpha+1}}, N^\bullet)$,
as sought.

\item
Suppose $\Omega_\alpha = h_\alpha$.
In particular, $\Omega_\alpha : B_\alpha \embeds N^\bullet$ is an embedding satisfying
$\Omega_\alpha \circ g_\alpha = h_\alpha \circ g_\alpha = \id_{C_\alpha}$,
meaning that the $\alpha$-level predicted amalgamation exists, and
$(\Omega_\alpha, \id_{C_{\alpha+1}}, N^\bullet)$
is the $\alpha$-level predicted amalgamation of 
$B_\alpha$ and $C_{\alpha+1}$ over $(C_\alpha, g_\alpha, \id_{C_\alpha})$.
Furthermore, by Clause~(5) above we obtain
$(\id_{B_\alpha}, g_{\alpha+1}, B_{\alpha+1}) \mathrel{E} (\Omega_\alpha, \id_{C_{\alpha+1}}, N^\bullet)$.
Thus, by Clause~(9) of the recursion, it must be that $(A_\alpha, B_\alpha, b) \in \AECK^{3,\uq}$.
\qedhere
\end{enumerate}
\end{proof}

For every $\alpha<\lambda^+$, we see from Clause~(8) of the recursive construction that
\[
\tp(b/N; N^+) \restriction A_\alpha = \tp(b/A_\alpha; N^+) = \tp(b/A_\alpha; B_\alpha) \in S^\bs(A_\alpha)
\]
and this type does not fork over $A = A_0$,
so that by Lemma~\ref{big-dnf-equiv} 
we obtain that
$\tp(b/N; N^+)  
$
does not fork over $A$.

As $g_0 = \id_A$ and $h \circ g_0 = h \circ g_{\lambda^+} \restriction A = \id_A$,
it follows that $h \restriction A = \id_A$.

It remains to show that the set
\[
\Set{ \alpha<\lambda^+ | (A_\alpha, B_\alpha, b) \in \AECK^{3,\uq} }
\]
is stationary in $\lambda^+$.
In light of Claim~\ref{claim-after-recursion}(6), the following Claim suffices,
proven exactly like Claim~\ref{get-stationary-beta} in the proof of Theorem~\ref{main-thm}:

\begin{claim}
There are stationarily many $\beta<\lambda^+$ such that $\Omega_\beta = h_\beta$.
\end{claim}

This completes the proof of the Theorem.
\end{proof}

\section{Starting from the top}
\label{section:from-top}

\begin{corollary}\label{start-from-top}
\todo{Theorem or Corollary?}
Suppose that:
\begin{enumerate}
\item $\lambda$ is an infinite cardinal such that $\diamondsuit(\lambda^+)$ holds;
\item $\mathfrak s = (\AECK, {\dnf}, S^\bs)$ is a pre-$\lambda$-frame satisfying amalgamation, transitivity, 
extension, and continuity;
\todo{
CAN WE REMOVE EXTENSION?}
\item $N^\circ, N^\bullet \in \AECK_{\lambda^+}$ are two models such that $N^\circ \in \AECK^\sat_{\lambda^+}$ and
$N^\circ \lessKuniv N^\bullet$;
\todo{
Notice we don't need $N^\bullet \in \AECK^\sat_{\lambda^+}$.
Also check whether this hypothesis is strictly weaker than assuming amalgamation and stability on $\lambda^+$.}
\item $p \in S^\bs_{>\lambda}(N^\circ)$.

\end{enumerate}
Then there 
exist a model $N' \in \AECK_{\lambda^+}$ such that $N^\circ \lessK N' \leqK N^\bullet$,
an element $d \in N' \setminus N^\circ$, and 
representations $\langle C_\alpha \mid \alpha<\lambda^+ \rangle$ of $N^\circ$
and $\langle D_\alpha \mid \alpha<\lambda^+ \rangle$ of 
$N'$ 
such that:
\begin{enumerate}
\item $p \restriction C_0 = \tp(d/C_0; N')$;
\item both $p$ and $\tp(d/N^\circ; N')$ do not fork over $C_0$;
\item $(C_\alpha, D_\alpha, d) \in \AECK^{3,\bs}$
for every $\alpha<\lambda^+$;
and
\item the set
\[
\Set{ \alpha<\lambda^+ | (C_\alpha, D_\alpha, d) \in \AECK^{3,\uq} }
\]
is stationary in $\lambda^+$.
\todo{Can we express this conclusion in terms of weakly primeness triples?
We need to specify the representation in order to say something about $p \restriction C_0$.}
\end{enumerate}



\end{corollary}

\begin{proof}
By Lemma~\ref{basic-triple-from-big}, fix 
a triple $(A, B, b) \in \AECK^{3,\bs}$ such that 
$A \lessK N^\circ$, 
$p$ does not fork over $A$,
and $\tp(b/A; B) = p \restriction A$.

We then apply Theorem~\ref{repres-uq-extend-triple} to obtain
models $N, N^+ \in \AECK_{\lambda^+}$, 
representations $\langle A_\alpha \mid \alpha<\lambda^+ \rangle$ of $N$
and $\langle B_\alpha \mid \alpha<\lambda^+ \rangle$ of $N^+$,
and an embedding $h : N^+ \embeds N^\bullet$, 
altogether satisfying the conclusion of that Theorem.

Let $N' := h[N^+]$, let $d := h(b)$, 
and for every $\alpha<\lambda^+$, let $C_\alpha := h[A_\alpha]$ and $D_\alpha := h[B_\alpha]$.

Notice that $C_0 = h[A] = A$ and $A \lessK B = B_0 \lessK N^+$,
so that $\tp(d/C_0; N') = \tp(h(b)/A; h[N^+]) = \tp(b/A; N^+) = \tp(b/A; B) = p \restriction A = p \restriction C_0$.
All of the required properties follow from their invariance under the isomorphism $h : N^+ \cong N'$.
\end{proof}

\begin{lemma}\label{apply-uniqueness}
Suppose $\mathfrak s = (\AECK, {\dnf}, S^\bs)$ is a pre-$\lambda$-frame satisfying the uniqueness property,
$A \in \AECK_\lambda$ and $N \in \AECK_{>\lambda}$ 
are models that satisfy $A \lessK N$,
and $p, q \in S(N)$ are two types that both do not fork over $A$,
and such that $p \restriction A = q \restriction A$.
Then $p \restriction M = q \restriction M$ for every $M \in \AECK_\lambda$ with $M \lessK N$.
\end{lemma}

\begin{proof}
Consider any $M \in \AECK_\lambda$ with $M \lessK N$.
Apply $\LST(\AECK) \leq\lambda$ and Remark~\ref{LST-remark} to fix
$B \in \AECK_\lambda$ containing every point of $A \cup M$, and such that $A \leqK B \lessK N$ and $M \leqK B$.
By Definition~\ref{big-dnf} of non-forking for types over larger models,
the fact that $p$ and $q$ do not fork over $A$ implies, in particular,
that both $p \restriction B$ and $q \restriction B$ do not fork over $A$.
But $p \restriction B, q \restriction B \in S(B)$ where $B \in \AECK_\lambda$,
and $(p \restriction B) \restriction A = p \restriction A = q \restriction A = (q \restriction B) \restriction A$,
so that by the uniqueness property of $\mathfrak s$ we obtain $p \restriction B = q \restriction B$.
Then $p \restriction M = (p \restriction B) \restriction M = (q \restriction B) \restriction M = q \restriction M$,
as sought.
\end{proof}

\begin{corollary}\label{from-top-with-uniqueness}
Suppose that
$\mathfrak s = (\AECK, {\dnf}, S^\bs)$ is a pre-$\lambda$-frame satisfying the uniqueness property.
If $N^\circ, N' \in \AECK_{\lambda^+}$, $p$, and $d$ satisfy the conclusion of Corollary~\ref{start-from-top},
then in fact they satisfy the following property, stronger than Clause~(1) there: 
\begin{enumerate}
\item[($1^*$)] $p \restriction M = \tp(d/M; N')$ for every $M \in \AECK_\lambda$ with $M \lessK N^\circ$.
\end{enumerate}
\end{corollary}

\begin{proof}
Apply Lemma~\ref{apply-uniqueness} with $A := C_0$, $N := N^\circ$, and $q := \tp(d/N^\circ; N')$.
\end{proof}

\begin{corollary}
Suppose that, in addition to all of the hypotheses of Corollaries \ref{start-from-top} and~\ref{from-top-with-uniqueness},
$\AECK$ is $(\lambda, \lambda^+)$-weakly tame
\cite[Definition~11.6(1)]{Baldwin-Categoricity}.

Then any triple $(N^\circ, N', d)$ satisfying the conclusion of Corollary~\ref{start-from-top} satisfies the following property,
even stronger than Clause~($1^*$) above:
\begin{enumerate}
\item[($1^{**}$)] $p = \tp(d/N^\circ; N')$.
\end{enumerate}
\end{corollary}

\begin{proof}
Property~($1^{**}$) follows from ($1^*$) and $(\lambda, \lambda^+)$-weak tameness,
recalling that we have assumed $N^\circ \in \AECK^\sat_{\lambda^+}$.
\end{proof}


\section{Acknowledgements}

The main results of this paper were presented by the first author at the
\emph{Set Theory, Model Theory and Applications} workshop (in memory of Mati Rubin),
Eilat, April 2018,
and by the second author at the
\emph{Logic Colloquium 2018},
Udine, July 2018.
We thank the organizers of the respective meetings for the invitations.

We also thank John~T. Baldwin, Rami Grossberg, and Sebastien Vasey for some useful communication.
\todo[color=red]{Is this correct?}


\bibliographystyle{alpha}


\makeatletter
\providecommand\@dotsep{5}
\makeatother

\todo{Which e-mail addresses to include?}
\end{document}